\documentclass[11pt,reqno,UTF8,twoside]{amsart}
\usepackage{mathrsfs}
\usepackage{amsfonts,amssymb,amsmath,amsthm}
\usepackage{cite}
\usepackage[colorlinks=true,citecolor=red,linkcolor=blue]{hyperref}
\usepackage{titletoc}
\usepackage{geometry}
\usepackage{bm}
\usepackage{indentfirst}
\usepackage{graphicx}
\usepackage{stmaryrd}
\usepackage{tikz}
\usepackage{ulem}
\usepackage{color,soul}
\usepackage{setspace}
\usepackage[pagewise]{lineno} 
\usepackage{float}
\usepackage{caption}

\usepackage{todonotes}

\usepackage{tabu}

\usepackage{pifont}
\usepackage{subfigure}
\usetikzlibrary{shapes, arrows.meta, positioning}
\usepackage[all]{xy}
\usepackage{bbold}
\usepackage{amscd}
\usepackage{lipsum}
\usepackage{enumitem}

\providecommand{\U}[1]{\protect\rule{.1in}{.1in}}

\def\theenumi{\arabic{enumi}}

\def\theenumii{\alph{enumii}}
\def\p@enumii{\theenumi.}

\def\theenumiii{\arabic{enumiii}}
\def\p@enumiii{(\theenumi)(\theenumii)}

\def\p@enumiv{\p@enumiii.\theenumiii}
\usepackage{todonotes}

\newcommand{\ii }{{\rm i} }
\newcommand{\N}{{\mathbb N}}
\newcommand{\A}{{\mathcal A}}

\newcommand{\Z}{{\mathbb Z}}
\newcommand{\R}{{\mathbb R}}

\newcommand{\C}{{\mathbb C}}

\newcommand{\G}{{\varphi}}

\newcommand{\bigO}{{\mathcal O}}
\newcommand{\diff}{\mathop{}\!\mathrm{d}}

\def\p{\partial}
\def\supp#1{{\textrm{supp}(#1)}}
\def\Im{{\rm Im}}
\def\Re{{\rm Re}}
\newtheorem{Thm}{Theorem}[section]
\newtheorem{theorem}[Thm]{Theorem}
\newtheorem{corollary}[Thm]{Corollary}
\newtheorem{remark}[Thm]{Remark}
\newtheorem{lemma}[Thm]{Lemma}
\newtheorem{proposition}[Thm]{Proposition}
\newtheorem{definition}[Thm]{Definition}

\newtheorem*{question}{Open problem}
\newtheorem{aquestion}[Thm]{Problem}
\newtheorem*{maintheorem}{Main Theorem}

\theoremstyle{definition}
\newtheorem*{example}{Example}

\numberwithin{equation}{section}
\newcommand{\cqfd}
{%
\mbox{}%
\nolinebreak%
\hfill%
\rule{2mm}{2mm}%
\medbreak%
\par%
}

\title{Small-time local controllability of a KdV system for all critical lengths}
\author{Jingrui Niu, Shengquan Xiang}

\begin{document}
\address[Jingrui Niu]{Sorbonne Université, CNRS, Université Paris Cité, Inria Team CAGE, Laboratoire Jacques-Louis Lions (LJLL), F-75005 Paris, France
}
\email{jingrui.niu@sorbonne-universite.fr}

\address[Shengquan  Xiang]{School of Mathematical Sciences, Peking University, 100871, Beijing, China.}
\email{shengquan.xiang@math.pku.edu.cn}
   
\subjclass[2020]{35Q53, 76B15, 93C20}
	
\keywords{KdV, small-time local controllability, obstruction, nonlinearity, power series expansion}

    \vspace{5mm}
	\begin{abstract}
     In this paper, we consider the small-time local controllability problem for the KdV system on an interval with a Neumann boundary control.   Rosier discovered in \cite{Rosier97} that the linearized system is uncontrollable if and only if the length is critical, namely $L=2\pi\sqrt{(k^2+ kl+ l^2)/3}$ for some integers $k$ and $l$.  

     Coron and Cr\'epeau \cite{CC04} proved that the nonlinear system is small-time locally controllable even if the linearized system is not, provided that $k= l$ is the only solution pair. Later,  Cerpa \cite{Cerpa07}, Cerpa and Crepeau \cite{CC09} showed that the system is large-time locally controllable for all critical lengths.   Coron, Koenig and Nguyen \cite{CKN} found that the system is not small-time locally controllable if $2k+l\not \in 3\mathbb{N}^*$. 
     
    We demonstrate that if the critical length satisfies $2k+l \in 3\mathbb{N}^*$ with $k\neq l$, then the system is not small-time locally controllable. This paper, together with the above results, gives a complete answer to the longstanding open problem on the small-time local controllability of KdV on all critical lengths since the pioneer work by Rosier \cite{Rosier97}.
    \end{abstract}

\maketitle

\setcounter{tocdepth}{1}
\tableofcontents

\section{Introduction}\label{sec: introduction}
In this article, we are interested in the {\it small-time local controllability} property of the classical controlled Korteweg–de Vries (KdV) model on an interval: 
\begin{equation}\label{intro-sys-KdV}
\left\{
\begin{array}{ll}
\p_t y + \p_x^3y  + \p_x y + y  \p_x y = 0, &  \text{ in }  (0, T)\times(0, L), \\
y(t, 0) = y(t, L) = 0, & \text{ in } (0, T), \\
\p_x y(t , L) = u(t), & \text{ in } (0, T),\\
y(0, \cdot)  = y_0 (\cdot), &\text{ in } (0, L),
\end{array}\right.
\end{equation}
where $y$ is the state and $u$ is the control.

The KdV equation is an important nonlinear model that has been extensively studied in the literature from various perspectives. The well-posedness problem has been extensively studied in the literature, including the works \cite{Bona-Smith,Bourgain93-2, Bourgain97,KPV93,KPV96  ,CKSTT,KV-19}, and among others. The stability of solitary wave of KdV is investigated in the works  \cite{MV,MM-01,KV-22} along with the references therein. 
We refer the reader to the books \cite{KTV-book,Tao06} for an overview of recent mathematical advancements related to this equation.

\subsection{An open problem}\label{sec: history}
Equation \eqref{intro-sys-KdV} is among the most investigated models emphasizing nonlinear controls. Its control properties are strongly related to the interaction of the nonlinear term. 
Let us briefly review the
existing results. We also refer to the surveys \cite{RZ09, Cerpa14} and the references therein.

In the pioneer paper \cite{Rosier97}, Rosier introduced the so-called {\it critical lengths set}:
\begin{equation}\label{eq: defi of critical set-intro}
     \mathcal{N}:=\{2\pi\sqrt{\frac{k^2+kl+l^2}{3}}:k,l\in\N^*  \}.
\end{equation}
Using spectral theory and complex analysis techniques, he  proved that the
linearized system is controllable if and only if the length of the interval $L \not \in \mathcal{N}$,
\begin{equation}\label{eq: intro-linear-KdV}
\left\{
\begin{array}{ll}
\p_t y + \p_x^3y  + \p_x y = 0, &  \text{ in }  (0, T)\times(0, L), \\
y(t, 0) = y(t, L) = 0, & \text{ in } (0, T), \\
\p_x y(t , L) = u(t), & \text{ in } (0, T).
\end{array}\right.
\end{equation}

When $L \in \mathcal{N}$ and $u=0$, the linear  system \eqref{eq: intro-linear-KdV} admits a family of time-periodic ``traveling wave" solutions  $e^{\ii t\lambda}\G_{\lambda}(x)$, where $(\G_{\lambda},\ii \lambda)$  is an eigenfunction
\begin{equation}\label{eq: type-1-eigenfunction-intro}
  \left\{
 \begin{array}{c}
      \G'''_{\lambda}+\G'_{\lambda}+\ii\lambda \G_{\lambda}=0,  \\
     \G_{\lambda}(0)=\G_{\lambda}(L)=\G'_{\lambda}(0)=\G'_{\lambda}(L)=0. 
 \end{array}
 \right.    
 \end{equation}
 Notice that such an eigenfunction satisfies four boundary conditions. Indeed, each integer solution $(k, l)$ of 
\begin{equation}\label{eq:L=kl}
    L= 2\pi\sqrt{\frac{k^2+kl+l^2}{3}},
\end{equation}
corresponds to such an eigenfunction.  
 Thus
there are only finitely many linearly independent eigenfunctions of this form, and all these functions form a subspace of $L^2(0, L)$,
\begin{equation}\label{eq: defi-M}
M:=\mathrm{Span}\{\Re \G_{\lambda}, \Im \G_{\lambda}: \G_{\lambda} \mbox{ is defined in \eqref{eq: type-1-eigenfunction-intro}}\}   
\end{equation}
$M$ is also called the unreachable subspace for the linear system  \eqref{eq: intro-linear-KdV}. Since the solution of  \eqref{eq: intro-linear-KdV} with initial state 0 and any control can never enter $M$. Meanwhile, there is a reachable subspace $H$. The state space can be decomposed as  $L^2(0,L)=M\oplus H$.

\vspace{2mm}
When $L\not \in \mathcal{N}$, the system \eqref{sec: introduction} is {\it small-time locally controllable} thanks to the good property of the linearized system. However, when $L \in \mathcal{N}$, since the linear system is uncontrollable, it was widely open whether the nonlinear system \eqref{sec: introduction} is small-time locally controllable or not:

\begin{question} \label{OQ} 
Is the system  \eqref{intro-sys-KdV}  small-time locally controllable for all critical lengths? 
\end{question}
This open problem has since garnered great attention. Over the past two decades, numerous partial results have been established, yet this problem continues to be revisited.  Notably, it has been highlighted many times in subsequent works, including  \cite[Remark 4]{CC04}, \cite[Open problem 8.8]{Coron07}, 
\cite[Open problem 10]{Coron07-Survey}, \cite[Remark 1.7]{Cerpa07}, \cite[Section 7.2]{M18},   \cite[Open problem 1.6]{CKN}. 

\vspace{3mm}
A significant step forward, yielding a somewhat unexpected result on this open problem, was accomplished by Coron and Crépeau \cite{CC04}. They considered the critical lengths such that $\dim M= 1$  and showed that even though the linearized system is uncontrollable the nonlinear system is small-time locally controllable.  
They introduced the so-called {\it power series expansion} method. Since then this method has become a powerful tool in the study of nonlinear control problems and in the understanding of the above-mentioned open problem. 
 Its idea is to search a control $u$ in the form:
\[
u=\varepsilon u_1+\varepsilon^2u_2+\varepsilon^3u_3+\cdots.
\]
The corresponding solution $y$ to \eqref{intro-sys-KdV} can formally decompose as
\[
y=\varepsilon y_1+\varepsilon^2 y_2+ \varepsilon^3y_3+\cdots.
\]
And the nonlinear term $y\p_xy$ can be written into
\[
y\p_xy=\varepsilon^2y_1\p_xy_1+\varepsilon^3(y_1\p_xy_2+y_2\p_xy_1)+\cdots.
\]
The first order term $y_1$ is a solution of the linearized system \eqref{eq: intro-linear-KdV} with $u=u_1$. While the second order term $y_2$ satisfies the linear system with a source term:
\begin{equation}\label{eq: y-2-KdV}
\left\{
\begin{array}{ll}
\p_t y_2 + \p_x^3y_2 + \p_x y_2=  -y_1\p_xy_1, &  \text{ in }  (0, T)\times(0, L), \\
y_2(t, 0) = y_2(t, L) = 0, & \text{ in } (0, T), \\
\p_x y_2(t , L) = u_2(t), & \text{ in } (0, T).
\end{array}\right.
\end{equation}
And continue to find equations for $y_3$ and even higher order terms.  Let the initial states $y_1|_{t=0}=y_2|_{t=0}=0$. The key ingredient is to find $u_1$ and $u_2$ such that the final states satisfy $y_1|_{t=T}=0$ and the  projection of $y_2|_{t=T}$ on $M$ is a given (nonzero) element in $M$. A key quantity associated with this projection is the quantity $Q_M$. For any element $\varphi\in M$ we define
\begin{equation}\label{eq: defi-Q_M-intro}
Q_M(\varphi; y):=\int_{0}^{\infty} \int_0^L |y (t, x)|^2 e^{-\ii  pt }\varphi'(x) dx dt,
\end{equation}
where $p\in\R$.\footnote{Here, we choose $(\varphi,\ii p)$ to be an eigenmode in $M$ corresponding to the stationary KdV operator.} 

\vspace{3mm}
To tackle this open problem on more complicated situations, in \cite{Cerpa07} Cerpa considered the case when $\dim M=2$. Using the power series expansion method he proved a {\it large-time local controllability} result. Later on, Cerpa and Cr\'epeau proved large-time local controllability for all critical lengths in \cite{CC09}.
In parallel, in the study of stabilization or asymptotic stability problems of KdV systems, the dimension of unreachable subspace $M$ also plays an important role.  We make a detailed description of these properties and how the dimension influence the results in Section \ref{sec: revisit of literature}. Finally, we summarize some existing results in Table \ref{tab: Review of existing results}.

In all these results the critical lengths are naturally classified by ``two classes":  1)  when the dimension of $M$ is even thus every pair $(k, l)$ satisfying the algebraic equation \eqref{eq: defi of critical set-intro} verifies $k\neq l$. In this circumstance, it suffices to perform a power series expansion to the second order to obtain large-time local controllability,  exponential stabilization,  asymptotic stability without control, and among other properties;   2)  when the dimension of $M$ is odd thus there exists $k= l$ satisfying \eqref{eq: defi of critical set-intro}. Then one needs to rely on an expansion up to the third order to obtain the required properties. See 
Section \ref{sec: revisit of literature} for more details and \cite{coron-rivas-xiang} for a nice description of this classification.
 
\par
\begin{table}[h]
\renewcommand{\arraystretch}{1.3}
\begin{tabular}{|m{2.4cm}<{\centering}|m{2cm}<{\centering}|m{2cm}<{\centering}|m{2cm}<{\centering}|m{2cm}<{\centering}|m{2cm}<{\centering}|}
\hline
$\dim M$ &$0$ &$1$ & $2$ & odd&even \\
\hline
Small-time controllability& \cite{Rosier97} & \cite{CC04} & Partial result  \cite{CKN}&Unkonwn& Partial result\cite{CKN} \\
\hline
Large-time controllability & \cite{Rosier97}& \cite{CC04}  &  \cite{Cerpa07} &  \cite{CC09} & \cite{CC09}\\
\hline
Exponential Stabilization & \cite{Coron-Lv} & Unknown   & \cite{coron-rivas-xiang}  & Unknown &\cite{coron-rivas-xiang} \\
\hline
Asymptotic stability & \cite{PVZ} &  \cite{Chu-Coron-Shang}  \cite{Nguyen}   &  \cite{Tang-Chu-Shang-Coron}  \cite{Nguyen} & Unknown & Partial result \cite{Nguyen} \\
\hline
\end{tabular}
\vspace{1em}
\caption{Existing results based on the parity of dim $M$}
\label{tab: Review of existing results}
\end{table}
In 2020, another major step has been made by Coron--Koenig--Nguyen in \cite{CKN} concerning this challenging open problem. For the first time, they found a negative result on the small-time local controllability of this classical model:   the {\it obstruction} to small-time local controllability if the critical length admits some pair satisfying $2k+l\notin 3\mathbb{N}^*$.  At least for technical reasons, this result indicates that the case $(k, l)= (2, 1)$ may be different from the case $(k, l)= (4, 1)$. 
\vspace{3mm}

Recently, the {\it obstruction to small-time local controllability problem} has garnered considerable attention for both ODE and nonlinear PDE models.  For NLS with interior control, Coron \cite{Coron06} and Beauchard--Morancey \cite{2014-Beauchard-Morancey-MCRF} used the iterated Lie bracket (see \cite{1987-Sussmann-SICON} in finite-dimensional cases) to find a direction which one could not move in a small time. Marbach proved a beautiful result concerning the obstruction to small-time local controllability of a viscous Burgers equation in \cite{M18}, where a coercive estimate is introduced to prove the obstruction.  Later on a nonlinear parabolic equation was considered by Beauchard and Marbach \cite{BM20} with interior controls. More recently, this property has been systematically developed by Beauchard, Marbach and their coauthors in a series of works \cite{BM-JDE,BDE20,BLBM,BM24,BMP-sch}, and in particular, it has been used to resolve the challenging problem in numerical splitting methods \cite{beauchard-splitting}. 

Finally, we note that substantial progress has been made on the controllability and stabilization of KdV equations with various types of controls. For results on KdV equations with internal controls, we refer, among others, to \cite{CRPR, RZ96, Rosier-Zhang-06, LRZ10}. Regarding other boundary-controlled KdV models, relevant works include but not limited to \cite{CC13, glass-guerrero, 
KX, nguyen2023, nguyen2024stabilization, Rosier04, Xiang-18, Xiang-19}.

\subsection{Statement of the result}
In this paper, we consider the remaining case of the open problem,  where there exists $(k, l)$ such that $2k+l\in 3\mathbb{N}^*$ and $k\neq l$, and demonstrate the following negative result.

\begin{maintheorem}\label{thm-main} 
Let $L\in \mathcal{N}$. Assume there exist different integers $k, l$ satisfying  $2 k + l  \in  3\N^*$ such that \eqref{eq:L=kl} holds. 
Then system \eqref{intro-sys-KdV}  is not
small-time locally null-controllable with controls in $H^{\frac{4}{3}}$ and initial and final datum in $H^4(0, L) \cap H^1_0(0, L)$. 

More precisely, there exist $T_0>0$ and $\varepsilon_0 > 0$ such
that,  for all $\delta >0$, there is $y_0 \in H^4(0, L) \cap H^1_0(0, L)$ with $\| y_0\|_{H^4(0, L)} <
\delta$ such that for all control $u \in H^{\frac{4}{3}}(0, T_0)$ with $\| u\|_{H^{\frac{4}{3}}(0, T_0)} < \varepsilon_0$ and the compatible condition $u(0) = y_0'(L)$, we
have
\[
y(T_0, \cdot) \not \equiv 0,
\]
where $y \in C\left([0, T_0]; H^4(0, L) \right)  \cap L^2\left([0, T_0]; H^5(0, L)\right)$ is the unique solution of \eqref{intro-sys-KdV} with initial state $y_0$ and control $u$.
\end{maintheorem}

This result, together with \cite{CC04} and \cite{CKN}, and the novel classification of critical sets (see Definition \ref{def:new:classification} and Remark \ref{rem:S123}), gives a complete answer to the open problem of small-time local controllability.
    \begin{itemize}
        \item[(1)] Case $L\in \mathcal{N}^1$.  In this class the equation \eqref{eq:L=kl} only admits one solution, and this solution is given by that is $k= l$. The dimension of the linearly uncontrollable subspace $M$ is one.\\  Coron and Crepeau proved that the system is small-time locally controllable \cite{CC04}. 
        \item[(2)] Case $L\in \mathcal{N}^2$.  The equation only admits pairs from $\mathcal{S}_2$, namely $2k+l\not \in 3\N^*$. In this class the dimension of $M$ is even. But meanwhile, there are many critical lengths from $\mathcal{N}^3$ such that the dimension of $M$ is also even.\\  Coron, Koenig and Nguyen's result \cite{CKN} shows that the system is not small-time locally controllable. On the other hand, the system is large-time locally controllable due to Cerpa \cite{Cerpa07}, Cerpa and Crepeau \cite{CC09}. 
        \item[(3)] Case $L\in \mathcal{N}^3$.  In this final class, the equation \eqref{eq:L=kl}  must have solutions from $\mathcal{S}_3$ and may also include a solution from $\mathcal{S}_1$; specifically, there exists a solution pair such that  $2k+l \in 3\mathbb{N}^*$. And the dimension of $M$ can be any integer that is strictly greater than one. Consequently, the dynamical behavior of both the linear and nonlinear KdV systems in this case is the most intricate.   \\
        According to  \cite{Cerpa07, CC09}, the system is large-time locally controllable. Now, thanks to our Main Theorem, we know that the system is not small-time locally controllable. 
    \end{itemize}
\begin{remark}
A key ingredient is on the application of a new classification of critical lengths, derived from quantitative and asymptotic analysis for the eigenmodes of stationary KdV operators; see Section \ref{sec: critical length} for details. Thus, even within the situation when dim $M= 2$, where the two-dimensional subspace $M$ is generated by the same rotation process,   the cases from $\mathcal{N}^2$ and $\mathcal{N}^3$ display completely different behavior. 

We believe this new criterion is fundamental to the critical lengths and is
applicable in addressing various problems associated with the KdV systems \eqref{intro-sys-KdV}, including issues of asymptotic stability, controllability, and stabilization. 
 Further discussion on these perspectives can be found in Section \ref{sec: further perspectives}.  
  Moreover, our Main Theorem  serves as a concrete illustration of its successful application.
\end{remark}

One shall mention that the spaces used in the main theorem and in \cite{CKN} are higher than $L^2(0, L)$.  
It would be interesting to extend these results to address the small-time local controllability problem in the less regular spaces, or even in the $L^2$ space. This may require an even more refined understanding of how regularity influences controllability.

\subsection{Strategy of the proof}\label{sec: strategy of the proof}
Our proof primarily relies on the newly established classification of critical lengths and the related spectral analysis. With this insight, we refine the power series expansion method and establish a trapping direction for the KdV system \eqref{intro-sys-KdV} to deduce the obstruction result. 
\subsubsection{A novel classification of critical lengths: two types of eigenfunctions from the spectrum point of view.}\label{subsubsec:intro:nocelclass}
As we have seen in the aforementioned papers, the study of KdV systems at critical lengths, including controllability, stability, stabilization, etc., essentially relies on the analysis on eigenmodes of the stationary KdV operator at critical lengths. 

Assume that for some pair $(k, l)$ there is $L=2\pi\sqrt{(k^2+kl+l^2)/3}\in\mathcal{N}$.  We consider the eigenfunctions of the operator given by \eqref{def:ope:A}, thus 
\begin{equation}\label{eq: intro-eigenmode-eq}
\left\{
\begin{array}{ll}
    \G'''+\G'+\ii\lambda\G=0, &\mbox{ in }(0,L)  \\
     \G(0)=\G(L)=0,& \\
     \G'(0)= \G'(L).
\end{array}
\right.
\end{equation}
By the choice of $L$ there is always a solution $\G$ satisfying $\G'(0)=\G'(L)=0$. We call such an eigenfunction ``Type 1". 
Type 1  eigenfunctions are related to the controllability of the linear system.

We observe that if and only if $2k+l\in 3\N^*$, there exists another linear independent eigenfunction $\Tilde{\G}$ satisfying \eqref{eq: intro-eigenmode-eq} with $\Tilde{\G}'(0)=\Tilde{\G}'(L)\neq0$, which is precisely given by $\tilde \G= \G'$. This eigenfunction is called of ``Type 2".

Rosier has extensively studied Type 1 eigenmodes and demonstrated that the unreachable subspace $M$ is generated by them.   Notice that even though Type 2 eigenfunctions do not belong to $M$, the control property of the nonlinear system is indirectly influenced by them; see, for instance, the equation \eqref{eq: defi-Q_M-intro}. We shall observe this influence in the present article when considering the obstruction to small-time controllability of \eqref{intro-sys-KdV}. Moreover, in a forthcoming paper, we further demonstrate that Type 2 eigenfunctions can affect other properties such as the stability and controllability of KdV equations.

\vspace{3mm}
Based on this key observation, we discover that when investigating control related problems,  the critical lengths should be distinguished by three classes according to the property of the pairs $(k,l)$ satisfying \eqref{eq: defi of critical set-intro}: $k=l$, $2k+l\notin3\N^*$, and $2k+l\in3\N^*$ and $k\neq l$.   Different types of pairs $(k, l)$ may provide distinct behaviors concerning eigenvalues and eigenfunctions.   This is to be compared with the classification widely used in the previous works depending on the parity of the dimension of $M$; see for instance \cite{coron-rivas-xiang} for a description of such a classification.  See Section \ref{sec: revisit of literature}--\ref{sec: new classification} for details and effects on these old and new classifications.

\begin{table}[h]
\renewcommand{\arraystretch}{1.5}
\begin{tabular}{|m{2.4cm}<{\centering}|m{3.5cm}<{\centering}|m{3.8cm}<{\centering}|m{3.8cm}<{\centering}|}
\hline
&$k=l$ & $2k+l\in 3\N^*, k\neq l $ & $2k+l\notin 3\N^*$ \\
\hline
Eigenvalues&  zero (double)&  nonzero (double) & \;  nonzero (simple)\; \\
\hline
Eigenfunctions&\; both Type 1 and 2\; &\; both Type 1 and 2\; & only Type 1\\
\hline
Small-time controllability &  Positive, \cite{CC04}  & Negative, {\bf Main Thm} & Negative, \cite{CKN}\\
\hline
\end{tabular}
\vspace{3mm}
\caption{Different types of $(k,l)$ satisfying  \eqref{eq:L=kl}.}
\label{tab: (k,l)-type-small-time}
\end{table}

\subsubsection{A trapping direction}
A major step is to construct a trapping direction $\Psi(t,x)$, and prove that the solution $y$ to \eqref{intro-sys-KdV} with initial condition $y(0,\cdot)=\varepsilon\Psi(0,\cdot)$ and $u= 0$ satisfies:
\begin{equation}\label{eq: trap-est-intro}
\|y(t,\cdot)-\varepsilon\Psi(t,\cdot)\|_{L^2(0,L)}\lesssim\varepsilon^2,\mbox{ for }t \mbox{ small enough}.    
\end{equation}
\begin{itemize}
    \item \textit{ A reduction approach.} This reduction approach is raised by Coron-Koenig-Nguyen in \cite{CKN}, which is inspired by the power series expansion method introduced by Coron-Cr\'epeau in \cite{CC04}, and the coercive type argument dues to Marbach \cite{M18}. Thus we reduce the procedure of establishing a trapping direction by relating it to a quantitative coercive estimate of the quantity $Q_M$. The obstruction to small-time controllability is linked with this trapping direction. 
    \item \textit{The remaining case under the new classification is degenerate for the reduction approach.}  Due to the appearance of Type 2 eigenmodes in the case $2k+l\in3\N^*$ and the non-triviality of eigenvalues, we cannot detect the non-vanishing leading term of $Q_M$ in the same order as in \cite{CKN}. In this sense, we say  the remaining case is technically {\it degenerate}.  We need to perform a more delicate analysis to recover the leading term at a higher order, which is compatible with our classification.
    \item \textit{Higher-order asymptotic expansion.} In this degenerate case, we point out that there are some technical difficulties due to the appearance of Type 2 eigenmodes associated with unreachable pairs $2k+l\in3\N^*$. To handle these, we introduce a higher-order expansion toolbox involving pseudodifferential operators, a quantitative embedding for compactly supported functions, etc. 
\end{itemize}
\subsubsection{Obstruction to small-time controllability.}The last step is based on a contradiction argument. Suppose that $y(t,\cdot)\equiv0$ for $t\geq T$, we derive the estimate \eqref{eq: trap-est-intro} for the well-prepared trapping direction $\Psi$. Therefore, at time $T$, $y(T,\cdot)$ is sufficiently close to $\Psi(T,\cdot)$, which implies that $y(T,\cdot)\neq0$, which completes contradiction argument.

\subsection{Organize of the article}

The rest part of this paper is organized as follows.

Section \ref{sec: critical length} is devoted to introducing our new classification criteria on critical lengths. After a review of previous literature in Section \ref{sec: revisit of literature}, we present the details of classification in Section \ref{sec: new classification}, followed by its application to the characterization of the dimension of $M$ in Section \ref{sec: dimension characteristics} and further application in Section \ref{sec: further perspectives}. 

In Section \ref{sec: preparations}, we prepare some useful elements necessary for the proof of Main Theorem, including well-posedness and control formulation of KdV equations in Section \ref{sec: well-posedness} and Section \ref{sec: control-0-0}, along with basic microlocal analysis elements in Section \ref{sec: 1d microlocal analysis}. 

Section \ref{sec: An unreachable direction} is dedicated to constructing a trapping direction (with details in Section \ref{sec: construction trap direction}) and proving the coercive property associated with this direction. More precisely, we give a brief idea of proof in Section \ref{sec: idea of proof-trapping direction}. Then we establish an asymptotic expansion for our key quantity $Q_M$ in Section \ref{sec: Asymptotic analysis of B} and prove the quantitative coercive proposition in Section \ref{sec: quantitative estimates}. 

Section \ref{sec: obstruction} present the proof of Main Theorem based on a contradiction argument. We finish this paper with three technical appendices.

\section{A novel classification on critical lengths}\label{sec: critical length}

The main contribution of this section is on the following key observation: even within the class $\mathcal{N}_2$, where the two-dimensional subspace $M$ is generated by the same rotation process,   the cases $2k+l\in 3\N^*$ and $2k+l\notin 3\N^*$ exhibit fundamentally different behavior. 
This distinction arises from various factors, including the presence of Type 2 eigenmodes of stationary KdV operators  and the limiting spectral analysis, which will be detailed in a forthcoming paper. Furthermore, this difference has important implications for both control and analysis problems.

Based on this observation we propose a new classification criterion for critical lengths. Recall that the widely used classification is based on the parity of dim $M$ or the Type 1 eigenfunctions. 
While the new classification scheme is based on the congruence relation modulo $3$, i.e., whether $2k+l$ is an element of $3\N^*$ or not, and whether $k= l$ or not.

This section is structured as follows. In Section \ref{sec: revisit of literature}, we outline the mathematical reasoning, specifically the rotation process, underlying the widely used existing classification and provide a literature review based on this criterion. In Section \ref{sec: new classification}, we introduce our proposed classification rule. Section \ref{sec: dimension characteristics} focuses on a quantitative analysis of the dimension of $M$ under the new classification. We believe this new criterion is fundamental and applicable to a broad range of problems, which we summarize in Section \ref{sec: further perspectives}.

\subsection{Revisit the literature: the study based on the parity of dim $M$}\label{sec: revisit of literature}

Following the pioneering work of \cite{Rosier97}, extensive research has been conducted, yielding diverse results based on the properties of the critical set $\mathcal{N}$.
As the dimension of the linearly unreachable subspace $M$ increases, the interaction between the linear and nonlinear components of the system becomes more intricate. Notably, the set $\mathcal{N}$ comprises countably many critical lengths.  A comprehensive understanding of the system's controllability, stabilization, and asymptotic stability necessitates classifying the critical lengths and conducting a detailed analysis of each category.

Over the past two decades, as the understanding of the nonlinear system \eqref{intro-sys-KdV} has deepened, the following classification and investigation criteria have been introduced. This rule is essentially based on the parity of dimension $M$; see for instance \cite[Introduction]{coron-rivas-xiang}.
\begin{itemize}
 \item[0.]  $\mathcal{C} := \mathbb{R}^{+}\setminus{\mathcal{N}}$. Then $M = \{ 0 \}$.
\item[1.] $\mathcal{N}_1 := \big{\{} L \in \mathcal{N};$ there exists one and only one ordered pair $(k, l)$ satisfying  \eqref{eq: defi of critical set-intro} and one has $k= l\big{\}}$. Then  the dimension of $M$ is 1.
\item[2.] $\mathcal{N}_2 := \big{\{} L \in \mathcal{N};$ there exists one and only one ordered pair  $(k,l)$ satisfying  \eqref{eq: defi of critical set-intro} and one has $k>l\big{\}}$. Then  the dimension of $M$ is 2.
\item[3.] $\mathcal{N}_3 := \big{\{} L \in \mathcal{N};$ there exist $n \geqslant 2$ different ordered pairs $(k,l)$ satisfying  \eqref{eq: defi of critical set-intro}, and none of them satisfies $k= l\big\}$. Then the dimension of $M$ is $2n$.
\item[4.] $\mathcal{N}_4 := \big{\{} L \in \mathcal{N};$ there exist $n \geqslant 2$ different ordered pairs $(k,l)$ satisfying  \eqref{eq: defi of critical set-intro},  and one of them satisfies $k= l$ $\big{\}}$.  Then  the dimension of $M$ is $2n-1$.
\end{itemize}
These sets are disjoint and  
\begin{equation}\label{eq:old:classifi}
    \mathbb{R}^{+} = \mathcal{C} \cup \mathcal{N}_1 \cup \mathcal{N}_2 \cup \mathcal{N}_3 \cup \mathcal{N}_4,\,\; 
\mathcal{N}= \mathcal{N}_1 \cup \mathcal{N}_2 \cup \mathcal{N}_3 \cup \mathcal{N}_4. 
\end{equation}

\subsubsection{The controllability problem}

First,  Rosier solved the small-time
local controllability problem of the nonlinear equation for $L\in \mathcal{C}$.
The first result on critical lengths dues to  Coron and Cr\'epeau \cite{CC04}, they proved the small-time local controllability  for 
$L\in \mathcal{N}_1$, namely  $L= 2k\pi$ and 
\begin{equation}\label{eq: 2pi-condition}
\not\exists(m,n)\text{ such that }m^2+n^2+mn=3k^2 \text{ and }m\neq n.    
\end{equation}
More precisely, they considered the typical case that $k= l= 1$ thus $L= 2\pi$, where the linearly unreachable subspace $M=\mathrm{span}\{1-\cos{x}\}$, and observed that the quantity $Q_M$ is identically $0$. Recall its definition in \eqref{eq: defi-Q_M-intro}.
\[
Q_M(1-\cos{x}; y)=\int_{0}^{T} \int_0^L y (t, x)^2 \sin{x} \diff x \diff t\equiv0,
\]
provided that $y$ is a solution to linearized KdV system satisfying $y(0,x)=y(T,x)=0$. The triviality of the quantity $Q_M$ implies that the second order approximated system on $y\sim \varepsilon y_1+ \varepsilon^2 y_2$ is still uncontrollable. Indeed, the projection of $y_2$ on $M$ does not change despite the nonlinear interaction:
 \begin{gather}
 \left\{
 \begin{array}{l}
 \left\{
 \begin{array}{l}
\p_t y_1 + \p_x^3y_1 + \p_x y_1=  0,   \text{ in }  (0, T)\times(0, L), \\
y_1(t, 0) = y_1(t, L) = 0,  \text{ in } (0, T), \\
\p_x y_1(t , L) = u_1(t),  \text{ in } (0, T).
 \end{array}
 \right.
 \\
  \left\{
  \begin{array}{l}
\p_t y_2 + \p_x^3y_2 + \p_x y_2=  -y_1\p_xy_1,   \text{ in }  (0, T)\times(0, L), \\
y_2(t, 0) = y_2(t, L) = 0,  \text{ in } (0, T), \\
\p_x y_2(t , L) = u_2(t),  \text{ in } (0, T).
 \end{array}
 \right.
 \end{array}
 \right.
 \end{gather}
Through a third-order  expansion, they arrived at both directions $\pm(1-\cos{x})$ within any short time, which ensures the small-time local controllability.

\vspace{3mm}
Later on, Cerpa considered the case $L\in\mathcal{N}_2$ in \cite{Cerpa07}. Following the idea in \cite{CC04}, he proved a large-time local controllability result of \eqref{intro-sys-KdV} for $\dim M=2$. 
He also showed that $\mathcal{N}_2$ contains infinitely many elements.  In this circumstance, the behavior of the quantity $Q_M$ is more complex. Using complex analysis techniques Cerpa showed that it is not identically $0$. Thanks to this observation, he showed that the second order approximated system can arrive at a certain direction $\varphi_0\in M$ at any short time. 

The large-time controllability is fulfilled by a rotation process. While the rotation from $\varphi_0$ to any direction $e^{\ii pt}\varphi_0$ takes a time $T\geq \frac{\pi}{p}$.
More precisely, let us consider the case that $k=2, l=1$, thus $\dim M=2$ and $M=\mathrm{Span}\{\varphi_1,\varphi_2\}$ with $\varphi_1+\ii\varphi_2$ being an eigenfunction associated with an eigenvalue $\ii p$. 
\begin{figure}[htp]
 \begin{minipage}[c]{0.3\textwidth}
  {\centering
  \begin{tikzpicture}[scale=1]
    \draw (0,0) circle (2cm);
    \filldraw[black] (2,0) circle (1pt) node[anchor=north] {$\varphi_1$};
    \filldraw[black] (0,2) circle (1pt) node[anchor=east] {$\varphi_2$};
    \draw[-latex, line width=0.5mm] (1.41,1.41) arc (45:5:2cm);
    \draw[-latex, line width=0.5mm] (-1.41,-1.41) arc (225:185:2cm);
\end{tikzpicture}}
 \end{minipage}
 \begin{minipage}[c]{0.55\textwidth}
 
\begin{equation*}
\left\{
\begin{aligned}
    \frac{\diff}{\diff t}(y(t), \varphi_1)_{L^2(0,L)} &= -p (y(t), \varphi_2)_{L^2(0,L)}, \\
    \frac{\diff}{\diff t}(y(t), \varphi_2)_{L^2(0,L)} &= \phantom{-}p (y(t), \varphi_1)_{L^2(0,L)},
\end{aligned}
\right.
\end{equation*}

\label{fig: rotation}
 \end{minipage}
\end{figure}
\noindent
Then one notice that the solution $y$ to \eqref{intro-sys-KdV} projects on $M$ verifies a rotation. This implies that $\frac{2\pi}{p}$ is a rotation period.  Since the solution can reach the direction $\varphi_0= \alpha \varphi_1+ \beta \varphi_2$ within a short time $T_0$, the rotation process indicates that the solution can further reach all states in $M$ provided that the time period is larger than $T_0+ \frac{2\pi}{p}$.

This rotation strategy can not answer the open problem on small-time controllability for $\dim M=2$. 
Because this process is so natural such that,  since then people do not further distinguish different $L$ such that dim$M=2$. For example $(k, l)= (2, 1)$ and $(k, l)= (4, 1)$ are treated in the same rotation method, and it was believed that there is no essential difference between these two specific cases. 

\vspace{3mm}
Following this rotation approach introduced in \cite{Cerpa07} and the power series expansion method, in \cite{CC09} Cerpa-Cr\'epeau gave a large-time local controllability result for \eqref{intro-sys-KdV} for $L$ belongs to $\mathcal{N}_3\cup \mathcal{N}_4$. In $\mathcal{N}_3$ the analysis is similar to \cite{Cerpa07}, and it suffices to benefit on the different rotation vitesse of eigenfunctions to reach each direction in $M$. For example, assume that $M$ is of dimension four:
\begin{equation*}
    M= span \{\varphi_1, \varphi_2, \phi_1, \phi_2\}
\end{equation*}
 Using the complex analysis argument, one can show that the state $y$ can reach a certain direction $\varphi_0= \alpha \varphi_1+ \beta \varphi_2+ \gamma \phi_1+ \delta \phi_2$.  Now, using the fact that the angle velocity of $\varphi_i$ is different from the velocity of $\phi_j$ and the simple combination of linear/nonlinear solutions, the solution can reach every state in $M$ provided that the period is large enough to allow the prepared rotation. 
The discussion for $\mathcal{N}_4$ is more involved. One has to combine the rotation process with the third order expansion. 
\vspace{3mm}

The breakthrough result on the small-time controllability is the one by Coron-Koenig-Nguyen in 2020, published in \cite{CKN}. They treated a part of the critical lengths belong to $\mathcal{N}_2, \mathcal{N}_3$ by adding another assumption: every pair $(k, l)$ must satisfy  $
2k+l\not\equiv 0 \; 
({\rm mod} 3)$. They proved that small-time controllability cannot be achieved for these critical lengths.

They introduced an approach to reduce the problem to the non-triviality of a constant $E$, which governs the leading order of non-vanishing term of $Q_M$;  see more details in Section \ref{sec: reduction approach}.
\[
\int_{0}^{\infty} \int_0^L |y (t, x)|^2 e^{-\ii pt }\varphi_x(x) dx dt \sim
\|u\|^2_{H^{-\frac{2}{3}}(\R)}\left(E + \bigO(T) \right).
\]
The condition $2k+l\notin3\N^*$ is used to ensure that $E\neq 0$. And in their paper, they made no further comments on this condition.

\subsubsection{The asymptotic stability problem}
Now, we consider the stability of the linearized KdV equation \eqref{eq: intro-linear-KdV} and the nonlinear KdV equation \eqref{intro-sys-KdV} with control $u= 0$. 

Following Hilbert Uniqueness Method and the observability inequality dues to Rosier \cite{Rosier97}, the linearized system is exponentially stable if and only if $L\not \in \mathcal{N}$. Thus, the nonlinear system is locally exponentially stable at non-critical lengths. For critical lengths, due to the existence of $M$, the nonlinear system is not exponentially stable.  What is the asymptotic stability of the nonlinear system at critical lengths? 

The answer also depends on the value of different critical lengths. In the literature people also followed the existing classification criterion \eqref{eq:old:classifi}, where the dimension of $M$ is also a useful index to distinguish different kinds of critical lengths. 
In \cite{Chu-Coron-Shang}, using the ideas of center manifolds,  the authors considered the case $L\in \mathcal{N}_1$ and proved a polynomial decay for KdV systems \eqref{intro-sys-KdV}. Then, similar results were proved for the case $L=2\pi\sqrt{7/3}$ in \cite{Tang-Chu-Shang-Coron}, and the authors indicated that they think the method is useful for  all $L\in \mathcal{N}_2$. A more general result \cite{Nguyen} reveals that the asymptotic stability holds for KdV systems \eqref{intro-sys-KdV} at critical lengths in $\mathcal{N}_2\cup \mathcal{N}_3$ that satisfies another specific condition, where the author used the quasi-periodic function.

\subsubsection{The exponential stabilization problem}
The analysis of critical sets based on the dimension of $M$ is also used in stabilization problems. When $L\not\in \mathcal{N}$, the system is exponentially stable thanks to the exponential stability of the linearized system. Then, it is further proved that with the help of a well-designed feedback law,  the system at non-critical lengths can become rapidly stable (namely, the decay rate can be as fast as we want)  \cite{Coron-Lv}. 
Later on, for critical lengths in $\mathcal{N}_2\cup \mathcal{N}_3$, using power series expansion method, the exponential stabilization has been achieved in  \cite{coron-rivas-xiang}. So far, the exponential stabilization in $\mathcal{N}_1\cup \mathcal{N}_4$ is still open.

\subsection{A novel classification of critical lengths}\label{sec: new classification}

Our new classification scheme is based on the congruence relation modulo $3$.
Our point of view is based on the spectral analysis of the stationary KdV operator, namely, what are the different behaviors of the eigenmodes for the stationary operator at different critical lengths. 
\begin{align}\label{def:ope:A}
    \mathcal{A}_L: D(\mathcal{A}_L)\subset H^3(0, L)&\rightarrow L^2(0, L) \\
    \varphi \mapsto -\varphi'''- \varphi'
\end{align}
with
\begin{equation*}
    D(\mathcal{A}_L) = \{\varphi\in H^3(0, L): \varphi(0)= \varphi(L)=0, \varphi'(0)= \varphi'(L)\}.
\end{equation*}
It looks strange that we consider an operator that is not compatible with the linear equation \eqref{eq: intro-linear-KdV}. But this operator shares the advantage of being skew-adjoint, and it turns out to be a good operator for the study of nonlinear KdV problems and for the classification of critical lengths. 
\subsubsection{Motivations}\label{sec: motivations}
 As illustrated in Section \ref{sec: history},  we have observed markedly different behaviors for both control and stability problems depending on whether $L\in\mathcal{N}$ or $L\notin \mathcal{N}$. Additionally, at a critical length $L\in\mathcal{N}$, the space $L^2(0,L)$ decomposes into $H\oplus M$. In the subspace $H$, the linearized KdV equations exhibit both null controllability and exponential stability. Conversely, in the subspace $M$, neither null controllability nor exponential stability is achieved due to the existence of Type I eigenfunctions of the form \eqref{eq: type-1-eigenfunction-intro}. \\ What about the case when $L$ is close to $\mathcal{N}$? From a limiting perspective, a good understanding of lengths near critical values can provide valuable insights into the analysis of critical lengths.
\begin{aquestion}\label{OP: H-M-decompose}
Given $T>0$ and $L_0\in \mathcal{N}$, when $L\notin\mathcal{N}$ approaches $L_0$, can we define a similar decomposition $H(L)\oplus M(L)$ such that different behaviors are observed in $H(L)$ and $M(L)$? What is the asymptotic description of this behavior?  Furthermore, how does this behavior extend to the nonlinear case?
\end{aquestion}
Briefly, we would like to better understand how the unreachable subspace $M$ is formulated. We view this formation of $M$ as a limiting process where $M(L)\to M(L_0)$ as $L\to L_0\in\mathcal{N}$. Moreover, based on that, we are able to give a quantitative description of the stability and control problems related to KdV equations.

Recall that Rosier \cite{Rosier97} proved that for every $L_0\in \mathcal{N}$ there exists Type 1 eigenfunction of the operator $\mathcal{A}_{L_0}$, namely the eigenfunction verifying $\varphi'(0)= \varphi'(L_0)= 0$. We further observe that only for critical lengths $L_0$ such that the pair $(k, l)$ satisfies $2k+ l\in 3\N^*$, the operator $\mathcal{A}_{L_0}$ admits Type 2 eigenfunctions, that is the modes satisfying  $\varphi'(0)= \varphi'(L_0)\neq 0$.

 Heuristically speaking,  the analysis for the operator $\mathcal{A}_L$ with $L$ sufficiently close to $L_0$ can be understand in a perturbative point of view. Due to the operator perturbation theory,  the asymptotic analysis not only depends on Type 1 eigenfunctions but also relies on Type 2 eigenfunctions.  We refer to Section \ref{sec: further perspectives} for more discussion on this limiting problem. 
 In general, depending on the value of $(k, l)$ there are different behaviors of eigenvalues and eigenfunctions: 
 \begin{itemize}
     \item {\it Case $k= l$.}  The corresponding eigenvalue is zero, with multiplicity two. It has two eigenfunctions, one of Type 1 and the other of Type 2.
      \item {\it Case  $2k+ l\notin 3\N^*$. } The corresponding eigenvalue is nonzero, with multiplicity one. It has only one eigenfunction, one of Type 1.
      \item {\it Case  $2k+ l\in 3\N^*$ and $k\neq l$. }  The corresponding eigenvalue is nonzero, with multiplicity two. It has two eigenfunctions, one of Type 1 and the other of Type 2.
 \end{itemize}
In conclusion, when considering problems at critical lengths,  the effective factor is not the dimension of $M$ but the unreachable pair $(k,l)$, and one should consider $(k, l)$ according to the above-mentioned classes.

 After resisting the papers \cite{CC04,Cerpa07,CC09,CKN}, combining with our classification on pairs $(k,l)$, we found that the remaining case for the Open Problem is $2k+l\in3\N^*$, and this case is degenerate from analytic viewpoint. We refer to more details in Section \ref{sec: degenerate case}. 

\subsubsection{Classification criteria}
Now, we are in a position to present our classification of the critical sets. Define following index for $L\in\mathcal{N}$
\begin{equation}
\mathcal{I}(L):=3\left(\frac{L}{2\pi}\right)^2\in\N^*,\forall L\in\mathcal{N}.
\end{equation}
Consider the following simple Diophantine Equation 
\begin{equation}\label{integersolution}
    a^2+ab+b^2=n,
\end{equation}
or equivalently
\begin{equation}\label{eq:klIC}
    k^2+ kl+ l^2= \mathcal{I}(L), \textrm{ for } L\in \mathcal{N}.
\end{equation}
The solutions of the preceding algebraic equations lead to a description of dim $M$.
For each $L\in\mathcal{N}$, the equation \eqref{eq:klIC} may have multiple integer solutions. 
Inspired by the analysis of Type 1 and Type 2 eigenfunctions, we have the following natural classification for those solutions $(k,l)$ and we believe that pairs within the same class exhibit similar analytical and control properties. 
\begin{definition}\label{defi: types of unreachable pairs}
Let $L\in \mathcal{N}$. We define the following sets for the unreachable pairs $(k,l)$:
\begin{align*}
    \mathcal{S}_1(L)&:=\{(k,l) \textrm{ solution of } \eqref{eq:klIC}: k=l\}, \\
    \mathcal{S}_2(L)&:=\{(k,l) \textrm{ solution of } \eqref{eq:klIC}: k\equiv l\mod{3}, \; k\neq l\}\\
    \mathcal{S}_3(L)&:=\{(k,l) \textrm{ solution of } \eqref{eq:klIC}: k\not\equiv l\mod{3}\}.
\end{align*}
\end{definition}

One easily observe that 
\begin{equation*}
     \mathcal{S}_3(L)\cap  \mathcal{S}_1(L)=  \mathcal{S}_3(L)\cap  \mathcal{S}_2(L)= \emptyset.
\end{equation*}
\begin{remark}\label{rem:S123}
    Thus, for a given critical length, it either contains only pairs $(k, l)$ from $\mathcal{S}_3$  or pairs from $\mathcal{S}_1\cup \mathcal{S}_2$. In the commonly used classification criterion, no distinction is made between pairs from $\mathcal{S}_2$ and $\mathcal{S}_3$. This approach, based on the aforementioned rotation mechanism, is effective for large-time local controllability problems.  However, in practice, pairs from $\mathcal{S}_2$  and $\mathcal{S}_3$ exhibit fundamentally different spectral behaviors. These differences can heavily affect more delicate problems, such as small-time local controllability, asymptotic stability, and others.
\end{remark}

Based on the features of the integer pairs $(k,l)$ solving \eqref{eq:klIC}, we classify the critical lengths under three cases:
\begin{definition}\label{def:new:classification}
We have the following types of critical lengths,
\begin{align*}
    \mathcal{N}^1&:=\{L\in \mathcal{N}: \textrm{ there exists only one pair $(k, l)$ solving \eqref{eq:klIC}, which belongs to } \mathcal{S}_1(L)\}, \\
    \mathcal{N}^2&:= \{L\in \mathcal{N}: \textrm{ all solutions $(k, l)$ of \eqref{eq:klIC}  belong to } \mathcal{S}_3(L)\}\\
       \mathcal{N}^3&:=
    \{L\in \mathcal{N}: \textrm{ there exists pair $(k, l)$ solving \eqref{eq:klIC}, which belongs to } \mathcal{S}_2(L)\}.
\end{align*}
\end{definition}
Clearly, these sets are disjoint and 
\begin{equation}
    \mathcal{N}=\mathcal{N}^1\cup \mathcal{N}^2\cup\mathcal{N}^3. \footnote{In abuse of the notations, we use $\mathcal{N}_j$ for the classes under the previous category, while $\mathcal{N}^i$ for our new classification. }
\end{equation}
The situation for critical lengths from $\mathcal{N}^3$ is  most complicated, since it may contain both pairs from $\mathcal{S}_1$ and $\mathcal{S}_2$. While the other two classes only contain pairs from $\mathcal{S}_1$ or $\mathcal{S}_2$.
\begin{example}
For each case above, we could choose a model to catch a glimpse of its features. We choose $(k, l)= (1, 1)$ thus $L=2\pi$ as a model for $\mathcal{N}^1$,  $(k, l)= (2, 1)$ thus $L=2\pi\sqrt{7/3}$ for $\mathcal{N}^2$, and $(k, l)= (4, 1)$ thus $L=2\pi\sqrt{7}$ for $\mathcal{N}^3$.
\end{example}

\subsubsection{Type 1 and Type 2 eigenfunctions in different cases.}
Let $L\in\mathcal{N}$. As noticed in \cite{Rosier97}, there exists a finite number of pairs $\{(k_m,l_m)\}^{N}_{m=1}\subset \N^*\times \N^*$, with $k_m\geq l_m$ such that $L=2\pi\sqrt{\frac{k_m^2+k_m l_m+l_m^2}{3}}$.
We provide a brief description of the eigenmodes at the critical length in the following proposition and one can find its proof in Appendix \ref{sec: proof of prop-eigenmodes}. 
\begin{proposition}\label{prop: type-1-2}
For a fixed critical length $L\in \mathcal{N}$,  let $(k_m,l_m)$ be an unreachable pair such that 
\begin{equation*}
    L=2\pi\sqrt{\frac{k_m^2+k_m l_m+l_m^2}{3}}.
\end{equation*}
 Consider the eigenvalues of the operator $\mathcal{A}$ given by \eqref{def:ope:A}: 
\begin{equation}\label{eq: eigenvalue problem with critical length}
\left\{
\begin{array}{ll}
     \G'''+\G'+ i\lambda_m\G=0,  & \text{ in }(0,L),\\
     \G(0)=\G(L)=\G'(0)-\G'(L)=0,& 
\end{array}
\right.       
\end{equation}
with 
\begin{equation}\label{eq: defi of lambda_c}
\lambda_{m}:=\lambda_{m}(k_m,l_m)=\frac{(2k_m+l_m)(k_m-l_m)(2l_m+k_m)}{3\sqrt{3}(k_m^2+k_m l_m+l_m^2)^{\frac{3}{2}}}. 
\end{equation}
Then,
\begin{enumerate}
    \item If $(k_m,l_m)\in \mathcal{S}_3(L)$, there exists a unique solution $\G_m$ in the form:
    \begin{equation}\label{eq: exact formula for critical eigenfunctions}
    \G_m(x):=-e^{ i x\frac{\sqrt{3} (2 k_m + l_m) }{3\sqrt{k_m^2 + k_m l_m + l_m^2}}}  -\frac{k_m}{l_m} e^{- i x\frac{\sqrt{3} (k_m + 2 l_m) }{3 \sqrt{k_m^2 + k_m l_m + l_m^2}}}  + \frac{k_m+l_m}{l_m}e^{ i x\frac{\sqrt{3} (-k_m + l_m)}{ 3\sqrt{k_m^2 + k_m l_m + l_m^2}}}.    
    \end{equation}
    In particular, $\G'_m(0)=\G'_m(L)=0$.
    \item If $(k_m,l_m)\in\mathcal{S}_1(L)\cup\mathcal{S}_2(L)$, or equivalently $2k_m+l_m\in3\N^*$, there are two solutions to \eqref{eq: eigenvalue problem with critical length}. One is $\G_m$ defined in \eqref{eq: exact formula for critical eigenfunctions}, while the other is in the following form:
    \begin{equation}\label{eq: defi of tilde-G}
\Tilde{\G}_m(x):= e^{ i x\frac{\sqrt{3} (2 k_m + l_m) }{3\sqrt{k_m^2 + k_m l_m + l_m^2}}}  - e^{- i x\frac{\sqrt{3} (k_m + 2 l_m) }{3 \sqrt{k_m^2 + k_m l_m + l_m^2}}}.   
\end{equation}
Additionally, $\G_m$ and $\Tilde{\G}_m$ are linearly independent, while 
\[
\Tilde{\G}_m'(0)=\Tilde{\G}_m'(L)=\frac{ i\sqrt{3}(k_m+ l_m)}{\sqrt{k_m^2 + k_m l_m + l_m^2}}\neq0.
\]
All solutions to \eqref{eq: eigenvalue problem with critical length} are the linear combination of $\G_m$ and $\Tilde{\G}_m$.
\end{enumerate}
\end{proposition}

We call $\G_m$ defined in \eqref{eq: exact formula for critical eigenfunctions} the \textit{Type 1} eigenfunctions.  As is mentioned in \cite[Remark 3.6]{Rosier97}, for $L\in\mathcal{N}$, the eigenfunction $\G_m$ can generate an unreachable state for linearized KdV equations. We denote by $N_0$ the number of Type 1 eigenfunctions, which is also the dimension of the unreachable subspace.

Surprisingly, we notice the existence of another type of eigenfunctions $\Tilde{\G}_m$ if $2k_m+l_m\in3\N^*$ with nonvanishing Neumann boundary conditions. These solutions to \eqref{eq: eigenvalue problem with critical length} with nonzero Neumann boundary conditions are referred to as \textit{Type 2 eigenfunctions}. In summary,
\begin{definition}\label{def:type12}
Let $L\in \mathcal{N}$ and $\A$ be given by \eqref{def:ope:A}. Let $(k,l)$ be an unreachable pair solving \eqref{eq:klIC} and $\lambda=\lambda(k,l)$ be given by \eqref{eq: defi of lambda_c}. 
\begin{itemize}
    \item We say $\G$ is a \textbf{Type 1} eigenfunction if 
    \[
    \left\{
\begin{array}{l}
     \G'''+\G'+ i\lambda\G=0,\\
     \G(0)=\G(L)=\G'(0)=\G'(L)=0.
\end{array}
\right.
    \]
    \item We say $\Tilde{\G}$ is a \textbf{Type 2} eigenfunction if 
    \[
    \left\{
\begin{array}{l}
     \Tilde{\G}'''+\Tilde{\G}'+ i\lambda\Tilde{\G}=0,\\
     \Tilde{\G}(0)=\Tilde{\G}(L)=0,\;
     \Tilde{\G}'(0)=\Tilde{\G}'(L)\neq0.
\end{array}
\right.
    \]
\end{itemize}
\end{definition}

\begin{remark}
We emphasize that Type 1 eigenfunctions appear for all critical lengths. However, Type 2 eigenfunctions only appear for $L\in \mathcal{N}^1\cup\mathcal{N}^3$.
Equivalently, Type 2 eigenmodes exist if there exists a pair $(k,l)$ such that $2k+l\in3\N^*$.
\end{remark}
\begin{corollary}
Let $(k,l)$ satisfy that $2k+l\in3\N^*$ and $L=2\pi\sqrt{(k^2+kl+l^2)/3}$ and $\G$ be a Type 1 eigenfunction. Then, its derivative $\G'$ is automatically of Type 2.  
\end{corollary}
\begin{proof}
It is easy to check that $f=\G'$ satisfies 
\[
\left\{
\begin{array}{ll}
     f'''+f'+ i\lambda(k,l)f=0,  & \text{ in }(0,L),\\
     f(0)=f(L)=0.& 
\end{array}
\right.      
\]
We only focus on its Neumann boundary conditions. Direct computations derive
\[
\G''(0)=\frac{3k(k+l)}{k^2+kl+l^2}, \;
\G''(L)=\frac{3e^{-\ii\frac{2(2k+l)\pi}{3}}k(k+l)}{k^2+kl+l^2}.
\]
Since $2k+l\in3\N^*$, then $e^{-\ii\frac{2(2k+l)\pi}{3}}=1$, which implies that $\G''(0)=\G''(L)\neq 0$. Then $\G'$ is of Type 2.
\end{proof}

\subsection{Characteristics of dimensions of unreachable subspaces}\label{sec: dimension characteristics}
In this sequel, we prove the following proposition.
\begin{proposition}
For any $d\in\N^*$, there are infinitely many $L_0\in \mathcal{N}$ such that dim $M= d$.
\end{proposition}
This is a direct consequence of the following stronger and quantitative result. Although this result is not our primary objective, it appears to be the first full characterization of the unreachable dimension $N_0$ for all $L_0\in\mathcal{N}$.  Moreover, this full characterization is strongly linked to our new classification criteria. 
\begin{proposition}\label{lem: counting}
For any $n\in\N$, suppose that the decomposition of the positive integer $n$ is given by
\begin{equation}\label{eq: decomposition of integer n}
    n=3^\alpha p_1^{\beta_1}\cdots p_r^{\beta_r} q_1^{\gamma_1}\cdots q_s^{\gamma_s},\ \alpha,\beta_i,\gamma_j\in \mathbb{N}.
\end{equation}
Here $p_i\equiv 1\pmod{3}$ and $q_j\equiv 2\pmod{3}$ are distinct prime integers. Let $N(n)$ denote the number of positive integer solutions to the equation \eqref{integersolution} for $(a,b)\in\N^*\times\N^*$. In general, let $Z(n)$ denote the number of integer solutions to the equation \eqref{integersolution}. Thus, $Z(n)$ and $N(n)$ have the following properties
\begin{enumerate}
    \item If at least one of the $\gamma_j$'s in \eqref{eq: decomposition of integer n} is odd, then the equation \eqref{integersolution} has no integer solutions, i.e. $Z(n)=0$;
    \item If the $\gamma_j$'s in \eqref{eq: decomposition of integer n} are all even, then 
    \begin{equation*}
        Z(n)=\left\{\begin{array}{ll}
           6(\beta_1+1)\cdots (\beta_r+1), & r>0 \\
            6, & r=0.
        \end{array}
    \right.
    \end{equation*}
    \item If the $\gamma_j$'s in \eqref{eq: decomposition of integer n} are all even, then
    \begin{equation*}
        N(n)=\left\{\begin{array}{ll}
           \frac{1}{6}Z(n), & n  \text{ is not a perfect square},\\
           \frac{1}{6}Z(n)-1, &n  \text{ is a perfect square}.
        \end{array}
    \right.
    \end{equation*}
\end{enumerate}
\end{proposition}

\begin{proof}[Proof of Proposition \ref{lem: counting}]
Let $\omega:=e^{\frac{2\pi i}{3}}=\frac{-1+\sqrt{3}i}{2}$ and $\mathbb{Z}[\omega]:=\{a+b\omega\mid a,b\in \mathbb{Z}\}$. The norm of $a+b\omega\in \mathbb{Z}[\omega]$ is defined as
\[N(a+b\omega):=|a+b\omega|^2=(a+b\omega)(a+b\bar{\omega})=a^2-ab+b^2.\]
Then for $n\in \mathbb{N}$, the equation \eqref{integersolution} $ a^2+ab+b^2=n$, where $a,b\in \mathbb{Z}$, is associated to the norm relation $N(a-b\omega)=n$. We shall identify $(a,b)$ with $a-b\omega$.

We apply the following well-known properties of $\mathbb{Z}[\omega]$. See the book \cite[\S 4.A]{Cox13} for details.

\begin{lemma} $\mathbb{Z}[\omega]$ has the following properties:
    \begin{enumerate}
        \item[(1)] $\mathbb{Z}[\omega]$ is a unique factorization domain.

        \item[(2)] The units in $\mathbb{Z}[\omega]$ are $\{\pm 1,\pm \omega,\pm \omega^2\}=\{(-\omega)^k\mid k=1,\dots 6\}$.

        \item[(3)] Let $p\in \mathbb{N}$ be a prime integer.

        \begin{itemize}
            \item If $p=3$, then $p=-\omega^2(1-\omega)^2$, where $1-\omega$ is a prime and $N(1-\omega)=3$;

            \item If $p\equiv 1\pmod{3}$, then $p=\pi_p\bar{\pi}_p$, where $\pi_p,\bar{\pi}_p$ are primes and $N(\pi_p)=N(\bar{\pi}_p)=p$;

            \item If $p\equiv 2\pmod{3}$, then $p$ remains a prime in $\mathbb{Z}[\omega]$, and $N(p)=p^2$.
        \end{itemize}

        Moreover, every prime element in $\mathbb{Z}[\omega]$ is a divisor of a unique prime integer $p$.
    \end{enumerate}
\end{lemma}
Armed with the preceding result,  we are now in a position to prove the first two statements. It's reduced to find the number of $a-b\omega\in \mathbb{Z}[\omega]$ with norm $n$. In view of the classification of prime elements in $\mathbb{Z}[\omega]$, the power of each $q_j$ in $N(a-b\omega)$ must be even. And if all of the $\gamma_j$'s are even, then $N(a-b\omega)=n$ if and only if
    \[a-b\omega=(-\omega)^k \cdot \prod_{i=1}^r \pi_{p_i}^{u_i}\bar{\pi}_{p_i}^{\beta_i-u_i}\cdot \prod_{j=1}^s q_j^{\frac{\gamma_j}{2}}.\]
    Here $k=1,\dots ,6$ and each $u_i=0,1,\dots ,\beta_i$. And the first two statements hold. 
    
    We turn to the last statement. It can be verified that, if $n$ is not a perfect square, then each class contains exactly one solution $a-b\omega$ with $a,b>0$; and if $n=m^2$, then this property still holds for all classes but one:
    \[\{(-\omega)^k m\mid k=1,\dots ,6\}=\{\pm m,\pm m\omega,\pm (m+m\omega)\},\]
    which contains no positive integer solution. 
\end{proof}
\subsection{Further applications of this classification}\label{sec: further perspectives} 
We believe this new criterion is essential and can be applied to a wide range of problems.
In the sequel, we briefly comment on the further perspectives where our new classification may be applicable.

\subsubsection{Controllability and stability via a  limiting process}
As presented in Section \ref{sec: history}, considering controllability and stability,  there is a significant distinction between the cases $L\notin\mathcal{N}$ and $L\in\mathcal{N}$. Let us take the linear KdV system \eqref{eq: intro-linear-KdV} as an example.
\begin{itemize}
    \item Null controllability holds for $L\notin\mathcal{N}$ and there exists a constant $C(T,L)>0$ such that 
    \[
    \|u\|_{L^2(0,T)}\leq C(T,L)\|y_0\|_{L^2(0,L)}.
    \]
    However, for $L\in\mathcal{N}$, there exists an unreachable subspace $M$ defined as \eqref{eq: defi-M}. A natural expectation is that 
    \[
    C(T,L)\to\infty,\mbox{ as } L\to\mathcal{N}.
    \]
    Unfortunately, the classical compactness-uniqueness method (see \cite{Rosier97}) can only ensure the existence of this constant $C(T,L)$. There is a lack of an estimate of the size of $C(T,L)$ and meanwhile the dependence of the parameter $L$ is also unknown. 

    To track the dependence of $L$, we need to adapt the moment method involving the eigenfunctions of $\A_{L}$ (defined in \eqref{def:ope:A}). Considering the eigenvalue problem:
    \[
\A_L\G(L)=\ii\lambda(L)\G(L), \mbox{ in }(0,L).
\]
We have to analyze the asymptotic behavior of each eigenmode $(\G(L),\ii\lambda(L))$ as $L\to\mathcal{N}$ to detect where the control cost blows up. At this stage, our new classification criteria come into play. For $L_0\in \mathcal{N}^2$, there are no Type 2 eigenfunctions. We need to track those eigenmodes $(\G(L),\ii\lambda(L))$ such that
\[
(\G(L),\ii\lambda(L))\to\mbox{Type 1 eigenmodes, as }L\to L_0.
\]
As for $L_0\in\mathcal{N}^1$ or $L_0\in\mathcal{N}^3$, the asymptotic behaviors of $(\G(L),\ii\lambda(L))$ are rather complicated due to the appearance of Type 2 eigenmodes. We will provide a detailed and complete analysis in a forthcoming paper and as a consequence, we give a quantitative estimate of the important constant $C(T,L)$.
\vspace{2mm}
\item When it comes to the stability problem, we refer to \eqref{eq: intro-linear-KdV} with $u\equiv0$. For $L\notin\mathcal{N}$, an exponential stability holds
\[
\|y(t)\|^2_{L^2(0,L)}\leq Ce^{-R(L) t}\|y_0\|^2_{L^2(0,L)}, \forall t>0\mbox{ with a constant }R(L)>0,
\]
while for $L\in\mathcal{N}$, no stability is observed in $M$ due to the existence of traveling waves $e^{\ii t\lambda}\varphi_{\lambda}$ with $\G_{\lambda}$ defined in \eqref{eq: type-1-eigenfunction-intro}. A natural expectation is 
\[
R(L)\to0,\mbox{ as } L\to\mathcal{N}.
\]
Unfortunately, the existence of this constant $R(L)$ is again ensured by a compactness-uniqueness method. Applying Hilbert uniqueness method twice, we can transform this stability problem into a controllability problem. Then the asymptotic behaviors of eigenmodes $(\G(L),\ii\lambda(L))$ will play a significant role.  
\end{itemize}
\begin{remark}
Now we change a point of view. By a rescaling process, we could write 
\[
\A_L=\A_{L_0}+(L-L_0)\mathcal{R}_L,
\]
with $\mathcal{R}_L=\frac{L_0^2+L_0L+L^2}{L_0^3}\A_{L_0}-\frac{L(L+L_0)}{L_0^3}\p_x$.  For $L_0\in\mathcal{N}$, we can also see the previous discussion on controllability/stability as a uniform controllability/stability problem for the following system
\[
\left\{
\begin{array}{ll}
\p_t y -\A_{L_0}y-(L-L_0)\mathcal{R}_Ly = 0, &  \text{ in }  (0, T)\times(0, L_0), \\
y(t, 0) = y(t, L_0) = 0, & \text{ in } (0, T), \\
\p_x y(t , L) = u(t), & \text{ in } (0, T),
\end{array}\right.
\]
under the perturbation $(L-L_0)\mathcal{R}_L$ as $L\to L_0$.
\end{remark}

\subsubsection{Small-time local controllability and the open problem}
We observe that the existing results on small-time local controllability pertain to the cases $\mathcal{N}^1$ \cite{CC04} and $\mathcal{N}^2$ \cite{CKN}. The only remaining case to be addressed is $L \in \mathcal{N}^3$, which forms the primary objective of this paper.

\subsubsection{Large-time local controllability}
For $L\in \mathcal{N}_2$,  Cerpa \cite{Cerpa07} demonstrated that the system is large-time locally controllable. As discussed, when the dimension of MM is 2, the critical length can belong to either $\mathcal{N}^2$ or $\mathcal{N}^3$. We believe that by distinguishing between these two classes, it may be possible to gain deeper insights into large-time local controllability, such as obtaining more precise quantitative estimates. Likewise, this approach could lead to improvements in the results of \cite{CC09}.

\subsubsection{Exponential stabilization at critical legnths}
The exponential stabilization in $\mathcal{N}_2 \cap \mathcal{N}_3$ depends on the large-time local controllability property \cite{coron-rivas-xiang}. Thus it is reasonable to consider exponential stabilization under the new classification. We first consider the problem in $\mathcal{N}^2$, then in $\mathcal{N}^1$ and $\mathcal{N}^3$. 

\subsubsection{Asymptotic stability without control at critical lengths}

We also aim to achieve more comprehensive and refined results concerning asymptotic stability. Currently, it remains an open question whether the KdV system \eqref{intro-sys-KdV} exhibits polynomial decay when the dimension $\dim M$ is odd and greater than 1. Even in the even-dimensional setting, where Nguyen established polynomial decay \cite{Nguyen}, several cases remain unclear. In particular, it would be valuable to determine whether the decay rates differ between the cases $\mathcal{N}^2$ and $\mathcal{N}^3$ (still in even-dimensional setting). For the latter, the presence of Type 2 eigenmodes may introduce certain degeneracies, potentially affecting the decay behavior.

\section{Some preparations for the proof}\label{sec: preparations}
For clarity and completeness, we present some useful elements necessary for subsequent proofs.  
\begin{itemize}
    \item Section \ref{sec: well-posedness} reviews classic and recent results on the well-posedness of both linear and nonlinear KdV equations. We apply these well-posedness estimates in Section \ref{sec: obstruction} to prove the obstruction to small-time controllability. 
    \item In Section \ref{sec: 1d microlocal analysis}, we introduce fundamental concepts of microlocal analysis, focusing particularly on one-dimensional cases. These technical lemmas play a crucial role in deriving the coercive estimate for $Q_M$, detailed in Proposition \ref{prop: coercive property}. 
    \item In Section \ref{sec: control-0-0}, we characterize controls that steer $0$ to $0$ and establish the asymptotic expansion for the eigenvalues of stationary KdV operators. This analysis provides some special functions and expansions, which will be utilized in the proof of Lemma \ref{lem: B asymp}.
\end{itemize}

\subsection{Well-posedness of  KdV systems}\label{sec: well-posedness}
The following well-posedness results on both linear and nonlinear KdV equations can be found in \cite{CKN}, which are useful in the proof of Main Theorem. This subject has received extensive study over recent decades (see for example \cite{Bona03,BSZ-2, coron-rivas-xiang}). Additionally, studies on Kato smoothing effects, as detailed in \cite{KX,Rosier97}, are also noteworthy.

\begin{lemma} \label{lem: inhomogenous-kdv}  (The inhomogeneous linearized system,  \cite[Lemma 4.6]{CKN})
 Let  $h = (h_1, h_2, h_3) \in
H^{\frac{1}{3}}(\R_+)\times H^{\frac{1}{3}}(\R_+) \times L^2(\R_+)$, $f_1 \in L^1 \left((0, T) \times
(0, L) \right)$, and $f_2 \in L^1 \left((0, T); W^{1,1}(0, L) \right)$ satisfying
$f_2 (t, 0) = f_2 (t, L) = 0$.
Assume that the function $f = f_1 +\p_x f_2$ belongs to $L^1(\R_+; L^2(0, L))$. Let  $y \in
C\left([0, + \infty); L^2(0, L) \right) \cap L^2_{loc}\left([0, + \infty); H^1(0, L) \right) $
be the unique solution of 
\begin{equation}\label{eq: general linear kdv}
\left\{
\begin{array}{cl}
\p_t y + \p_x^3y  + \p_xy = f &  \mbox{ in } (0, +\infty) \times (0,
L), \\
y(t, 0) = h_1(t), \;  y(t, L) = h_2 (t), \;  \p_xy(t ,  L) = h_3(t) & \mbox{ in } (0,
+\infty),\\
y(0, x)  = 0, &\mbox{ in } (0, L).
\end{array}\right.
\end{equation}
Then 
\begin{equation*}
\| y \|_{L^2 \left( (0, T) \times (0, L) \right)} \le C_T \left( \| (h_1, h_2) \|_{L^2(\R_+)}
+ \|h_3 \|_{H^{-\frac{1}{3}}(\R)} + \| f\|_{L^1(\R_+ \times (0, L))} \right),
\end{equation*}
and
\begin{equation*}
\| y \|_{L^2 \left( (0, T); H^{-1} (0, L) \right)} \le C_T \left( \| (h_1, h_2)
\|_{H^{-\frac{1}{3}}(\R)} + \|h_3 \|_{H^{-\frac{2}{3}}(\R)} + \| (f_1, f_2)\|_{L^1(\R_+ \times (0, L))}
\right).
\end{equation*}
In addition, if for some $T_1\in (0, T)$ there is   $h(t, \cdot) = 0$ and $f(t, \cdot) = 0$ for $t \ge T_1$, then  for any $\delta > 0$ and for $t\in [T_1+ \delta, T]$ there is 
\begin{equation*}
|\p_ty (t, x)| + |\p_xy (t, x)| \le C_{T, T_1, \delta} \left( \| (h_1, h_2) \|_{H^{-\frac{1}{3}}(\R)}
+ \|h_3 \|_{H^{-\frac{2}{3}}(\R)} + \| (f_1, f_2) \|_{L^1(\R_+ \times (0, L))} \right).
\end{equation*}
\end{lemma}
In particular, concerning the controlled nonlinear KdV systems, we have the following 
\begin{lemma} \label{lem: kdv-NL} (The nonlinear system, \cite[Lemma 5.4]{CKN},\cite[Theorem 1.2]{Bona03})  There exists  constants $\varepsilon_0>0 $ and $C>0$ such that for any  $y_0 \in L^2(0, L)$ and any $u \in L^2(\R_+)$
satisfying 
\[
\| y_0 \|_{L^2(0, L)} + \| u \|_{L^2(\R_+)} \le \varepsilon_0,
\]
 the unique solution of 
\begin{equation*}\left\{
\begin{array}{cl}
\p_t y  + \p_x^3y  + \p_x y   +  y \p_x y = 0 &  \text{ in }(0, + \infty)
\times  (0, L), \\
y(t, 0) = y(t, L) = 0, \p_xy(t ,  L) = u(t) & \mbox{ in } (0, + \infty),  \\
y(0, \cdot) = y_0
\end{array}\right.
\end{equation*}
belongs to $C\left([0, + \infty); L^2(0, L) \right) \cap L^2_{loc}\left([0, + \infty); H^1(0, L) \right)$. Furthermore, it satisfies 
\begin{gather*}
\| y \|_{L^2\left((0, T) \times (0, L) \right)} \le C \left( \| y_0 \|_{L^2(0, L)} + \| u
\|_{H^{-\frac{1}{3}}(\R)} \right),\\
\| y \|_{L^2\left((0, T); H^{1} (0, L) \right)} \le C \left( \| y_0 \|_{L^2(0, L)} + \| u
\|_{L^2(\R_+)} \right),\\
\| y \|_{L^2\left((0, T); H^{-1} (0, L) \right)} \le C \left( \| y_0 \|_{L^2(0, L)} + \| u
\|_{H^{-\frac{2}{3}}(\R)} \right).
\end{gather*}
\end{lemma}
\begin{remark}
As observed in the preceding lemmas, we apply these results to linear and nonlinear KdV equations with low regularity data. This approach differs from traditional studies of well-posedness.
\end{remark}
\subsection{Basis elements in microlocal analysis}\label{sec: 1d microlocal analysis}
This section is devoted to presenting a few facts about the microlocal analysis on $1-$dimensional Euclidean spaces. For more details, we refer to \cite{Lerner-book,AG-book,Hormander-3} and its references. For the sake of completeness, we include the following results here. We say $a\in C^{\infty}(\R^2)$ is a symbol, belonging to $S^{m}(\R^2)$, if there are constants $C_{kl},k,l\in\N$ such that $a$ satisfies 
\begin{equation*}
|(\p_x^k\p_{\xi}^la)(x,\xi)|\leq C_{kl}\langle\xi\rangle^{m-l}, \langle\xi\rangle=(1+|\xi|^2)^{\frac{1}{2}},
\end{equation*}
For any $a\in S^{m}(\R^2)$, we define a pseudodifferential operator $a(x,D_x)$, with a convention $D_x=\frac{1}{\ii}\p_x$, by the standard quantization as follows:
\begin{equation}\label{eq: defi of pseudo-op}
a(x,D_x)u(x) = \frac{1}{2\pi} \int_{\R} \int_{\R} e^{\ii (x-z)\xi} a \left(x, \xi\right) u(z) \diff \xi\diff z,\forall u\in\mathcal{S}(\R),
\end{equation}
As an operator, $a(x,D_x)$ has the following properties:
\begin{proposition}\label{prop: properties of pseudo-op}
Let $a(x,D_x), b(x,D_x)$ be the pseudodifferential operators defined in \eqref{eq: defi of pseudo-op}. Then we have
\begin{enumerate}
    \item {\cite[Theorem 1.1.18]{Lerner-book}}Let $s,m\in\R$ and $a \in S^m(\R^2)$. Then, the operator $a(x,D_x)$ is bounded from $H^{s+m}(\R)$ to $H^s(\R)$. 
    \item {\cite[Theorem 1.1.20]{Lerner-book}}Let $m_1,m_2\in\R$ and $a\in S^{m_1},b\in S^{m_2}$. Then $a(x,D_x)\circ b(x,D_x)$ is also a pseudodifferential operator with a symbol $c\in S^{m_1+m_2}$ and we have the following asymptotic expansion\footnote{ Note that $D^k_{\xi}aD^k_{x}b\in S^{m_1+m_2-k}$.}, for $\forall n\in\N$,
\begin{equation*}
c=\sum_{k\leq n}\frac{1}{k!}D^k_{\xi}aD^k_{x}b+r_{n}(a,b),
\end{equation*}
with $r_n(a,b)\in S^{m_1+m_2-n}$. In particular, for the commutator of two pseudodifferential operators, we have $[a(x,D_x),b(x,D_x)]=\Tilde{c}(x,D_x)$, where
\[
\Tilde{c}=D_{\xi}aD_{x}b-D_{\xi}bD_{x}a+r_2(a,b)\in S^{m_1+m_2-1},
\]
with $r_2(a,b)\in S^{m_1+m_2-2}$.
\end{enumerate}
\end{proposition}
\begin{proposition}[Calder\'on--Vaillancourt Theorem]
Let $a\in C^{\infty}(\R^{2d})$ with all derivatives bounded. Then, $a(x,D_x)$ is a bounded operator on $L^2(\R^d)$. Moreover, there exist constants $C_0>0$ and $N_0\in\N^*$, only depending on the dimension $d$ such that 
\begin{align*}
\|a(x,D_x)\|_{\mathcal{L}(L^2(\R^d))}\leq C_0\sum_{|\alpha|+|\beta|\leq N_0}\|\p^{\alpha}_{x}\p^{\beta}_{\xi}a\|_{L^{\infty}(T^*\R^d)}.\label{eq: L^2-bound of Opa}
\end{align*}
\end{proposition}
Due to Calder\'on--Vaillancourt Theorem, we know that the operator norm of a pseudodifferential operator only depends on finite derivatives of its symbol. For the convenience of the proof later, we present an estimate of a particular commutator, between a Fourier multiplier, represented as $a(D_x)$, and a multiplication operator $b(x)$, symbolically expressed as $[a(D_x),b(x)]$. More precisely, here the case we consider is a specific example, $[\langle D_x\rangle^{-s},\varrho(\frac{\cdot}{R})]$ (we refer to Corollary \ref{cor: commutator-est} below for clear definitions). Since this is a simple example, by direct computation and standard oscillation integral techniques, we can get the dependence of the parameter $R$ even though it is not the optimal one. We include the proof in Appendix \ref{sec: commutator}. 
\begin{corollary}\label{cor: commutator-est}
 Let $\varrho\in C^{\infty}_c(\R)$ and $\supp{\varrho}\subset(-1,1)$. Set $\varrho_R:=\varrho(\frac{\cdot}{R})$. Then, for $s>0$, there exists a constant $C>0$, which is independent of $R$, such that the commutator $[\langle D_x\rangle^{-s},\varrho_R]$ satisfies
\[
\|[\langle D_x\rangle^{-s},\varrho_R]\|_{\mathcal{L}(H^{-s-1},L^2)}\leq \frac{C}{R^{s+4}},
\]
where $\langle D_x\rangle^{-s}=(1+D_x^2)^{-\frac{s}{2}}$.
\end{corollary}
Now we aim to present some basic lemmas on the Sobolev norms of compactly supported distributions. One can find the general cases in \cite[Appendix 4.3.3]{Lerner-book}. 
\begin{lemma}\label{lem: embedding-compact-support}
Let $s_1>-\frac{1}{2}$ and $s_1>s_2$ and $I=[-1,1]$. There exists a constant $C=C(I,s_1,s_2)>0$ such that for all $R\in(0,1)$ and $u\in H^{s_1}(\R)$ with $\supp u\subset [-R,R]\subset[-1,1]$, we have 
\begin{equation*}
\|u\|_{H^{s_2}(\R)}\leq CF(R,s_1,s_2)\|u\|_{H^{s_1}(\R)},
\end{equation*}
where $\lim_{R\to0}F(R,s_1,s_2)=0$ and more precisely, for $R\in(0,1)$
\begin{equation*}
F(R,s_1,s_2)=\left\{
\begin{array}{ll}
    R^{s_1-s_2} &s_2>-\frac{1}{2},  \\
    R^{s_1-s_2}(1+\ln{\frac{1}{R}})^{\frac{1}{2}} & s_2=-\frac{1}{2},\\
     R^{s_1+\frac{1}{2}}(1+\ln{\frac{1}{R}})^{\frac{1}{2}}&s_2<-\frac{1}{2}.
\end{array}
\right.
\end{equation*}
\end{lemma}
Its proof can be found in \cite[Appendix 4.3.3]{Lerner-book}. And we emphasize that the constant $C$ is independent of $R$.

The next proposition is called Peetre's inequality, which is an important tool in microlocal analysis.
\begin{proposition}[Peetre's inequality]\label{prop: peetre}
Let $s\in\R$. For any $\xi,\eta\in\R$,  the following inequality holds:
\begin{equation}\label{eq: peetre}
\left(\frac{1+|\xi|^2}{1+|\eta|^2}\right)^s\leq 2^{|s|}\left(1+|\xi-\eta|^2\right)^{|s|}.
\end{equation}
\end{proposition}

\subsection{Formulation of the controls steering 0 to 0 during time period $(0, T)$}\label{sec: control-0-0}
As what is done in \cite[Section 2]{CKN}, we define the following quantities
\begin{definition}\label{def:Q-P}
For $\tau \in \C$,  let $ \left(\lambda_j(\tau) \right)_{1\leq j \leq 3}$ be the three solutions   of (counting multiplicity)
 \begin{equation}\label{eq: lambda-eigenvalue-tau}
  \lambda^3 + \lambda + i \tau = 0.
 \end{equation}
Define 
 \begin{equation}\label{eq: defi-P-Q}
 Q (\tau) : = \sum_{j=1}^3 (\lambda_{j+1} - \lambda_j) e^{\lambda_{j} L  + \lambda_{j+1} L
},  P(\tau) : =  \sum_{j=1}^3  \lambda_j(e^{\lambda_{j+2} L } - e^{\lambda_{j+1} L}) 
\end{equation}
 and
 \begin{equation}\label{eq: defi-Xi}
\Xi(\tau) := -  (\lambda_2 - \lambda_1) (\lambda_3 - \lambda_2) (\lambda_1 - \lambda_3), 
\end{equation}
where we adapted  the convention  that $\lambda_{j+3} = \lambda_{j}$ for $j= 1,2,3$. 
\end{definition}

 For an appropriate function $v$ defined
 on $\R_+ \times (0, L)$,  we define its Fourier transform w.r.t time  as following: 
$\hat v$ its Fourier transform \footnote{It coincides with the Fourier transform of the extension of $v$ by $0$ for $t<0$.} with respect to $t$, 
\begin{equation*}
\hat v(\tau, x) = \frac{1}{\sqrt{2 \pi} }\int_0^{+\infty} v(t, x) e^{- i \tau t} dt, \forall \tau \in \C.
\end{equation*}

\begin{lemma}(\cite[Lemma 2.4]{CKN})\label{lem: formulation sol} 
Let $u \in L^2(0, + \infty)$  and let $y$ be the
unique solution of
\begin{equation}\label{eq: linearized controled kdv-0-0}
\left\{
\begin{array}{cl}
\p_t y + \p_x^3y  + \p_x y  = 0 & \mbox{ in }(0, +\infty) \times (0, L),
\\
y(t, 0) = y(t, L) = 0, \p_xy(t ,  L) = u(t) & \mbox{ in }(0, +\infty), \\
y(0, x) =  0 &\mbox{ in }(0, L).
\end{array}\right.
\end{equation}
Then,  outside of a discrete set $\tau\in \R$, there is
\begin{equation}\label{eq: defi-Fourier-y}
\Hat{y} (\tau, x) =  \frac{\Hat{u} }{\det Q}  \sum_{j=1}^3 (e^{\lambda_{j+2} L } - e^{\lambda_{j+1}
L
}) e^{\lambda_j x}, \mbox{ for a.e. } x \in (0, L).
\end{equation}
and $\partial_x\Hat{y}(\tau,0)  = \frac{\Hat{u} (\tau) P(\tau)}{\det Q(\tau)}$.
\end{lemma}
Recall  that $\lambda_j(\tau)$ are solutions of 
\begin{equation*}
\lambda_j^3+\lambda_j+i \tau=0.
\end{equation*}
For $p\in\R$, we further define  $\Tilde{\lambda}_j(\tau; p)$ be solutions of 
\[
\Tilde{\lambda}_j^3+\Tilde{\lambda}_j-i(\Bar{\tau}-p)=0.
\]

Thanks to the equation~\eqref{eq: lambda-eigenvalue-tau} satisfied by
the $\lambda_j$, one has the following asymptotic behavior of $\lambda_j$ concerning large positive $\tau$: 
\begin{lemma}(\cite[Lemma 3.3]{CKN})\label{lem: lambda-asymp}
Let   $p \in \R$ and $\tau$ in a small enough conic neighborhood of $\R_+$. 
Consider the convention $\Re(\lambda_1) < \Re(\lambda_2) < \Re(\lambda_3)$ and similarly
for $\Tilde{\lambda}_j$. We have
\begin{equation}
\lambda_j = \mu_j\tau^{\frac{1}{3}} - \frac1{3\mu_j}\tau^{-\frac{1}{3}} + \bigO(\tau^{-\frac{2}{3}}), \;\; |\tau|\gg 1
\end{equation}
\begin{equation}
\Tilde{\lambda}_j = \Tilde{\mu}_j \Bar{\tau}^{\frac{1}{3}} - \frac1{3\Tilde{\mu}_j} \Bar{\tau}^{-\frac{1}{3}}+\bigO(\tau^{-\frac{2}{3}}), \;\;   |\tau|\gg 1, 
\end{equation}
where $\mu_j$ and $\tilde \mu_j$ are defined as (see Figure \ref{fig:mu}. )
\begin{equation*}
\mu_j = e^{-\frac{\ii\pi}{6}-\frac{2\ii j\pi}{3}} \textrm{ and }
\Tilde{\mu}_j = e^{\frac{\ii\pi}{6}+\frac{2\ii j\pi}{3}}.
\end{equation*}
 Here $\tau^{\frac{1}{3}}$ denotes the cube root of $\tau$ with the real part positive.  
\end{lemma}
\begin{figure}[h] 
    \centering 
    \includegraphics[width=0.4\textwidth]{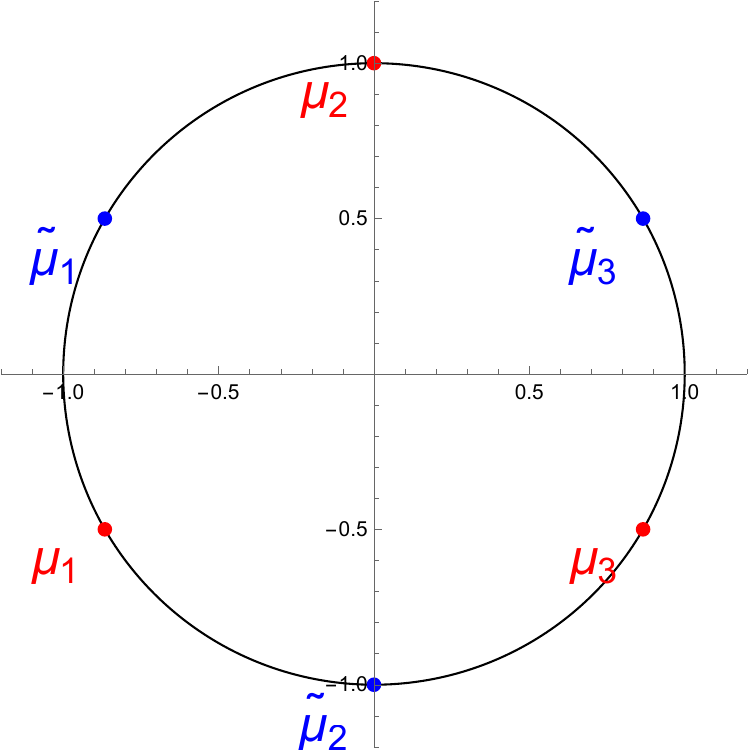}
\caption{Distribution of $\mu_j$ and $\Tilde{\mu}_j$}
    \label{fig:mu} 
\end{figure}
Next lemma introduces two special functions, which can be found in \cite[Page 1214]{CKN}.
\begin{lemma}\label{lem: defi-G-H}
Let $H(\tau):=\frac{\det Q(\tau)}{\Xi(\tau)}$ and $G(\tau):=\frac{P(\tau)}{\Xi(\tau)}$ ($P$,  $Q$, and $\Xi$ are defined in \eqref{eq: defi-P-Q} and \eqref{eq: defi-Xi}). Then $G$ and $H$ are entire functions. Moreover, let $\tau_1,\cdots,\tau_k$ be the distinct common roots of $G$ and $H$ in $\C$ and
\begin{equation*}
\Gamma(\tau):=\prod_{j=1}^k(\tau-\tau_j).
\end{equation*}
Then the following two functions are entire
\begin{equation*}
\mathcal{G}(\tau):=\frac{G(\tau)}{\Gamma(\tau)},\mathcal{H}(\tau):=\frac{H(\tau)}{\Gamma(\tau)}
\end{equation*}
and $\mathcal{G},\mathcal{H}$ have no common roots.
\end{lemma}
As we observed in Lemma \ref{lem: formulation sol}, $G$ and $H$ appear in the formula of the solution $\Hat{y}$. This is a direct consequence of the structure of KdV equations. We need all these special functions to be entire functions because we want to apply the Paley-Wiener Theorem. Later in Section \ref{sec: quantitative estimates}, we will present how to use these entire functions to bound the quantity $Q_M$.

\section{A trapping direction  }\label{sec: An unreachable direction}
As presented in Section \ref{sec: strategy of the proof}, this section is devoted to constructing a trapping direction $\Psi(t,x)$ for the KdV system \eqref{intro-sys-KdV} for $L\in\mathcal{N}$. This trapping direction $\Psi$ satisfies the following equations:
\begin{equation}\label{eq: defi-trapping direction}
\left\{
\begin{aligned}
\p_t\Psi (t,x) + \p_x^3\Psi (t,x) + \p_x\Psi (t,x) = 0, \mbox{ in } \R_+ \times (0, L),\\
\Psi(t, 0) = \Psi(t, L) = \p_x\Psi(t, 0) = \p_x\Psi(t, L) = 0, \mbox{ in } \R_+.   
\end{aligned}
\right.
\end{equation}
Furthermore, we prove a coercive property for this direction $\Psi$ as follows
\begin{proposition}\label{prop: coercive property} 
There exist $\Psi(t,x)$ and $T_*>0$ such that, for any  $u \in L^2(0, + \infty)$ with $u(t) = 0$ for $t > T_*$ and $y(t,\cdot) = 0$ for $t>T_*$ where $y$ is the unique solution of the linearized KdV system \eqref{eq: linearized controled kdv-0-0}, 
we have
\begin{equation}\label{eq: coercive inequality}
\int_{0}^\infty \int_0^{+\infty} |y|^2(t, x) \p_x\Psi(t, x) dx dt  \ge C \| u
\|_{H^{-1}(\R)}^2. 
\end{equation}    
\end{proposition} 
Here we used the following definition for the negative Sobolev norm of $u\in L^2(\R_+)$:
\begin{equation}\label{eq: defi-negative-norm}
\|u\|^2_{H^{s}(\R)}:=\int_{\R}|\Hat{u}(\tau)|^2(1+\tau^2)^s\diff\tau, \mbox{ for }s<0,
\end{equation}
where $\Hat{u}$ is the Fourier transform of the extension of $u$ by $0$ for $t<0$.
\subsection{Idea of the proof}\label{sec: idea of proof-trapping direction}
As previously discussed in Section \ref{sec: strategy of the proof}, our proof strategy combines three key elements: 
\begin{itemize}
    \item The reduction approach, detailed further in Section \ref{sec: reduction approach};  
    \item Applying our new classification of unreachable pairs $(k,l)$ (defined in Definition \ref{defi: types of unreachable pairs}), a key observation is that the remaining case under our new classification is degenerate. In Section \ref{sec: degenerate case}, we address the influence and difficulties of a degenerate case. 
    \item A higher order expansion scheme is introduced to establish a coercive estimate \eqref{eq: coercive inequality}. We refer to the details in Section \ref{sec: higher-order-expansion}. 
\end{itemize}
As an application of the key concepts discussed above, we outline a three-step proof structure in Section \ref{sec: outline of section}.

\subsubsection{Reduction approach}\label{sec: reduction approach} 
We need to point out that this approach characterizes all possible control $u$ which steers $0$ to $0$ at time $T$ for the linearized KdV system \eqref{eq: linearized controled kdv-0-0}, which is based on taking Fourier transform with respect to the time variable $t$ of the solution $y$ to \eqref{eq: linearized controled kdv-0-0}.

As is shown in the classical power series expansion method, we need to deal with a multiple of the  $L^2(0, L)$-projection of the solution $y(T, \cdot)$ into $M$, which is governed by the following quantity
\begin{equation}\label{eq: quantity-p}
Q_M(\varphi):=\int_{0}^{\infty} \int_0^L |y (t, x)|^2 e^{-\ii  pt }\varphi'(x) dx dt,  
\end{equation}
where $\varphi$ is a direction contained in $M$. This quantity is related to the quadratic order system in the power series expansion of the nonlinear KdV equation, which plays a central role in the controllability of the nonlinear KdV equations.  \\
Roughly speaking, if $Q_M=0$ for any $\varphi\in M$, we shall expect the small-time controllability holds true (see \cite{CC04} for example), while if on the contrary, $Q_M\neq0$ for some $\varphi_0\in M$, we expect this leads to a failure of the small-time controllability. Moreover, we expect to construct the obstruction to the small-time controllability based on this particular $\varphi_0$. 

As is observed in \cite{CKN}, once you see this $Q_M$ as a Fourier transform with respect to the time variable $t$, i.e., 
\[
Q_M=\int_0^L\mathcal{F}_{t\to p}(|y|^2(\cdot,x)\varphi'(x))\diff x.
\]
Combined with a direct computation, the description of both high-frequency and low-frequency parts will provide you with insights about the relationship between $Q_M$ and the regularity level of $u$, which is governed by a certain Sobolev norm of $u$. 

At last, we finish this reduction approach by introducing the following quantity $B$:
\begin{equation}\label{eq: def-B}
B(\tau, x) = \frac{\sum_{j=1}^3(e^{\lambda_{j+1} L } - e^{\lambda_{j} L })e^{\lambda_{j+2} x} }{\sum_{j=1}^3 (\lambda_{j+1} - \lambda_{j}) e^{-\lambda_{j+2} L }} \cdot  \frac{\sum_{j=1}^3(e^{\Tilde{\lambda}_{j+1} L } - e^{\Tilde{\lambda}_{j} L })e^{\Tilde{\lambda}_{j+2} x} }{\sum_{j=1}^3 (\Tilde{\lambda}_{j+1} - \Tilde{\lambda}_{j}) e^{-\Tilde{\lambda}_{j+2} L }}  \cdot \varphi'(x).
\end{equation}
By simple computation
\begin{align*}
\int_0^L \int_{0}^{\infty} |y(t, x)|^2 \varphi'(x) e^{ - i p t} \diff t \diff x = &
   \int_0^L
\varphi'(x)  \hat y * \widehat{ \bar{y}} (p, x) \diff x \\
= &   \int_0^L  \varphi'(x)\int_{\R}\hat y (\tau, x)  \overline{\hat y} (\tau -   p, x) \diff
\tau  \diff x.
\end{align*}
Using the formula \eqref{eq: defi-Fourier-y}
\[
\hat y(\tau, x) = \hat u(\tau) \frac{\sum_{j=1}^3(e^{\lambda_{j+1} L } - e^{\lambda_{j} L })e^{\lambda_{j+2} x} }{\sum_{j=1}^3 (\lambda_{j+1} - \lambda_{j}) e^{-\lambda_{j+2} L }},
\]
we obtain
\begin{equation*}
Q_M=\int_0^L \int_{0}^{\infty} |y(t, x)|^2 \varphi_x(x) e^{- i p t} \diff t \diff x  =
\int_{\R}  \Hat{u}(\tau) \overline{\Hat{u}(\tau - p )} \int_0^L B(\tau, x) \diff x \diff \tau.
\end{equation*}
Now the high-frequency analysis of $Q_M$ is reduced to the asymptotic analysis of $\int_0^LB(\tau,x)\diff x$ as $\tau\to\infty$. In summary, we conclude this reduction approach into three steps: (1)derive the asymptotic analysis of $\int_0^LB(\tau,x)\diff x$ as $\tau\to\infty$; (2)Based on this asymptotic expansion, prove a coercive property for $Q_M$; (3)Construct the trapping direction using (1) and (2).

\subsubsection{The remaining case under the new classification is degenerate.}\label{sec: degenerate case}
 Revisiting the papers on the positive result \cite{CC04} and the negative result \cite{CKN} of small-time controllability, we noticed that these two cases coincide with types $\mathcal{S}_1,\mathcal{S}_3$ and the remained degenerate case is of type $\mathcal{S}_2$. Another observation shows that $\mathcal{S}_2$ has similarities with both $\mathcal{S}_1$ and $\mathcal{S}_3$: the corresponding eigenvalues are pure imaginary but double. This new perspective motivates us to consider this degenerate case $\mathcal{S}_2$ has a negative result of small-time controllability but with a non-vanishing order much higher than the case $\mathcal{S}_1$. 
 
We point out that a crucial feature of $\mathcal{S}_2$ is the appearance of Type 2 eigenfunctions. Let $\varphi$ be a Type 1 eigenfucntion. Then, $\varphi'$ will become a Type 2 eigenfunction automatically if $2k+l\in3\N^*$.  
More precisely, we notice that $Q_M(\varphi)$ can be seen as a projection on this Type 2 eigenfunction $\varphi'$, which implies that we need to detect a non-vanishing term at higher orders than the non-degenerate case $2k+l\notin3\N^*$, i.e. $\mathcal{S}_3$. The technical difficulties will be specified in Section \ref{sec: higher-order-expansion}.

Even though the authors in \cite{CKN} have established this reduction approach, they cannot deal with all critical lengths. They ended up with cases where $L=2\pi\sqrt{\frac{k^2+kl+l^2}{3}}$ and $2k+l\notin 3\N^*$. For other situations where $2k+l\in 3\N^*$, their asymptotic expansion for $\int_0^LB(\tau,x)\diff x$, i.e. 
\[
\int_0^LB(\tau,x)\diff x=\frac{\Tilde{E}}{|\tau|^{\frac{4}{3}}}+\bigO(|\tau|^{-\frac{5}{3}})
\]
fails because $\Tilde{E}=0$ under the assumption $2k+l\in3\N^*$. Thus, we see $2k+l\in3\N^*$ as a degenerate case.

In Section \ref{sec: critical length}, we introduced a new perspective on classifying critical lengths by focusing on unreachable pairs $(k,l)$.  It's important to note that the nature of the unreachable space $M$ is heavily influenced by the characteristics of these pairs. A priori, it is essential to categorize the features of KdV systems into three distinct types, although there are instances where these types exhibit similar behaviors. Let's take an example, for $\int_0^LB(\tau,x)\diff x$,
\begin{equation*}
\int_0^LB(\tau,x)\diff x=\left\{
\begin{array}{ll}
    0,&(k,l)\in \mathcal{S}_1, \mbox{ in \cite{CC04}}  \\
    \frac{E}{|\tau|^{2}}+\bigO(|\tau|^{-\frac{7}{3}}), & (k,l)\in \mathcal{S}_2,\mbox{ see Section \ref{sec: Asymptotic analysis of B}},\\
   \frac{\Tilde{E}}{|\tau|^{\frac{4}{3}}}+\bigO(|\tau|^{-\frac{5}{3}}) & (k,l)\in \mathcal{S}_3, \mbox{ in \cite{CKN}}.
\end{array}
\right.
\end{equation*}
while for the dimension of $M$,
\begin{equation*}
\dim M\left\{
\begin{array}{ll}
    \mbox{is odd},&(k,l)\in \mathcal{S}_1, \\
 \mbox{is even},&(k,l)\in \mathcal{S}_2\cup\mathcal{S}_3.
\end{array}
\right.
\end{equation*}
\subsubsection{Higher-order expansion}\label{sec: higher-order-expansion}
As we explained in Section \ref{sec: reduction approach}, we concentrate on the analysis of the quantity $\int_0^LB(\tau,x)\diff x$ and $Q_M$. Precisely, we encounter two difficulties:
\begin{enumerate}
    \item When considering the asymptotic behaviors of $\int_0^LB(\tau,x)\diff x$, as we've previously discussed, due to the appearance of Type 2 eigenfunctions, the case $\mathcal{S}_2$ is degenerate and requires us to perform a more delicate analysis to detect the non-vanishing term at higher orders.
    \item As for the coercive property of $Q_M$, due to the higher non-vanishing order of $\int_0^LB(\tau,x)\diff x$, the regularity level is significantly lower than that in the case studied in \cite{CKN}. This analysis involves the non-local pseudodifferential operator $(1 + |D_t|^2)^{-\frac{1}{6}}$. Consequently, direct application of the Plancherel Theorem or direct use of compact support is insufficient. Instead, advanced techniques in microlocal analysis and specific lemmas concerning Sobolev norms on compactly supported functions are required.
\end{enumerate}

\subsubsection{Outline of the section}\label{sec: outline of section}
Armed with the primary ideas of proof discussed above, we organize our proof in the following three steps: 
\begin{itemize}
    \item \textbf{Step 1: }We first establish a new version of asymptotic expansion for $\int_0^LB(\tau,x)\diff x$ in Lemma \ref{lem: B asymp} of Section \ref{sec: Asymptotic analysis of B}:
    \[
     \int_0^L B(\tau, x) dx  =
\frac{E}{|\tau|^{2}}    + \bigO(|\tau|^{- \frac{7}{3}}) \mbox{ for $\tau\in \R$ with large $|\tau|$ and $E\neq0$.}
    \]
    \item \textbf{Step 2: }With the help of this asymptotic expansion, we derive a quantitative coercive estimate of $Q_M$ in Proposition \ref{prop: coercive-1/3norm} of Section \ref{sec: quantitative estimates}:
    \[
    \int_{0}^{\infty} \int_0^L |y (t, x)|^2 e^{-\ii pt }\varphi_x(x) dx dt\geq C\|u\|^2_{H^{-1}(\R)}.
    \]
    \item \textbf{Step 3: }Finally, equipped with good estimates on $Q_M$ and $B$, we construct the trapping direction $\Psi$ in the classical way, see details in Section \ref{sec: construction trap direction}.
\end{itemize}

\subsection{Asymptotic analysis of $\int_0^LB(\tau,x)\diff x$}\label{sec: Asymptotic analysis of B}
We study the high-frequency asymptotics in the Fourier side (Lemma \ref{lem: B asymp}) and then back to the physical space estimates (Proposition \ref{prop: coercive-1/3norm}). After establishing these estimates, we are able to obtain a ``well-prepared" direction $\Psi$ associated with its estimates, which is crucial in the proof of our main theorem.\\

Let  $u \in L^2(0,  + \infty)$ and denote $y$ the corresponding solution
of the linear KdV equation~\eqref{eq: linearized controled kdv-0-0}. We assume that $y_0=0$ and that $y$ satisfies $y(t, \cdot) = 0$ in $(0, L)$ for $t \geq T$. We have, by Lemma \ref{lem: formulation sol}, the formulation of $y$ as follows:

Let $ \lambda_j(\tau)$ for $j=1, \, 2, \, 3$  be the three solutions of
\begin{equation*}
\lambda^3 + \lambda = -\ii \tau \text{ for } \tau \in \C.
\end{equation*}
Let $\eta_1, \, \eta_2, \,  \eta_3 \in i \R$. Define
\begin{equation}\label{eq: defi varphi}
\varphi(x) = \sum_{j=1}^3 (\eta_{j+1} - \eta_{j}) e^{\eta_{j+2} x} \mbox{ for } x \in [0, L].
\end{equation}
 It is easy to check that
\begin{equation*}
\varphi(0) = \sum_{j=1}^3 (\eta_{j+1} - \eta_{j}) =0=\varphi(L),\; \varphi'(0)= \sum_{j=1}^3\eta_{j+2} (\eta_{j+1} - \eta_{j}) =0=\varphi'(L)
\end{equation*}
In addition, we further assume that 
\begin{equation}\label{eq: e^eta=1}
e^{\eta_1  L} = e^{\eta_2 L} = e^{\eta_3 L}=1,
\end{equation}
which is equivalent to $\eta_1, \eta_2, \eta_3 \in  \frac{2 \pi i}{L} \Z$. The definition of $\varphi$ in \eqref{eq: defi varphi} ensures that $\varphi$ satisfies the boundary condition of Type 1 eigenfunctions defined in \eqref{eq: exact formula for critical eigenfunctions} (See also \cite{Cerpa07, CC09}). In particular, if $\eta_j^3+\eta_j+i\lambda=0$, with $i\lambda$ is an eigenvalue defined in \eqref{eq: eigenvalue problem with critical length}, then $\varphi\in M$ is an uncontrollable direction. 

\medskip

We are ready to establish the behavior of
\[
\int_0^L B(\tau, x) dx
\]
for $\tau \in \R$ with $|\tau|\gg1$, which is one of the main ingredients for the analysis in this section.

\begin{lemma} \label{lem: B asymp}  Let   $p \in \R$,  and let $\varphi$ be defined by \eqref{eq: defi varphi}.  Assume that $\eta_j \neq 0$ and moreover, $e^{\eta_j L}=1$ and $\eta_j^3+\eta_j+ip=0$, for $j=1, \,  2, \,  3$.   
We have
\begin{equation*}
 \int_0^L B(\tau, x) dx  =
\frac{E}{|\tau|^{2}}    + \bigO(|\tau|^{- \frac{7}{3}}) \mbox{ for $\tau\in \R$ with large $|\tau|$},
\end{equation*}
where $E$ is  defined by
\begin{equation}\label{eq: defi-E}
E=\frac{1}{9}p^2L\sum_{j=1}^3 \frac{\eta_{j+1}-\eta_{j}}
{  \eta_{j+2} }
\end{equation}
\end{lemma}

\begin{proof} 
Before presenting our proof we point out that $\eta_j$ satisfies the characteristic equation $\eta_j^3+\eta_j+ip=0$. This condition is a crucial part when we analyze the dominating order of the asymptotic behavior of $\int_0^LB(\tau,x)\diff x$, to be compared with the analysis in \cite{CKN} where lower order expansion is considered and such a condition is not required. For more details, we refer to \eqref{eq: cancel-eta-eq}.

Without loss of generality, we are able to assume that $\tau$ is positive and large. Indeed, if $\tau$ is large and negative,  we define
\[
F(\tau, x) =  \frac{\sum_{j=1}^3(e^{\lambda_{j+1} L } - e^{\lambda_{j} L
})e^{\lambda_{j+2} x} }{\sum_{j=1}^3 (\lambda_{j+1} - \lambda_{j}) e^{-\lambda_{j+2} L
}}.
\]
Then
\[
B(\tau, x) = F(\tau, x) \overline{F (\tau-p, x)} \varphi_x(x).
\]
By the definition of $F$, it is easy to check that
$F(-\tau, x) = \overline{F(\tau, x)}$. Thus, 
\begin{equation}\label{eq: B-negative-z}
B(-\tau, x) =  \overline{ F(\tau, x)
\overline{F (\tau+ p, x)}  \, \overline{\varphi_x(x)}} . 
\end{equation}

If we obtain the asymptotic behaviors for $\tau$ positive and large, then by \eqref{eq: B-negative-z}, we obtain the corresponding behavior for $\tau$ negative and large.\\

Now let us concentrate on the asymptotics for $\tau$ positive and large. Using Lemma \ref{lem: lambda-asymp}  (same as \cite[Lemma 3.3]{CKN}), for the denominator of $B(\tau, x)$, we obtain the asymptotics as follows:
\begin{multline*}
\frac{1}{\sum_{j=1}^3 (\lambda_{j+1} - \lambda_{j}) e^{-\lambda_{j+2} L }} \cdot
\frac{1}{ \sum_{j=1}^3 (\Tilde{\lambda}_{j+1} - \Tilde{\lambda}_{j}) e^{-\Tilde{\lambda}_{j+2} L }}
= \frac{e^{\lambda_1 L } e^{\Tilde{\lambda}_1 L}}{(\lambda_3 - \lambda_2)  (\Tilde{\lambda}_3 -
\Tilde{\lambda}_2)  } \left( 1 + O \left(e^{-C |\tau|^{\frac{1}{3}}} \right) \right).
\end{multline*}
Then, we consider the numerator of $B(\tau, x)$. We define
\begin{equation*}
f(\tau, x) = \sum_{j=1}^3(e^{\lambda_{j+1} L } - e^{\lambda_{j} L })e^{\lambda_{j+2} x},
\quad \Tilde{f} (\tau, x) = \sum_{j=1}^3(e^{\Tilde{\lambda}_{j+1} L } - e^{\Tilde{\lambda}_{j} L
})e^{\Tilde{\lambda}_{j+2} x},
\end{equation*}
We notice the main parts of $f$ and $\Tilde{f}$ are the following terms:
\[
f_m(\tau, x) =  - e^{\lambda_3 L} e^{\lambda_2 x} + e^{\lambda_2 L}  e^{\lambda_3 x} +
e^{\lambda_3 L } e^{\lambda_1 x}, \quad
\Tilde{f}_m (\tau, x) =  - e^{\Tilde{\lambda}_3 L} e^{\Tilde{\lambda}_2 x} + e^{\Tilde{\lambda}_2 L}  e^{\Tilde{\lambda}_3 x}
+ e^{\Tilde{\lambda}_3 L } e^{\Tilde{\lambda}_1 x}.
\]
Therefore, for the product of $f$ and $\Tilde{f}$, we have
\begin{multline*}
 \int_0^L f(\tau, x) \Tilde{f}(\tau, x) \varphi_x(x) dx
 = \int_0^L f_m(\tau, x) \Tilde{f}_m(\tau, x) \varphi_x(x) dx + \int_0^L (f- f_m)(\tau, x) \Tilde{f}_m(\tau, x)
\varphi_x(x) dx \\ 
+ \int_0^L f_m(\tau, x) (\Tilde{f} -\Tilde{f}_m) (\tau, x) \varphi_x(x) dx +
\int_0^L (f- f_m)(\tau, x) (\Tilde{f} -\Tilde{f}_m) (\tau, x) \varphi_x(x) dx.
\end{multline*}
Using Lemma \ref{lem: lambda-asymp}, we have the following estimates.  
\begin{multline*}
 \int_0^L |(f- f_m)(\tau, x) \Tilde{f}_m(\tau, x) \varphi_x(x)| dx + \int_0^L |(f- f_m)(\tau, x) (\Tilde{f}
-\Tilde{f}_m) (\tau, x) \varphi_x(x)| dx  \\ + \int_0^L |f_m(\tau, x) (\Tilde{f} -\Tilde{f}_m) (\tau, x)
\varphi_x(x)| dx \le C |e^{(\lambda_3 + \Tilde{\lambda}_3) L }| e^{- C |\tau|^{\frac{1}{3}}}.
\end{multline*}
Then combining with the denominator of $B(\tau,x)$, these three terms are bounded by $e^{(\lambda_1 + \Tilde{\lambda}_1+\lambda_3 + \Tilde{\lambda}_3) L}e^{-C|\tau|^{\frac
{1}{3}}}$.
Due to Lemma \ref{lem: lambda-asymp}, we obtain
\[
\lambda_1 + \Tilde{\lambda}_1+\lambda_3 + \Tilde{\lambda}_3=\bigO(\tau^{-\frac{1}{3}}),
\]
which implies that these three terms are negligible. It suffices to estimate the main part $\int_0^L f_m(x, \tau) \Tilde{f}_m(x, \tau) \varphi_x(x)$. By the definition of $\varphi$, we obtain
\begin{equation*}
\int_0^L f_m(x, \tau)\Tilde{f}_m(x, \tau) \varphi_x(x) = \int_0^L f_m(x, \tau) \Tilde{f}_m(x, \tau) \left(
\sum_{j=1}^3 \eta_{j+2} (\eta_{j+1} - \eta_{j}) e^{\eta_{j+2} x} \right) dx.
\end{equation*}
We sort the terms of $\int_0^L f_m(x, \tau) \Tilde{f}_m(x, \tau) \varphi_x(x)$ into four groups.
\begin{itemize}
    \item The first part, by  \eqref{eq: e^eta=1} and  Lemma \ref{lem: lambda-asymp},
\begin{multline*}
\int_0^L \left(  - e^{\lambda_3 L} e^{\lambda_2 x}  e^{\Tilde{\lambda}_2 L}  e^{\Tilde{\lambda}_3 x}
- e^{\lambda_2 L}  e^{\lambda_3 x}e^{\Tilde{\lambda}_3 L} e^{\Tilde{\lambda}_2 x}
+ e^{\lambda_2 L}  e^{\lambda_3 x}e^{\Tilde{\lambda}_2 L}  e^{\Tilde{\lambda}_3 x} \right) \\
\times \left( \sum_{j=1}^3 \eta_{j+2} (\eta_{j+1} - \eta_{j}) e^{\eta_{j+2} x} \right)
dx
=  e^{ (\lambda_3 +   \Tilde{\lambda}_3 + \lambda_2 +  \Tilde{\lambda}_2) L } \left(Z_1(\tau) + O\left(e^{-C|\tau|^{\frac{1}{3}}} \right) \right),
\end{multline*}
where
\begin{equation}\label{eq: def-Z1}
Z_1(\tau) : =  \sum_{j=1}^3 \eta_{j+2}(\eta_{j+1}-\eta_{j}) \left( \frac{1}{  \lambda_3
+   \Tilde{\lambda}_3   + \eta_{j+2}} - \frac{1}{ \lambda_3  +   \Tilde{\lambda}_2   + \eta_{j+2}} -
\frac{1}{ \lambda_2  + \Tilde{\lambda}_3  + \eta_{j+2}}  \right).
\end{equation}
\item The second part is defined by 
\begin{multline*}
\int_0^L \left( e^{\lambda_3 L } e^{\lambda_1 x}e^{\Tilde{\lambda}_3 L } e^{\Tilde{\lambda}_1 x}
- e^{\lambda_3 L } e^{\lambda_1 x} e^{\Tilde{\lambda}_3 L} e^{\Tilde{\lambda}_2 x} - e^{\lambda_3 L}
e^{\lambda_2 x} e^{\Tilde{\lambda}_3 L } e^{\Tilde{\lambda}_1 x} \right) \\
\times \left( \sum_{j=1}^3 \eta_{j+2} (\eta_{j+1} - \eta_{j}) e^{\eta_{j+2} x} \right)
dx
=  e^{(\lambda_3 + \Tilde{\lambda}_3) L }\left( Z_2(\tau) + \bigO(e^{-C|\tau|^{\frac{1}{3}}}) \right),
\end{multline*}
where
\begin{equation}\label{eq: def-Z2}
Z_2 (\tau): = \sum_{j=1}^3 \eta_{j+2}(\eta_{j+1}-\eta_{j}) \left( - \frac{1}{\lambda_1
+  \Tilde{\lambda}_1  + \eta_{j+2}} + \frac{1}{\lambda_1   +  \Tilde{\lambda}_2   + \eta_{j+2}} +
\frac{1}{ \lambda_2  +   \Tilde{\lambda}_1   + \eta_{j+2}}\right).
\end{equation}
\item The third part is defined by 
\begin{equation*}
\int_0^L e^{\lambda_3 L} e^{\lambda_2 x} e^{\Tilde{\lambda}_3 L} e^{\Tilde{\lambda}_2 x}  \left(
\sum_{j=1}^3 \eta_{j+2} (\eta_{j+1} - \eta_{j}) e^{\eta_{j+2} x} \right) dx
=  e^{(\lambda_3 + \Tilde{\lambda}_3) L } Z_3(\tau),
\end{equation*}
where
\begin{equation}\label{eq: def-Z3}
Z_3 (\tau): = \left( e^{\lambda_2 L  + \Tilde{\lambda}_2 L } - 1\right) \sum_{j=1}^3
\frac{\eta_{j+2}(\eta_{j+1}-\eta_{j})  }{\lambda_2   + \Tilde{\lambda}_2   + \eta_{j+2}}.
\end{equation}
\item The last part is defined by 
\begin{equation*}
 \int_0^L \left(  e^{\lambda_3 L } e^{\lambda_1 x} e^{\Tilde{\lambda}_2 L} e^{\Tilde{\lambda}_3
x} + e^{\lambda_2 L} e^{\lambda_3 x} e^{\Tilde{\lambda}_3 L} e^{\Tilde{\lambda}_1 x}  \right) \left(
\sum_{j=1}^3 \eta_{j+2} (\eta_{j+1} - \eta_{j}) e^{\eta_{j+2} x} \right) dx.
\end{equation*}
\end{itemize}
We estimate term by term. For the fourth part, it is easy to show that it is negligible, because thanks to Lemma \ref{lem: lambda-asymp}, we have
\begin{multline*}
\left| \int_0^L \left(  e^{\lambda_3 L } e^{\lambda_1 x} e^{\Tilde{\lambda}_2 L} e^{\Tilde{\lambda}_3
x} + e^{\lambda_2 L} e^{\lambda_3 x} e^{\Tilde{\lambda}_3 L} e^{\Tilde{\lambda}_1 x}  \right) \left(
\sum_{j=1}^3 \eta_{j+2} (\eta_{j+1} - \eta_{j}) e^{\eta_{j+2} x} \right) dx \right|
= |e^{(\lambda_3 + \Tilde{\lambda}_3)L} | \bigO(e^{-C \tau^{\frac{1}{3}}}).
\end{multline*} 
Hence, we only pay attention to the first three parts. It suffices to derive the asymptotic behaviors of $Z_1$, $Z_2$, and $Z_3$. We begin with $Z_1$ (defined in \eqref{eq: def-Z1}). Since
$\sum_{j=1}^3 \eta_{j+2}(\eta_{j+1}-\eta_{j}) = 0$, we obtain:
\begin{align*}
Z_1 (\tau)=  & \sum_{j=1}^3 \eta_{j+2}(\eta_{j+1}-\eta_{j}) \left( \frac{1}{\lambda_3 +
\Tilde{\lambda}_3 + \eta_{j+2}} -  \frac{1}{\lambda_3 + \Tilde{\lambda}_3 } \right) \\
&  \;\;\;\;\;\;\;\;  \;\;\;\;\;\;\;\;  +
 \sum_{j=1}^3 \eta_{j+2}(\eta_{j+1}-\eta_{j})  \left(  - \frac{1}{\lambda_3 +
\Tilde{\lambda}_2 + \eta_{j+2}}
 +  \frac{1}{\lambda_3 + \Tilde{\lambda}_2}  \right) \\
& \;\;\;\;\;\;\;\; \;\;\;\;\;\;\;\; \;\;\;\;\;\;\;\; \;\;\;\;\;\;\;\;  + \sum_{j=1}^3 \eta_{j+2}(\eta_{j+1}-\eta_{j})  \left(- \frac{1}{\lambda_2 +
\Tilde{\lambda}_3 + \eta_{j+2}}  + \frac{1}{\lambda_2 + \Tilde{\lambda}_3}  \right).
\end{align*}
At infinity, $\lambda_3 +\Tilde{\lambda}_3$, $\lambda_3 +
\Tilde{\lambda}_2$, and $\lambda_2 +
\Tilde{\lambda}_3$ exhibit the same behavior.
Without loss of generality, we compute the asymptotic expansion for $\frac{1}{\lambda_3 +
\Tilde{\lambda}_3 + \eta_{j+2}} -  \frac{1}{\lambda_3 + \Tilde{\lambda}_3 }$. 

Using  Lemma \ref{lem: lambda-asymp}, we obtain 
\begin{align*}
\lambda_3 +\Tilde{\lambda}_3&=(\mu_3 +
\Tilde{\mu}_3)\tau^{\frac{1}{3}}-\left(\frac{1}{3\mu_3}+\frac{1}{3\Tilde{\mu}_3}\right)\tau^{-\frac{1}{3}}+\bigO(\tau^{-\frac{2}{3}})\\
&=(\mu_3 +
\Tilde{\mu}_3)\tau^{\frac{1}{3}}\left(1-\frac{1}{3\mu_3\Tilde{\mu}_3}\tau^{-\frac{2}{3}}\right)+\bigO(\tau^{-\frac{2}{3}})
\end{align*}
Therefore, we obtain the asymptotic expansion:
\begin{align*}
\frac{1}{\lambda_3 +
\Tilde{\lambda}_3 + \eta_{j+2}} -  \frac{1}{\lambda_3 + \Tilde{\lambda}_3 }
=&-\frac{\eta_{j+2}}{(\lambda_3 +
\Tilde{\lambda}_3 + \eta_{j+2})(\lambda_3 + \Tilde{\lambda}_3)}\\
=&-\frac{\eta_{j+2}}{(\lambda_3 +
\Tilde{\lambda}_3)^2}\frac{1}{(1 + \frac{\eta_{j+2}}{\lambda_3 +
\Tilde{\lambda}_3})}\\
=&-\frac{\eta_{j+2}}{(\lambda_3 +
\Tilde{\lambda}_3)^2}(1 - \frac{\eta_{j+2}}{\lambda_3 +
\Tilde{\lambda}_3}+\frac{\eta_{j+2}^2}{(\lambda_3 +
\Tilde{\lambda}_3)^2})+\bigO((\lambda_3 +
\Tilde{\lambda}_3)^{-5})
\end{align*}
Here we notice that the leading order of $\lambda_3 +
\Tilde{\lambda}_3$ is $\bigO(\tau^{\frac{1}{3}})$. It is obvious to see that $\bigO((\lambda_3 +
\Tilde{\lambda}_3)^{-5})=\bigO(\tau^{-\frac{5}{3}})$. So we deduce that
\[
\frac{1}{\lambda_3 +
\Tilde{\lambda}_3 + \eta_{j+2}} -  \frac{1}{\lambda_3 + \Tilde{\lambda}_3 }=I_1\times I_2+\bigO(\tau^{-\frac{5}{3}}),
\]
where
\begin{gather*}
I_1:=\frac{\eta_{j+2}}{(\lambda_3 +
\Tilde{\lambda}_3)^2}=\bigO(\tau^{-\frac{2}{3}}),\\
I_2:=1 - \frac{\eta_{j+2}}{\lambda_3 +
\Tilde{\lambda}_3}+\frac{\eta_{j+2}^2}{(\lambda_3 +
\Tilde{\lambda}_3)^2}=\bigO(1).
\end{gather*}
Thus, it suffices to expand $I_1$ until the order $\bigO(\tau^{-\frac{5}{3}})$ and $I_2$ until the order $\bigO(\tau^{-1})$. More precisely,  we plug $\lambda_3 +\Tilde{\lambda}_3=(\mu_3 +
\Tilde{\mu}_3)\tau^{\frac{1}{3}}\left(1-\frac{1}{3\mu_3\Tilde{\mu}_3}\tau^{-\frac{2}{3}}\right)+\bigO(\tau^{-\frac{2}{3}})$ into the above expressions:
\begin{multline*}
I_1=-\frac{\eta_{j+2}}{\left((\mu_3 +
\Tilde{\mu}_3)\tau^{\frac{1}{3}}(1-\frac{1}{3\mu_3\Tilde{\mu}_3}\tau^{-\frac{2}{3}}+\bigO(\tau^{-1}))\right)^2}\\
=-\frac{\eta_{j+2}}{(\mu_3 +
\Tilde{\mu}_3)^2\tau^{\frac{2}{3}}}\frac{1}{\left(1-\frac{1}{3\mu_3\Tilde{\mu}_3}\tau^{-\frac{2}{3}}+\bigO(\tau^{-1})\right)^2}\\
=-\frac{\eta_{j+2}}{(\mu_3 +
\Tilde{\mu}_3)^2\tau^{\frac{2}{3}}}(1+\frac{2}{3\mu_3\Tilde{\mu}_3}\tau^{-\frac{2}{3}}+\bigO(\tau^{-1})).
\end{multline*}
For $I_2$, we also use Lemma \ref{lem: lambda-asymp} and obtain
\begin{align*}
I_2=1 - \frac{\eta_{j+2}}{(\mu_3 +
\Tilde{\mu}_3)\tau^{\frac{1}{3}}(1-\frac{1}{3\mu_3\Tilde{\mu}_3}\tau^{-\frac{2}{3}}+\bigO(\tau^{-1}))}+\frac{\eta_{j+2}^2}{\left((\mu_3 +
\Tilde{\mu}_3)\tau^{\frac{1}{3}}(1-\frac{1}{3\mu_3\Tilde{\mu}_3}\tau^{-\frac{2}{3}}+\bigO(\tau^{-1}))\right)^2}\\
=1 - \frac{\eta_{j+2}((1+\frac{1}{3\mu_3\Tilde{\mu}_3}\tau^{-\frac{2}{3}}+\bigO(\tau^{-1}))}{(\mu_3 +
\Tilde{\mu}_3)\tau^{\frac{1}{3}}}+\frac{\eta_{j+2}^2\left(1+\frac{2}{3\mu_3\Tilde{\mu}_3}\tau^{-\frac{2}{3}}+\bigO(\tau^{-1})\right)}{(\mu_3 +
\Tilde{\mu}_3)^2\tau^{\frac{2}{3}}}.
\end{align*}
In conclusion, we obtain
\[
I_1=\frac{\eta_{j+2}(1+2\frac{1}{3\mu_3\Tilde{\mu}_3}\tau^{-\frac{2}{3}})}{(\mu_3 +
\Tilde{\mu}_3)^2\tau^{\frac{2}{3}}}+\bigO(\tau^{-\frac{5}{3}})
=-\left(\frac{\eta_{j+2}\tau^{-\frac{2}{3}}}{(\mu_3 +
\Tilde{\mu}_3)^2}+\frac{2\eta_{j+2}\tau^{-\frac{4}{3}}}{3\mu_3\Tilde{\mu}_3(\mu_3 +
\Tilde{\mu}_3)^2}\right)+\bigO(\tau^{-\frac{5}{3}}),
\]
\begin{multline*}
I_2
=1 - \frac{\eta_{j+2}(1+\frac{1}{3\mu_3\Tilde{\mu}_3}\tau^{-\frac{2}{3}})}{(\mu_3+\Tilde{\mu}_3)\tau^{\frac{1}{3}}}+\frac{\eta_{j+2}^2(1+2\frac{1}{3\mu_3\Tilde{\mu}_3}\tau^{-\frac{2}{3}})}{(\mu_3 +
\Tilde{\mu}_3)^2\tau^{\frac{2}{3}}}+\bigO(\tau^{-\frac{4}{3}})
=1 - \frac{\eta_{j+2}\tau^{-\frac{1}{3}}}{(\mu_3+\Tilde{\mu}_3)}+\frac{\eta_{j+2}^2\tau^{-\frac{2}{3}}}{(\mu_3 +
\Tilde{\mu}_3)^2}+\bigO(\tau^{-1}).
\end{multline*}
As a consequence, we deduce that
\begin{align*}
&\frac{1}{\lambda_3 +
\Tilde{\lambda}_3 + \eta_{j+2}} -  \frac{1}{\lambda_3 + \Tilde{\lambda}_3 }\\
=&-\left(\frac{\eta_{j+2}}{(\mu_3 +
\Tilde{\mu}_3)^2}\tau^{-\frac{2}{3}}+\frac{2\eta_{j+2}}{3\mu_3\Tilde{\mu}_3(\mu_3 +
\Tilde{\mu}_3)^2}\tau^{-\frac{4}{3}}\right)
\left(1 - \frac{\eta_{j+2}}{(\mu_3+\Tilde{\mu}_3)}\tau^{-\frac{1}{3}}+\frac{\eta_{j+2}^2}{(\mu_3 +
\Tilde{\mu}_3)^2}\tau^{-\frac{2}{3}}\right)+\bigO(\tau^{-\frac{5}{3}})\\
=&-\frac{\eta_{j+2}}{(\mu_3 +
\Tilde{\mu}_3)^2}\tau^{-\frac{2}{3}}+\frac{\eta_{j+2}^2}{(\mu_3 +
\Tilde{\mu}_3)^3}\tau^{-1}-\left(\frac{2\eta_{j+2}}{3\mu_3\Tilde{\mu}_3(\mu_3 +
\Tilde{\mu}_3)^2}+\frac{\eta_{j+2}^3}{(\mu_3 +
\Tilde{\mu}_3)^4}\right)\tau^{-\frac{4}{3}}+\bigO(\tau^{-\frac{5}{3}}).
\end{align*}
Therefore, we obtain
\[
Z_1(\tau)=\sum_{j=1}^3 \eta_{j+2}^2(\eta_{j+1} -\eta_{j})\left(C_{11}\tau^{-\frac{2}{3}}+C_{12}\eta_{j+2}\tau^{-1}+(C_{13}+C_{14}\eta_{j+2}^2)\tau^{-\frac{4}{3}}\right)+\bigO(\tau^{-\frac{5}{3}}),
\]
where the coefficients $C_{11}$, $C_{12}$, $C_{13}$ and $C_{14}$ are defined via: 
\begin{align*}
C_{11}:=& -\frac{1}{(\mu_3 + \Tilde{\mu}_3)^2} + \frac{1}{(\mu_3 + \Tilde{\mu}_2)^2} +
\frac{1}{(\mu_2 + \Tilde{\mu}_3)^2}
=  -\frac13 + \frac{-1+i\sqrt 3}6 + \frac{-1-i\sqrt
3}6= -\frac{2}{3}\\
 C_{12}:=&\frac{1}{(\mu_3 + \Tilde{\mu}_3)^3} - \frac{1}{(\mu_3 + \Tilde{\mu}_2)^3} -
\frac{1}{(\mu_2 + \Tilde{\mu}_3)^3}
=  \frac{1}{3\sqrt{3}} +\frac{1}{3\sqrt{3}}+\frac{1}{3\sqrt{3}}= \frac{1}{\sqrt{3}}\\
 C_{13}:=&-\frac{2}{3\mu_3\Tilde{\mu}_3(\mu_3 + \Tilde{\mu}_3)^2} + \frac{2}{3\mu_3\Tilde{\mu}_2(\mu_3 + \Tilde{\mu}_2)^2} +
\frac{2}{3\mu_2\Tilde{\mu}_3(\mu_2 + \Tilde{\mu}_3)^2}
=  -\frac{2}{9} + \frac{2}{9} + \frac{2}{9}= \frac{2}{9}\\
 C_{14}:=&-\frac{1}{(\mu_3 + \Tilde{\mu}_3)^4}+ \frac{1}{(\mu_3 + \Tilde{\mu}_2)^4}+
\frac{1}{(\mu_2 + \Tilde{\mu}_3)^4}
=  -\frac{1}{9} + \frac{-1-i\sqrt 3}{18} + \frac{-1+i\sqrt
3}{18}= -\frac{2}{9}
\end{align*}
We derive that
\begin{equation}\label{eq: Z1-expression}
Z_1(\tau)  = - \frac{2}{3} \tau^{-\frac{2}{3}}  \sum_{j=1}^3 \eta_{j+2}^2(\eta_{j+1} -\eta_{j}) +  \frac{1}{\sqrt{3}}\tau^{-1}  \sum_{j=1}^3 \eta_{j+2}^3(\eta_{j+1} -\eta_{j})-  \frac{2}{9}\tau^{-\frac{4}{3}}  \sum_{j=1}^3 \eta_{j+2}^2(\eta_{j+2}^2-1)(\eta_{j+1} -\eta_{j})+
\bigO(\tau^{-\frac{5}{3}}).
\end{equation}
For $Z_2$ (defined in \eqref{eq: def-Z2}), thanks to $\sum_{j=1}^3 \eta_{j+2}(\eta_{j+1}-\eta_{j}) = 0$, we obtain
\begin{align*}
Z_2 (\tau)=  & \sum_{j=1}^3 \eta_{j+2}(\eta_{j+1}-\eta_{j}) \left( -  \frac{1}{\lambda_1
+ \Tilde{\lambda}_1 + \eta_{j+2}} +  \frac{1}{\lambda_1 + \Tilde{\lambda}_1 } \right) \\
& \;\;\;\;\;\;\;\; \;\;\;\;\;\;\;\;  +
 \sum_{j=1}^3 \eta_{j+2}(\eta_{j+1}-\eta_{j})  \left(   \frac{1}{\lambda_1 +
\Tilde{\lambda}_2 + \eta_{j+2}} -
   \frac{1}{\lambda_1 + \Tilde{\lambda}_2}  \right) \\
&  \;\;\;\;\;\;\;\; \;\;\;\;\;\;\;\; \;\;\;\;\;\;\;\; \;\;\;\;\;\;\;\; + \sum_{j=1}^3 \eta_{j+2}(\eta_{j+1}-\eta_{j})  \left( \frac{1}{\lambda_2 +
\Tilde{\lambda}_1 + \eta_{j+2}}  - \frac{1}{\lambda_2 + \Tilde{\lambda}_1}  \right).
\end{align*}
Following the same procedure, we derive that
\[
Z_2(\tau)=\sum_{j=1}^3 \eta_{j+2}^2(\eta_{j+1} -\eta_{j})\left(C_{21}\tau^{-\frac{2}{3}}+C_{22}\eta_{j+2}\tau^{-1}+(C_{23}+C_{24}\eta_{j+2}^2)\tau^{-\frac{4}{3}}\right)+\bigO(\tau^{-\frac{5}{3}}),
\]
where the coefficients $C_{21}$, $C_{22}$, $C_{23}$ and $C_{24}$ are defined via:
\begin{align*}
 C_{21}:=&\frac{1}{(\mu_1 + \Tilde{\mu}_1)^2} - \frac{1}{(\mu_1 + \Tilde{\mu}_2)^2} -
\frac{1}{(\mu_2 + \Tilde{\mu}_1)^2}
=  -\frac13 + \frac{-1+i\sqrt 3}6 + \frac{-1-i\sqrt
3}6= -\frac{2}{3}\\
C_{22}:=& -\frac{1}{(\mu_1 + \Tilde{\mu}_1)^3} + \frac{1}{(\mu_1 + \Tilde{\mu}_2)^3} +
\frac{1}{(\mu_2 + \Tilde{\mu}_1)^3}
=  \frac{1}{3\sqrt{3}} +\frac{1}{3\sqrt{3}}+\frac{1}{3\sqrt{3}}= \frac{1}{\sqrt{3}}
\end{align*}
\begin{align*}
 C_{23}:=& \frac{2}{3\mu_1\Tilde{\mu}_1(\mu_1 + \Tilde{\mu}_1)^2} - \frac{2}{3\mu_1\Tilde{\mu}_2(\mu_1 + \Tilde{\mu}_2)^2} -
\frac{2}{3\mu_2\Tilde{\mu}_1(\mu_2 + \Tilde{\mu}_1)^2}
=  -\frac{2}{9} + \frac{2}{9} + \frac{2}{9}= \frac{2}{9}\\
 C_{24}:=&\frac{1}{(\mu_1 + \Tilde{\mu}_1)^4} - \frac{1}{(\mu_1 + \Tilde{\mu}_2)^4} -
\frac{1}{(\mu_2 + \Tilde{\mu}_1)^4}
=  \frac{1}{9} - \frac{-1-i\sqrt 3}{18} - \frac{-1+i\sqrt
3}{18}= -\frac{2}{9}
\end{align*}
Therefore, we obtain
\textcolor{black}{
\begin{equation}\label{eq: Z2-expression}
Z_2(\tau)  =  \frac{2}{3} \tau^{-\frac{2}{3}}  \sum_{j=1}^3 \eta_{j+2}^2(\eta_{j+1} -\eta_{j}) +  \frac{1}{\sqrt{3}}\tau^{-1}  \sum_{j=1}^3 \eta_{j+2}^3(\eta_{j+1} -\eta_{j})+  \frac{2}{9}\tau^{-\frac{4}{3}}  \sum_{j=1}^3 \eta_{j+2}^2(\eta_{j+2}^2-1)(\eta_{j+1} -\eta_{j})+
\bigO(\tau^{-\frac{5}{3}}).
\end{equation}
}
We finally  consider  $Z_3(\tau)$ given in \eqref{eq: def-Z3}. We have, by \eqref{eq: lambda-eigenvalue-tau},
\[
\lambda_2^3 + \Tilde{\lambda}_2^3 + \lambda_2 + \Tilde{\lambda}_2 = - i \tau + i (\tau -p) =
- i p.
\]
This yields
\[
\lambda_2 + \Tilde{\lambda}_2 = - \frac{i p}{\lambda_2^2 + \Tilde{\lambda}_2^2 - \lambda_2
\Tilde{\lambda}_2+1}
\]
From  Lemma \ref{lem: lambda-asymp}, we have
\[
\lambda_2 + \Tilde{\lambda}_2 = -\frac{1}{3}\ii p  \tau^{-\frac{2}{3}}  + \bigO(\tau^{-1})
\]
It follows that
\begin{align}\label{eq: Z3-2}
\sum_{j=1}^3 \frac{\eta_{j+2}(\eta_{j+1}-\eta_{j})  }{\lambda_2 + \Tilde{\lambda}_2  +
\eta_{j+2}}  &=  \sum_{j=1}^3 \frac{\eta_{j+2}(\eta_{j+1}-\eta_{j})  }{-\frac{1}{3}\ii p  \tau^{-\frac{2}{3}}  + \eta_{j+2}}  + \bigO(\tau^{-1}) \nonumber \\
&=   \sum_{j=1}^3 (\eta_{j+1}-\eta_{j})  \left( 1+ \frac{1}{3} \frac{\ii p \tau^{-\frac{2}{3}}}{ \eta_{j+2} }
\right) + \bigO(\tau^{-1})\nonumber\\
&=   \frac{1}{3}\ii p   \sum_{j=1}^3 \frac{  \eta_{j+1}-\eta_{j}} {
\eta_{j+2} } \tau^{-\frac{2}{3}} +  \bigO(\tau^{-1}).
\end{align}
and 
\begin{gather*}
e^{(\lambda_2+\Tilde{\lambda}_2)L}-1=e^{ (-\frac{1}{3}\ii p  \tau^{-\frac{2}{3}} + \bigO(\tau^{-1}))L}-1
=-\frac{1}{3}\ii pL  \tau^{-\frac{2}{3}}  + \bigO(\tau^{-1}).
\end{gather*}
We derive from  \eqref{eq: Z3-2} and  Lemma \ref{lem: lambda-asymp} that
\begin{equation}\label{eq: Z3-expression}
Z_3(\tau) =  \frac{1}{3}\ii p    \left( -\frac{1}{3}\ii pL  \tau^{-\frac{2}{3}}  + \bigO(\tau^{-1})\right) \sum_{j=1}^3 \frac{\eta_{j+1}-\eta_{j}}
{  \eta_{j+2} } \tau^{-\frac{2}{3}}   + O (\tau^{-1})=\frac{1}{9}p^2L\tau^{-\frac{4}{3}}\sum_{j=1}^3 \frac{\eta_{j+1}-\eta_{j}}
{\eta_{j+2}}    + O (\tau^{-\frac{5}{3}}).
\end{equation}
Moreover, thanks to $\eta_j^3+\eta_j+ip=0$, we know that 
\begin{equation}\label{eq: cancel-eta-eq}
\sum_{j=1}^3\eta_{j+2}^3(\eta_{j+1}-\eta_j)=0,\sum_{j=1}^3(\eta_{j+1}-\eta_j)\eta_{j+2}^2=-ip\sum_{j=1}^3\frac{\eta_{j+1}-\eta_j}{\eta_{j+2}}.
\end{equation}

Therefore, $Z_1(\tau)+Z_2(\tau)=\bigO(\tau^{-\frac{5}{3}})$. As a consequence, we derive that
\begin{equation*}
e^{(\lambda_2+\Tilde{\lambda}_2)L}Z_1(\tau) + Z_2(\tau) + Z_3(\tau) = \bigO(\tau^{-\frac{4}{3}}) \mbox{ for large positive $\tau$}.
\end{equation*}
Using  Lemma \ref{lem: lambda-asymp}, we have
\begin{equation*}
(\lambda_3 - \lambda_2)(\Tilde{\lambda}_3 - \Tilde{\lambda}_2) = 3 \tau^{\frac{2}{3}} ( 1 + \bigO(\tau^{-\frac{1}{3}}) ).
\end{equation*}
Then, combining \eqref{eq: Z1-expression}, \eqref{eq: Z2-expression}, and \eqref{eq: Z3-expression}, we obtain
\begin{equation*}
\int_{0}^L B(\tau, x) \diff \tau  \\= \frac{1}{3 \tau^{\frac{2}{3}}} \left(e^{(\lambda_2+\Tilde{\lambda}_2)L} Z_1(\tau) +
Z_2 (\tau) + Z_3 (\tau) + \bigO(\tau^{-\frac{5}{3}}) \right)=E \tau^{-2}   + \bigO(\tau^{-\frac{7}{3}}),
\end{equation*}
which is the conclusion for large positive $\tau$.
\medskip
\end{proof}

By the definition of $B(\tau,x)$, we obtain
\begin{lemma}
Let   $p$, $\varphi$, and $\eta_j$ be the same as in Lemma \ref{lem: B asymp}.     Let $u \in L^2(0, + \infty)$ and let $y \in C([0, + \infty); L^2(0, L))
\cap L^2_{\textrm{loc}}\left([0, +\infty); H^1(0, L) \right)$ be the unique solution of
\eqref{eq: linearized controled kdv-0-0}. 
We have
\begin{equation*}
\int_{0}^{+\infty} \int_0^L  |y(t, x)|^2 \varphi_x(x) e^{- i p t} dx dt  =
\int_{\R}  \Hat{u}(\tau) \overline{\Hat{u}(\tau - p )} \left(\frac{E}{|\tau|^{2}}  + \bigO(|\tau|^{-\frac{7}{3}})
\right) \diff \tau.
\end{equation*}
\end{lemma}
\subsection{Quantitative estimates of $Q_M$}\label{sec: quantitative estimates}

In the following proposition, we will transfer the frequency asymptotic property into an estimate with certain Sobolev norms, which is the key
ingredient for the analysis of the obstruction of 
null-controllability for small time of the KdV system \eqref{intro-sys-KdV}.

\begin{proposition} \label{prop: coercive-1/3norm}
Let   $p$, $\varphi$, and $\eta_j$ be the same as in Lemma \ref{lem: B asymp}.    Let $u \in L^2(0, + \infty)$ with $u \not \equiv 0$,  $u(t) = 0$ for $t
> T$. Let $y \in C([0, + \infty);
L^2(0, L)) \cap L^2_{\textrm{loc}}\left([0, +\infty); H^1(0, L) \right)$ be the unique solution of \eqref{eq: linearized controled kdv-0-0} with $y(t, \cdot) = 0$ for $t>T$. Then, there exists a real number $N(u) \geq 0$  such that $C^{-1}\|u\|_{H^{-1}(\R)} \leq N(u) \leq C\|u\|_{H^{-1}(\R)}$ for some  constant $C=C(L) \geq 1$,  and 
\begin{equation}\label{eq: coercive-est-y}
\int_{0}^{\infty} \int_0^L |y (t, x)|^2 e^{-\ii pt }\varphi_x(x) dx dt =
N(u)^2\left(E + \bigO(1) T^{\frac{1}{100}}\right). 
\end{equation}
\end{proposition}
Before we present our proof of Proposition \ref{prop: coercive-1/3norm}, we first recall the definition of $\mathcal{H}$ in Lemma \ref{lem: defi-G-H}. Since $\mathcal{H}$ is a non-constant entire function, we know that there exists $\gamma > 0$ such that
\begin{equation}\label{eq: no-zero-H'}
\mathcal{H}'(\tau + i \gamma) \neq 0 \mbox{ for all } \tau \in \R.
\end{equation}
Fix such a $\gamma$ and denote 
\begin{equation}\label{eq: defi-H-gamma}
\mathcal{H}_\gamma (\tau) = \mathcal{H}(\tau + i \gamma), \tau\in\C.
\end{equation}
Thus, we have the following asymptotic behaviors for $\mathcal{H}_\gamma$ and $\mathcal{H}$, whose proof can be found in \cite[Page 1216]{CKN}
\begin{lemma}\label{lem: asymp-H}
$\mathcal{H}$ and $\mathcal{H}_\gamma'$ have the following asymptotic behaviors:
\begin{enumerate}
    \item $\mathcal{H}(\tau)=\kappa 
\tau^{-\frac{2}{3}- k}e^{ - \mu_1 L \tau^{\frac{1}{3}}}\left(1+\bigO(\tau^{-\frac{1}{3}})\right)$, $\mathcal{H}'_\gamma (\tau)=-  \frac{\mu_1 L }{3} \tau^{-\frac{2}{3}}  \kappa 
\tau^{-\frac{2}{3}- k}e^{ - \mu_1 L \tau^{\frac{1}{3}}}\left(1+\bigO(\tau^{-\frac{1}{3}})\right)$, where $\kappa = -  \frac{1}{ (\mu_2 - \mu_1) (\mu_1 - \mu_3)}$.
    \item $\lim_{\tau \in \R, \tau \to + \infty}\frac{\mathcal{H}(\tau) |\tau|^{-\frac{2}{3}}}{\mathcal{H}_\gamma'(\tau)}   = \alpha: = \frac{3 e^{ - \ii \frac{\pi}{6}}}{L}$, $\lim_{\tau \in \R, \tau \to - \infty}\frac{\mathcal{H}(\tau) |\tau|^{-\frac{2}{3}}}{\mathcal{H}_\gamma'(\tau)}   = -\overline{\alpha}$.
    \item In addition, we have
\begin{align*}
 \big |\mathcal{H}(\tau) |\tau|^{-\frac{2}{3}} -   \alpha  \mathcal{H}_\gamma'(\tau) \big|  &\leq C |\mathcal{H}_\gamma'(\tau)| |\tau|^{-\frac{1}{3}} \mbox{ for large positive $\tau$},\\
 \big |\mathcal{H}(\tau) |\tau|^{-\frac{2}{3}} -   \overline{\alpha}  \mathcal{H}_\gamma'(\tau) \big|  &\leq C |\mathcal{H}_\gamma'(\tau)| |\tau|^{-\frac{1}{3}} \mbox{ for large negative $\tau$}.
\end{align*}
\end{enumerate}
\end{lemma}
Now let us define
\begin{equation}\label{eq: defi-w}
\Hat{w}(\tau):=\frac{\Hat{u}(\tau)\mathcal{H}_\gamma'(\tau)}{\mathcal{H}(\tau)}.
\end{equation}
\begin{lemma}\label{lem: compact-supp-w}
$\Hat{w}$ is an entire function and satisfies Paley–Wiener’s conditions in $(-T-\varepsilon,T+\varepsilon)$, $\forall\varepsilon>0$. Therefore, $\Hat{w}$ is the Fourier transform of a $L^2$ function $w$, with $\supp{w}\subset[-T,T]$.
\end{lemma}
The properties and the construction of $w$ follow the same approach as Coron-Koenig-Nguyen. However, the main difficulty for us is to obtain a quantitative estimate associated with our asymptotic expansion of 
\[
 \int_0^L B(\tau, x) dx  =
\frac{E}{|\tau|^{2}}    + \bigO(|\tau|^{- \frac{7}{3}}) \mbox{ for $\tau\in \R$ with large $|\tau|$}.
\]
We need to emphasize that we are now in a degenerate situation, which implies we need to deal with a higher-order non-vanishing term $\sim\frac{1}{\tau^2}$. This in turn gives us a lower-order Sobolev norm $\sim\|u\|^2_{H^{-1}(\R)}$, which requires a more delicate analysis to derive our quantitative estimates. \\
Now we are in a position to prove Proposition \ref{prop: coercive-1/3norm}.
\begin{proof}[Proof of Proposition \ref{prop: coercive-1/3norm}]
Thanks to Lemma \ref{lem: asymp-H} and the definition of $\Hat{w}$, we have
\begin{equation}\label{eq: bound-hu}
|\Hat{u}(\tau)|=\left|\frac{\Hat{w}(\tau)\mathcal{H}(\tau)}{\mathcal{H}'_{\gamma}(\tau)}\right|\leq C|\Hat{w}(\tau)||\tau|^{\frac{2}{3}}.
\end{equation}
In this proof, to simplify the notation, we will use $C$ to denote a strictly positive constant that is independent of $T$. This $C$ may change from line to line.

By Lemma \ref{lem: B asymp}, we know that $\int_0^LB(\tau,x)\diff x$ is uniformly bounded by $\frac{C}{1+|\tau|^2}$ for $\forall\tau\in\R$ (we refer to \cite[Page 1216, (3.47)]{CKN} for more details). Therefore, we derive that
\begin{equation*}
\left|\Hat{u} (\tau) \overline{\Hat{u} (\tau - p)} \int_0^L B(\tau, x) \diff x \right|  \le  \frac{C|\tau|^{\frac{2}{3}}|\tau-p|^{\frac{2}{3}}}{1+|\tau|^2} |\Hat{w}(\tau)|
|\Hat{w} (\tau-p)| \mbox{ for } \tau \in \R.
\end{equation*}
As a consequence of Peetre's inequality,
\[
|\tau|^{\frac{2}{3}}|\tau-p|^{\frac{2}{3}}\leq (1+|\tau|^2)^{\frac{1}{3}}(1+|\tau-p|^2)^{\frac{1}{3}}\leq C(1+|\tau|^2)^{\frac{2}{3}}(1+|p|^2)^{\frac{1}{3}}, 
\]
which implies that 
\begin{equation}\label{eq: bound-integral-B-u}
\left|\Hat{u} (\tau) \overline{\Hat{u} (\tau - p)} \int_0^L B(\tau, x) \diff x \right|  \le  \frac{C}{(1+|\tau|^2)^{\frac{1}{3}}}|\Hat{w}(\tau)|
|\Hat{w} (\tau-p)| \mbox{ for } \tau \in \R.
\end{equation}
We claim that there exists $\delta_0>0$ such that
\[
\left|\int_{\R} \Hat{u} (\tau) \overline{\Hat{u} (\tau - p)} \int_0^L B(\tau, x) \diff x \diff \tau  - E
|\alpha|^2
\int_{\R} (1+|\tau|^2)^{-\frac{1}{6}}(1+|\tau-p|^2)^{-\frac{1}{6}} \Hat{w} (\tau) \overline{ \Hat{w}(\tau-p)} \diff \tau  \right|\leq CT^{\delta_0}\|w\|^2_{H^{-\frac{1}{3}}(\R)}.
\]
\\
\vspace{3mm}
Note that, for $m \ge 1$ to be specified later,
\begin{multline*}
\left|\int_{|\tau| > m} \Hat{u} (\tau) \overline{\Hat{u} (\tau - p)} \int_0^L B(\tau, x) \diff x \diff \tau - E
|\alpha|^2
\int_{|\tau| > m} (1+|\tau|^2)^{-\frac{1}{6}}(1+|\tau-p|^2)^{-\frac{1}{6}} \Hat{w} (\tau) \overline{\Hat{w}(\tau-p)} \diff \tau \right| \\
\le \int_{|\tau| > m} \left| \Hat{u} (\tau) \overline{\Hat{u} (\tau - p)} \left( \int_0^L B(\tau, x) \diff x
- E |\tau|^{-2} \right)   \right| \diff \tau \\
+ |E|  \int_{|\tau| > m} |\tau|^{-\frac{2}{3}}\left|(1+|\tau|^2)^{-\frac{1}{6}}(1+|\tau-p|^2)^{-\frac{1}{6}} |\tau|^{\frac{2}{3}}|\alpha|^2 \Hat{w} (\tau)  \overline{\Hat{w} (\tau - p)} -  |\tau|^{-\frac{4}{3}} \Hat{u}
(\tau)
\overline{\Hat{u} (\tau - p)} \right| d\tau.
\end{multline*}
On the one hand, for the first part of the right-hand side, by Lemma \ref{lem: B asymp} and \eqref{eq: bound-hu}, we obtain
\begin{equation}\label{eq: est-B-u-asy}
\int_{|\tau| > m} \left| \Hat{u} (\tau) \overline{\Hat{u} (\tau - p)} \left( \int_0^L B(\tau, x) \diff x
- E |\tau|^{-2} \right)   \right| \diff \tau\leq C\int_{|\tau| > m}|\tau|^{-1}|\Hat{w}(\tau)|
|\Hat{w} (\tau-p)| \diff \tau.
\end{equation}
\\
For $|\tau|$ large enough, we have 
\[|\tau|^{\frac{2}{3}}(1+|\tau|^2)^{-\frac{1}{6}}(1+|\tau-p|^2)^{-\frac{1}{6}}=1+\frac{p}{3\tau}+\bigO(\frac{1}{\tau^2}).\]
Therefore, for the second part of the right-hand side, we derive  that
\begin{equation}\label{eq: second-part}
\begin{aligned}
&|E|  \int_{|\tau| > m} |\tau|^{-\frac{2}{3}}\left|(1+|\tau|^2)^{-\frac{1}{6}}(1+|\tau-p|^2)^{-\frac{1}{6}} |\tau|^{\frac{2}{3}}|\alpha|^2 \Hat{w} (\tau)  \overline{\Hat{w} (\tau - p)} -  |\tau|^{-\frac{4}{3}} \Hat{u}
(\tau)
\overline{\Hat{u} (\tau - p)} \right| d\tau\\
&\leq C \int_{|\tau| > m} |\Hat{w} (\tau)| |\Hat{w} (\tau - p)|  |\tau|^{-\frac{5}{3}} d\tau+C \int_{|\tau| > m} |\Hat{w} (\tau)| |\Hat{w} (\tau - p)|  |\tau|^{-1} d\tau\\
&\le  C \int_{|\tau| > m} |\Hat{w} (\tau)| |\Hat{w} (\tau - p)|  |\tau|^{-1} d\tau.
\end{aligned}
\end{equation}
Putting \eqref{eq: est-B-u-asy} and \eqref{eq: second-part} together, we obtain
\begin{multline*}
\left|\int_{|\tau| > m} \Hat{u} (\tau) \overline{\Hat{u} (\tau - p)} \int_0^L B(\tau, x) \diff x \diff \tau - E
|\alpha|^2
\int_{|\tau| > m} (1+|\tau|^2)^{-\frac{1}{6}}(1+|\tau-p|^2)^{-\frac{1}{6}} \Hat{w} (\tau) \overline{\Hat{w}(\tau-p)} \diff \tau \right| \\
\leq  C \int_{|\tau| > m} |\Hat{w} (\tau)| |\Hat{w} (\tau - p)|  |\tau|^{-1} d\tau.
\end{multline*}
\\
\vspace{3mm}
Then we estimate the low-frequency part $|\tau|\leq m$. 
Combining the estimate \eqref{eq: bound-integral-B-u} and $(1+|\tau|^2)^{-\frac{1}{6}}(1+|\tau-p|^2)^{-\frac{1}{6}}\leq1$, we obtain
\begin{multline*}
\left|\int_{|\tau|\leq m} \Hat{u} (\tau) \overline{\Hat{u} (\tau - p)} \int_0^L B(\tau, x) \diff x \diff \tau  - E
|\alpha|^2
\int_{|\tau|\leq m} (1+|\tau|^2)^{-\frac{1}{6}}(1+|\tau-p|^2)^{-\frac{1}{6}} \Hat{w} (\tau) \overline{ \Hat{w}(\tau-p)} \diff \tau  \right| \\
\le  C \int_{|\tau| \le m} |\Hat{w} (\tau)| |\overline{\Hat{w} (\tau - p)}| \diff  \tau .
\end{multline*}
\begin{remark}
It is noteworthy to mention that here if we use \eqref{eq: bound-integral-B-u}, it is easy to obtain 
\[
\textrm{LHS}\leq   C \int_{|\tau| \le m}(1+|\tau|^2)^{-\frac{1}{3}} |\Hat{w} (\tau)| |\overline{\Hat{w} (\tau - p)}| \diff  \tau. 
\]
This is not sufficient to track the dependence of $T$. In fact, the difficult part of this proof is to gain the smallness of $T$ at the energy level $H^{-\frac{1}{3}}(\R)$, where we cannot apply the compact support of $w$ directly, due to the nonlocal effect of $\langle D_t\rangle^{-\frac{1}{6}}$.

Compared with \cite{CKN}, they have the bound
\[
\textrm{LHS}\leq   C \int_{|\tau| \le m} |\Hat{w} (\tau)| |\overline{\Hat{w} (\tau - p)}| \diff  \tau. 
\]
They gain the smallness of $T$ by 
\[
|\Hat{w} (\tau)|\leq C\|w\|_{L^1(\R)}=C\|w\|_{L^1((-T,T))}\leq CT^{\frac{1}{2}}\|w\|_{L^2(\R)}.
\]
The second equality holds because of $\supp{w}\subset[-T,T]$. However, in our case, in the appropriate energy level $H^{-\frac{1}{3}}(\R)$, we cannot directly have
\[
|(1+|\tau|^2)^{-\frac{1}{6}}\Hat{w} (\tau)|\leq C\|\langle D_t\rangle^{-\frac{1}{6}}w\|_{L^1(\R)}\sim C\|\langle D_t\rangle^{-\frac{1}{6}}w\|_{L^1((-T,T))}.
\]
So we introduce a method involving Lemma \ref{lem: embedding-compact-support} and pseudodifferential operators to quantify the smallness of $T$ gained from the compact support of $w$.

Here we choose to use $(1+|\tau|^2)^{-\frac{1}{3}}\leq 1$. In fact, this is not the optimal choice, but it is sufficient to conclude.
\end{remark}
Summing up the estimates for high-frequency part and low-frequency part, we derive
\begin{multline*}
\left|\int_{\R} \Hat{u} (\tau) \overline{\Hat{u} (\tau - p)} \int_0^L B(\tau, x) \diff x \diff \tau  - E
|\alpha|^2
\int_{\R} (1+|\tau|^2)^{-\frac{1}{6}}(1+|\tau-p|^2)^{-\frac{1}{6}} \Hat{w} (\tau) \overline{ \Hat{w}(\tau-p)} \diff \tau  \right| \\
\le  C \int_{|\tau| \le m} |\Hat{w} (\tau)| |\overline{\Hat{w} (\tau - p)}| \diff  \tau + C m^{-\frac{1}{3}}
\int_{|\tau| > m} |\tau|^{-\frac{2}{3}}|\Hat{w} (\tau)| |\Hat{w}(\tau-p)| \diff \tau.
\end{multline*}
We first deal with the term $\int_{|\tau| \le m} |\Hat{w} (\tau)| |\overline{\Hat{w} (\tau - p)}| \diff  \tau$,
\begin{multline*}
\int_{|\tau| \le m} |\Hat{w} (\tau)| |\overline{\Hat{w} (\tau - p)}| \diff  \tau \\
\le  \int_{|\tau| \le m}(1+|\tau|^2)^{\frac{1}{6}+\varepsilon}(1+|\tau-p|^2)^{\frac{1}{6}+\varepsilon}(1+|\tau|^2)^{-\frac{1}{6}-\varepsilon}(1+|\tau-p|^2)^{-\frac{1}{6}-\varepsilon} |\Hat{w} (\tau)| |\overline{\Hat{w} (\tau - p)}| \diff  \tau
\end{multline*}
Using Peetre's inequality in Proposition \ref{prop: peetre},
\begin{equation*}
(1+|\tau|^2)^{\frac{1}{6}+\varepsilon}(1+|\tau-p|^2)^{\frac{1}{6}+\varepsilon}\leq 2^{\frac{1}{6}+\varepsilon}(1+|\tau|^2)^{\frac{1}{3}+2\varepsilon}(1+|p|^2)^{\frac{1}{6}+\varepsilon}
\end{equation*}
Using $|\tau|\leq m$ and $m\geq1$,
\begin{equation*}
(1+|\tau|^2)^{\frac{1}{6}+\varepsilon}(1+|\tau-p|^2)^{\frac{1}{6}+\varepsilon}\leq 2^{\frac{1}{6}+\varepsilon}(1+|p|^2)^{\frac{1}{6}+\varepsilon}|m|^{\frac{2}{3}+4\varepsilon}
\end{equation*}
Therefore,
\begin{multline*}
\int_{|\tau| \le m} |\Hat{w} (\tau)| |\overline{\Hat{w} (\tau - p)}| \diff  \tau \\
\le  2^{\frac{1}{6}+\varepsilon}(1+|p|^2)^{\frac{1}{6}+\varepsilon}|m|^{\frac{2}{3}+4\varepsilon}\int_{|\tau| \le m}(1+|\tau|^2)^{-\frac{1}{6}-\varepsilon}(1+|\tau-p|^2)^{-\frac{1}{6}-\varepsilon} |\Hat{w} (\tau)| |\overline{\Hat{w} (\tau - p)}| \diff  \tau\\
\leq 2^{\frac{1}{6}+\varepsilon}(1+|p|^2)^{\frac{1}{6}+\varepsilon}m^{\frac{2}{3}+4\varepsilon}\|w\|^2_{H^{-\frac{1}{3}-2\varepsilon}(\R)}.
\end{multline*}
By Lemma \ref{lem: compact-supp-w}, we know that $w$ has compact support $[-T,T]$. For $T\in(0,1)$, by Lemma \ref{lem: embedding-compact-support}, we know that for $-\frac{1}{3}-2\varepsilon>-\frac{1}{2}$,
\begin{equation*}
\|w\|^2_{H^{-\frac{1}{3}-2\varepsilon}(\R)}\leq CT^{4\varepsilon}\|w\|^2_{H^{-\frac{1}{3}}(\R)}.
\end{equation*}
Hence, we derive that
\begin{equation*}
\int_{|\tau| \le m} |\Hat{w} (\tau)| |\overline{\Hat{w} (\tau - p)}| \diff  \tau \le Cm^{\frac{2}{3}+4\varepsilon}T^{4\varepsilon}\|w\|^2_{H^{-\frac{1}{3}}(\R)}.
\end{equation*}
For the other term $m^{-\frac{1}{3}}
\int_{|\tau| > m} |\tau|^{-\frac{2}{3}}|\Hat{w} (\tau)| |\Hat{w}(\tau-p)| \diff \tau$,
\begin{multline*}
\int_{|\tau| > m} |\tau|^{-\frac{2}{3}}|\Hat{w} (\tau)| |\Hat{w}(\tau-p)| \diff \tau\\
=\int_{|\tau| > m} \frac{(1+|\tau|^2)^{\frac{1}{6}}(1+|\tau-p|^2)^{\frac{1}{6}}}{|\tau|^{\frac{2}{3}}}(1+|\tau|^2)^{-\frac{1}{6}}(1+|\tau-p|^2)^{-\frac{1}{6}}|\Hat{w} (\tau)| |\Hat{w}(\tau-p)| \diff \tau
\end{multline*}
Applying Peetre's inequality again, we obtain
\begin{multline*}
\frac{(1+|\tau|^2)^{\frac{1}{6}}(1+|\tau-p|^2)^{\frac{1}{6}}}{|\tau|^{\frac{2}{3}}}\le \frac{2^{\frac{1}{6}}(1+|\tau|^2)^{\frac{1}{3}}(1+|p|^2)^{\frac{1}{6}}}{|\tau|^{\frac{2}{3}}}\le 2^{\frac{1}{6}}(1+|\tau|^{-2})^{\frac{1}{3}}(1+|p|^2)^{\frac{1}{6}}\le 2^{\frac{1}{2}}(1+|p|^2)^{\frac{1}{6}}.
\end{multline*}
Consequently, we have
\begin{equation*}
m^{-\frac{1}{3}}
\int_{|\tau| > m} |\tau|^{-\frac{2}{3}}|\Hat{w} (\tau)| |\Hat{w}(\tau-p)| \diff \tau \leq Cm^{-\frac{1}{3}}\|w\|^2_{H^{-\frac{1}{3}}(\R)}.
\end{equation*}
In summary, we obtain the following estimate:
\begin{multline*}
\left|\int_{\R} \Hat{u} (\tau) \overline{\Hat{u} (\tau - p)} \int_0^L B(\tau, x) \diff x \diff \tau  - E
|\alpha|^2
\int_{\R} (1+|\tau|^2)^{-\frac{1}{6}}(1+|\tau-p|^2)^{-\frac{1}{6}} \Hat{w} (\tau)  \overline{\Hat{w}(\tau-p)} \diff  \tau \right|\\
\le C \left( m^{\frac{2}{3}+4\varepsilon}T^{4\varepsilon} + m^{-\frac{1}{3}} \right) \|w\|^2_{H^{-\frac{1}{3}}(\R)}.
\end{multline*}
Let $m^{\frac{2}{3}+4\varepsilon}T^{4\varepsilon} = m^{-\frac{1}{3}}$. Thus, $m=T^{-\frac{4\varepsilon}{1+4\varepsilon}}$. Taking $\varepsilon=\frac{1}{36}$, then  $\delta_0=\frac{4\varepsilon}{3(1+4\varepsilon)}=\frac{1}{30}$ and $-2\varepsilon-\frac{1}{3}>-\frac{1}{2}$, we conclude the claim by the following inequality
\begin{multline*}
\left|\int_{\R} \Hat{u} (\tau) \overline{\Hat{u} (\tau - p)} \int_0^L B(\tau, x) \diff x \diff \tau  - E
|\alpha|^2
\int_{\R} (1+|\tau|^2)^{-\frac{1}{6}}(1+|\tau-p|^2)^{-\frac{1}{6}} \Hat{w} (\tau)  \overline{\Hat{w}(\tau-p)} \diff  \tau \right|\\
\le C T^{\frac{1}{30}} \|w\|^2_{H^{-\frac{1}{3}}(\R)}.
\end{multline*}
\\
\vspace{3mm}
Now we deal with the main term of $Q_M$. It is easy to write
\[
\int_{\R} (1+|\tau|^2)^{-\frac{1}{6}}(1+|\tau-p|^2)^{-\frac{1}{6}} \Hat{w} (\tau)  \overline{\Hat{w}(\tau-p)} \diff  \tau=\int_{\R}|f|^2e^{-\ii  tp}\diff t,
\]
where $\Hat{f}(\tau)=(1+|\tau|^2)^{-\frac{1}{6}}\Hat{w} (\tau)$ or we denote by $f=\langle D_t\rangle^{-\frac{1}{3}}w$. We choose two cutoff functions $\chi,\Tilde{\chi}$ such that $\chi=\Tilde{\chi}=1$ on $[-1,1]$, and $\supp\chi\subset[-2,2]\subset\supp{\Tilde{\chi}}\subset[-3,3]$. In addition, we choose $\Tilde{\chi}=1$ on the support of $\chi$, i.e. $\Tilde{\chi}\chi=\chi$. Let $\beta\in (0,1)$ and $T\in(0,1)$. We define
\[
\Tilde{\chi}_{\beta}(t):=\Tilde{\chi}(\frac{t}{T^{\beta}}),\quad \chi_{\beta}(t):=\chi(\frac{t}{T^{\beta}}),\forall t\in\R.
\]
Thanks to $\supp w\subset[-T,T]\subset(-T^{\beta},T^{\beta})$, we have the following identity:
\[
\Tilde{\chi}_{\beta}w=w,\;\chi_{\beta}w=w.
\]
Thus,
\begin{multline*}
\int_{\R}|f|^2e^{-\ii  tp}\diff t=\int_{\R}|\Tilde{\chi}_{\beta}f|^2e^{-\ii  tp}\diff t+\int_{\R}(1-\Tilde{\chi}_{\beta}^2)|f|^2e^{-\ii  tp}\diff t\\
=\int_{\R}|\Tilde{\chi}_{\beta}f|^2(1+\bigO(T^{\beta}))\diff t+\int_{\R}(1-\Tilde{\chi}_{\beta}^2)|f|^2e^{-\ii  tp}\diff t
\end{multline*}
Away from time $0$, we aim to prove $\int_{\R}(1-\Tilde{\chi}_{\beta}^2)|f|^2e^{-\ii  tp}\diff t$ is negligible when $T$ is sufficiently small. 
\begin{align*}
\left|\int_{\R}(1-\Tilde{\chi}_{\beta}^2)|f|^2e^{-\ii  tp}\diff t\right|&\leq \int_{\R}(1-\Tilde{\chi}_{\beta}^2)|f|^2\diff t\\
&\leq \int_{\R}(1-\Tilde{\chi}_{\beta}^2)|\langle D_t\rangle^{-\frac{1}{3}}\chi_{\beta} w|^2\diff t\\
&\leq 2\int_{\R}(1-\Tilde{\chi}_{\beta}^2)|[\langle D_t\rangle^{-\frac{1}{3}},\chi_{\beta}] w|^2\diff t+2 \int_{\R}(1-\Tilde{\chi}_{\beta}^2)\chi_{\beta}^2|\langle D_t\rangle^{-\frac{1}{3}} w|^2\diff t.
\end{align*}
For the term $\int_{\R}(1-\Tilde{\chi}_{\beta}^2)\chi_{\beta}^2|\langle D_t\rangle^{-\frac{1}{3}} w|^2\diff t$, due to $(1-\Tilde{\chi}_{\beta}^2)\chi_{\beta}^2\equiv0$, we know it vanishes. Then
\[
\left|\int_{\R}(1-\Tilde{\chi}_{\beta}^2)|f|^2e^{-\ii  tp}\diff t\right|\le \int_{\R}(1-\Tilde{\chi}_{\beta}^2)|[\langle D_t\rangle^{-\frac{1}{3}},\chi_{\beta}] w|^2\diff t.
\]
Applying Corollary \ref{cor: commutator-est}, we know the commutator $\|[\langle D_t\rangle^{-\frac{1}{3}},\chi_{\beta}]\|_{\mathcal{L}(H^{-\frac{4}{3}}(\R),L^2(\R))}\leq CT^{-s_0\beta}$ with $s_0=\frac{13}{3}$. We emphasize that this constant $C$ is independent of $T$. Hence, 
\[
\int_{\R}(1-\Tilde{\chi}_{\beta}^2)|[\langle D_t\rangle^{-\frac{1}{3}},\chi_{\beta}] w|^2\diff t\le \|[\langle D_t\rangle^{-\frac{1}{3}},\chi_{\beta}] w\|^2_{L^2(\R)}
\le CT^{-2s_0\beta}\|w\|^2_{H^{-\frac{4}{3}}(\R)}.
\]
Using Lemma \ref{lem: embedding-compact-support}, we have
\[
\int_{\R}(1-\Tilde{\chi}_{\beta}^2)|[\langle D_t\rangle^{-\frac{1}{3}},\chi_{\beta}] w|^2\diff t\le CT^{-2s_0\beta}\|w\|^2_{H^{-\frac{4}{3}}(\R)}\le CT^{2(\frac{1}{6}-s_0\beta)}(1+\ln{\frac{1}{T}})\|w\|^2_{H^{-\frac{1}{3}}(\R)}.
\]
On the other hand, 
\begin{multline*}
\int_{\R}|\Tilde{\chi}_{\beta}f|^2\diff t=\int_{\R}|\Tilde{\chi}_{\beta}\langle D_t\rangle^{-\frac{1}{3}}\chi_{\beta} w|^2\diff t\\
=\int_{\R}|\langle D_t\rangle^{-\frac{1}{3}}\Tilde{\chi}_{\beta}\chi_{\beta} w+[\Tilde{\chi}_{\beta},\langle D_t\rangle^{-\frac{1}{3}}]\chi_{\beta} w+|^2\diff t\\
=\int_{\R}|\langle D_t\rangle^{-\frac{1}{3}}w+[\Tilde{\chi}_{\beta},\langle D_t\rangle^{-\frac{1}{3}}]\chi_{\beta} w|^2\diff t\\
=\int_{\R}|\langle D_t\rangle^{-\frac{1}{3}}w|^2\diff t+2\Re\int_{\R}\langle D_t\rangle^{-\frac{1}{3}}w\overline{[\Tilde{\chi}_{\beta},\langle D_t\rangle^{-\frac{1}{3}}]\chi_{\beta} w}\diff t+\int_{\R}|[\Tilde{\chi}_{\beta},\langle D_t\rangle^{-\frac{1}{3}}]\chi_{\beta} w|^2\diff t
\end{multline*}
Using Cauchy-Schwarz's inequality,
\begin{gather*}
\left|2\Re\int_{\R}\langle D_t\rangle^{-\frac{1}{3}}w\overline{[\Tilde{\chi}_{\beta},\langle D_t\rangle^{-\frac{1}{3}}]\chi_{\beta} w}\diff t\right|\leq 2\|w\|_{H^{-\frac{1}{3}}(\R)}\|[\Tilde{\chi}_{\beta},\langle D_t\rangle^{-\frac{1}{3}}]w\|_{L^2(\R)}\le CT^{-s_0\beta}\|w\|_{H^{-\frac{1}{3}}(\R)}\|w\|_{H^{-\frac{4}{3}}(\R)}\\
\le CT^{\frac{1}{6}-s_0\beta}(1+\ln{\frac{1}{T}})^{\frac{1}{2}}\|w\|^2_{H^{-\frac{1}{3}}(\R)}
\end{gather*}
Similarly,
\begin{equation*}
\int_{\R}|[\Tilde{\chi}_{\beta},\langle D_t\rangle^{-\frac{1}{3}}]\chi_{\beta} w|^2\diff t\le CT^{2(\frac{1}{6}-s_0\beta)}(1+\ln{\frac{1}{T}})\|w\|^2_{H^{-\frac{1}{3}}(\R)},
\end{equation*}
we obtain
\begin{multline*}
\int_{\R}|f|^2e^{-\ii  tp}\diff t=\int_{\R}|\Tilde{\chi}_{\beta}f|^2(1+\bigO(T^{\beta}))\diff t+\|w\|^2_{H^{-\frac{1}{3}}(\R)}\bigO(T^{2(\frac{1}{6}-s_0\beta)}(1+\ln{\frac{1}{T}}))\\
=\|w\|^2_{H^{-\frac{1}{3}}(\R)}\left(1+\bigO(T^{\beta})+\bigO(T^{\frac{1}{6}-s_0\beta}(1+\ln{\frac{1}{T}})^{\frac{1}{2}})+\bigO(T^{2(\frac{1}{6}-s_0\beta)}(1+\ln{\frac{1}{T}}))\right).
\end{multline*}
Since $\lim_{T\to0^+}T^{s_0\beta}(1+\ln{\frac{1}{T}})^{\frac{1}{2}}=0$, for any $\beta\in(0,1)$, we choose $\beta$ such that $\frac{1}{6}-2s_0\beta>0$, then
\[
\int_{\R}|f|^2e^{-\ii  tp}\diff t=\|w\|^2_{H^{-\frac{1}{3}}(\R)}\left(1+\bigO(T^{\beta})+\bigO(T^{\frac{1}{6}-2s_0\beta})+\bigO(T^{2(\frac{1}{6}-2s_0\beta)})\right).
\]
Therefore, we have
\[
\int_{\R} \Hat{u} (\tau) \overline{\Hat{u} (\tau - p)} \int_0^L B(\tau, x) \diff x \diff \tau =  E
|\alpha|^2\|w\|^2_{H^{-\frac{1}{3}}(\R)}\left(1+\bigO(T^{\beta})+\bigO(T^{\frac{1}{6}-2s_0\beta})+\bigO(T^{2(\frac{1}{6}-2s_0\beta)})+\bigO(T^{\frac{1}{30}})\right).
\]
Choosing $\beta$ sufficiently small, for example, $\beta=\frac{1}{100}$, we obtain:
\[
\int_{\R} \Hat{u} (\tau) \overline{\Hat{u} (\tau - p)} \int_0^L B(\tau, x) \diff x \diff \tau =  E
|\alpha|^2\|w\|^2_{H^{-\frac{1}{3}}(\R)}\left(1+\bigO(T^{\frac{1}{100}})\right).
\]
The conclusion follows by noting that
\[
 \|w (\tau)\|_{H^{-\frac{1}{3}}(\R)}^2  \ge C \int_{\R}
\frac{|\Hat{u}(\tau)|^2}{1 + |\tau|^{2}} \diff \tau,
\]
and  by normalizing $u$ such that $|\alpha| \| w \|_{L^2(\R)} = 1$.
\end{proof}
\subsection{The explicit construction of the trapping direction $\Psi$}\label{sec: construction trap direction}
Let $(k,l)\in\mathcal{S}_2$, and 
\begin{equation}\label{eq: defi-p-L}
L=2 \pi \sqrt{\frac{k^2 + k l + l^2}{3}}, \;p=\frac{(2k + l)(k-l)(2 l + k)}{3 \sqrt{3}(k^2 + kl + l^2)^{3/2}}.  
\end{equation}
 We solve the characteristic equation $\eta^3 + \eta +\ii p =0$ and the roots read as follows:
\begin{equation}\label{eq: eta-eigenvalue-eq}
\eta_1 = - \frac{2 \pi i}{3 L} (2k + l), \quad \eta_2 = \eta_1 + \frac{2\pi i}{L} k, \quad
\eta_3 = \eta_2 + \frac{2\pi i}{L} l.
\end{equation}
\begin{lemma} \label{lem: E-not-0}
Let $(k,l)\in\mathcal{S}_2$ and let $E$ be given by \eqref{eq: defi-E}
with $\eta_j$  in \eqref{eq: eta-eigenvalue-eq} and with $p,L$ in \eqref{eq: defi-p-L}. Then
\begin{equation*}
E =  -\frac{8\pi^3p}{9 L^2}  kl (k+l)\neq0.
\end{equation*}
\end{lemma}
\begin{proof} 
We compute directly
\begin{align*}
\sum_{j=1}^3  \frac{\eta_{j+1} - \eta_j}{\eta_{j+2}}
=  \frac{3k}{k + 2 l } - \frac{3 l }{2 k + l} - \frac{3 (k+l)}{k -l} = - \frac{2 7 k l
(k+l)}{(k+2l) (2 k + l) (k-l)}.
\end{align*}
We then have, by \eqref{eq: defi-E},
\begin{equation*}
E = \frac{1}{9}p^2L\left(-\frac{2 7
kl (k+l)}{(k -l) (k+ 2l) (2l + k)} \right)= -\frac{8\pi^3p}{9 L^2}  kl (k+l).
\end{equation*}
The proof is complete.
\end{proof}
Let $(k,l)\in\mathcal{S}_2$, $\eta_j$ be in \eqref{eq: eta-eigenvalue-eq} with $p,L$ in \eqref{eq: defi-p-L}. As a consequence of Lemma \ref{lem: E-not-0}, we obtain a nonzero direction $\Psi= \Psi_{(k, l)}$ defined as follows
\begin{equation}\label{eq: defi-Psi}
\Psi(t, x) = \Re(E) \Re \{ \varphi (x) e^{-ipt}\} + \Im (E) \Im \{ \varphi (x) e^{-ipt}\},
\end{equation}
which satisfies satisfies
the linear KdV system \eqref{eq: defi-trapping direction} with $\varphi$ defined in \eqref{eq: defi varphi}, as noted in \cite{Cerpa07}.

\vspace{3mm}

Equipped with Proposition \ref{prop: coercive-1/3norm}, we are able to give proof to Proposition \ref{prop: coercive property}.
\begin{proof}[Proof of Proposition \ref{prop: coercive property}]
By Lemma \ref{lem: B asymp}, let   $p \in \R$ and let $\varphi$ be defined by \eqref{eq: defi varphi}, where \eqref{eq: e^eta=1} holds,  $\eta_j \neq 0$, $\eta_j^3+\eta_j+ip=0$ for $j=1,
2, 3$. For $E \neq 0$, there exists $T_* > 0$ such that, for any  $u \in
L^2(0, + \infty)$ with $u(t) = 0$ for $t > T_*$ and $y(t,
\cdot) = 0$ for $t>T_*$ where $y$ is the unique solution of \eqref{eq: linearized controled kdv-0-0}. \\
Multiplying \eqref{eq: coercive-est-y} by  $\overline{E}$ and normalizing appropriately,
 we have
\begin{equation*}
\Re\int_{0}^{\infty} \int_0^L |y (t, x)|^2\overline{E} e^{-\ii pt }\varphi_x(x) \diff x \diff t=\int_{0}^{\infty} \int_0^L |y (t, x)|^2\p_x\Psi(t,x) \diff x \diff t.
\end{equation*}
Therefore, for $T_*$ sufficiently small,
\begin{equation*}
\int_{0}^\infty \int_0^{+\infty} y^2(t, x) \p_x\Psi(t, x) dx dt  \ge C \| u
\|_{H^{-1}(\R)}^2. 
\end{equation*}
We conclude with the estimate above.
\end{proof}

\section{Obstruction to the small time local null-controllability  of the KdV system}\label{sec: obstruction}

The main result of this section is the following, which implies in particular the Main Theorem.

\begin{theorem}\label{thm: not-small-time}
Let $(k,l)\in\mathcal{S}_2$, $\eta_j$ be in \eqref{eq: eta-eigenvalue-eq} with $p,L$ in \eqref{eq: defi-p-L}, and $\Psi$ be defined in \eqref{eq: defi-Psi}. There exists $\varepsilon_0 > 0$ such that for all $0< \varepsilon <\varepsilon_0$,  for all  $0 < T < T_*/2$, where $T_*$ is the constant in Proposition \ref{prop: coercive property} with $p$, $\eta_j$, and $L$ given previously, and for all  solutions $y \in C\left([0, + \infty); H^4(0, L)\right) \cap L^2_{loc}\left([0,+ \infty); H^5(0, L) \right) $  of
\begin{equation}\label{eq: kdv-psi-sys}
\left\{
\begin{array}{cl}
\p_t y  + \p_x^3y  + \p_x y  + y \p_xy  = 0 &  \mbox{ in } (0, + \infty)
\times  (0, L), \\
y(t, 0) = y(t, L) = 0 & \mbox{ in }  (0, + \infty), \\
\p_xy(t ,  L) = u(t) & \mbox{ in } (0, \infty),  \\
y(0, \cdot) = y_0 (x) : = \varepsilon \Psi(0, \cdot),
\end{array}\right.
\end{equation}
with $u \in H^{\frac{4}{3}}(\R_+)$,  $\| u \|_{H^{\frac{4}{3}}(\R_+)} < \varepsilon_0$, $u(0) = 0$,  and $\supp
u \subset [0, T]$, we have
\[
y(T, \cdot) \neq 0,\]
i.e., the small-time null controllability fails.
\end{theorem}
\begin{proof}Before we present our proof, we recall the negative Sobolev norm defined in \eqref{eq: defi-negative-norm} and the fractional Sobolev norm in Appendix \ref{sec: Sobolev–Slobodeckij spaces}. We point out that all controls we use in this proof are \`a priori defined in $L^2(\R_+)$. Consequently, their negative Sobolev norms are both well-defined and finite. Regarding the fractional Sobolev norm, it is applied only to $u\in H^{\frac{4}{3}}(\R_+)$ and its zero extension, which is guaranteed by Lemma \ref{lem: fractional-sobolev} and Corollary \ref{cor: s-Sobolev} in Appendix \ref{sec: Sobolev–Slobodeckij spaces}.

\vspace{3mm}
By Lemma \ref{lem: E-not-0}, the constant $E$ is not $0$.  Let
$\varepsilon_0$ be a small positive constant, which depends only on $k$ and $l$ and is determined
later.  We prove Theorem \ref{thm: not-small-time} by contradiction.    Assume that there exists a solution $y
\in C\left([0, + \infty); H^4(0, L) \right) \cap L^2_{loc}\left([0, + \infty); H^5(0, L)
\right) $ of
\eqref{eq: kdv-psi-sys}  with $y(t, \cdot) =0$ for $t \ge T$,  for some $u \in H^{\frac{4}{3}}(0, +
\infty)$, for some $0< \varepsilon < \varepsilon_0$, and for some $0< T < T_*/2$  with   $\| u
\|_{H^{\frac{4}{3}}(\R_+)} < \varepsilon_0$, $u(0) = 0$,  and $\supp u \subset [0, T]$.

We adapt the approach introduced in \cite[Proof of Theorem 5.1]{CKN}, concerning the obstruction to small-time controllability. We introduce 
\begin{equation}\label{eq: def-y1}
y_1 (t, x) = y  (t, x) -  c \int_{0}^L y(t, \eta) \Psi(t, \eta) \diff  \eta \,   \Psi(t,
x),
\end{equation}
with a normalized constant $c^{-1}  := \int_0^L |\Psi(0, \eta)|^2 \diff  \eta $. It is easy to see that $y_1(0,x)=0$ for $x\in(0,L)$ and $y_1$ satisfies \eqref{eq: general linear kdv} with $h_1=h_2\equiv0$, $h_3(t)=u(t)$, and 
\[
f(t, x) = f_{1} (t, x) +\p_x f_2 (t, x),
\]
with
\[
f_1(t, x)  =
\frac{c}{2} \int_0^L y^2(t, \eta) \p_x\Psi (t, \eta) \diff  \eta \,   \Psi(t, x),\quad f_2(t, x) = \frac{1}{2} y^2 (t, x).
\]
Thanks to Lemma \ref{lem: kdv-NL}, we obtain
\begin{gather*}
\| y \|_{L^2\left( (0, T) \times (0, L) \right)}  \le C \left(\| y_0\|_{L^2(0, L)} + \| u
\|_{H^{-\frac{1}{3}}(\R)} \right),\\
\|y\|_{L^2\left( (0, T); H^{-1} (0, L) \right)} + \|y_1\|_{L^2\left( (0, T); H^{-1} (0, L) \right)}\le C\left(\| y_0 \|_{L^2(0, L)} + \|
u\|_{H^{-\frac{2}{3}} (\R)} \right).
\end{gather*}
Here we emphasize that this constant $C>0$ is independent of the solution $y$ and the control $u$. In this proof, to simplify the notation, we will use $C$ to denote a strictly positive constant that is independent of the solution $y$ and the control $u$. $C$ may change from line to line.

\vspace{3mm}
Now we introduce two functions $y_2$ and $y_3$, satisfying a decomposition $y_1=y_2+y_3$.
Let $y_2 \in C\left([0, + \infty); L^2(0, L) \right) \cap L^2_{loc}\left([0, + \infty);
H^1(0, L) \right)$ be the unique solution of \eqref{eq: general linear kdv} with $h_1=h_2=h_3\equiv0$ and the same $f=f_1+\p_x f_2$ defined above. Let $y_3 \in C\left([0, + \infty); L^2(0, L) \right) \cap L^2_{loc}\left([0, + \infty);
H^1(0, L) \right)$ be the unique solution of \eqref{eq: general linear kdv} with $h_1=h_2\equiv0$, $h_3(t)=u(t)$ and $f=0$.

There exists $u_4 \in L^2(0, +\infty)$  such that $\supp{u_4} \subset [2T_*/3, T_*]$,
\begin{equation*}
\| u_4\|_{L^2(0, + \infty)} \le C \| y_3(2T_*/3, \cdot)\|_{L^2(2T_*/3, T_*)},\;y_4(T_*, \cdot ) = 0,
\end{equation*}
where $y_4 \in C\left([0, + \infty); L^2(0, L) \right) \cap L^2_{loc}\left([0, + \infty);
H^1(0, L) \right)$ is the unique solution of
\begin{equation*}\left\{
\begin{array}{cl}
\p_t y_4+ \p_x^3y_4 +\p_x y_4   = 0&  \mbox{ in } (2T_*/3, +
\infty) \times (0, L), \\
y_4 (t, 0) = y_4(t, L) = 0 & \mbox{ in } (2T_*/3, +\infty), \\
\p_xy_4 (t , L) = u_4 (t) & \mbox{ in } (2T_*/3, +\infty), \\
y_4(T_*/2, \cdot) =  y_3(2T_*/3, \cdot).
\end{array}\right.
\end{equation*}
Such an $u_4$ exists since $y_3(2T_*/3, \cdot)$ is generated from zero at time $0$, see
\cite{Rosier97}. For $u_4$, applying Lemma \ref{lem: kdv-NL}, we obtain
\begin{align*}
\| u_4\|_{L^2(0, + \infty)} &\le  C \| y_2(2T_*/3, \cdot )\|_{L^2(0, L)}\\
&\le C \min \left\{ \| y \|_{L^2 \left( (0, T)\times (0, L)\right)}^2, \| y\|_{L^2\left( (0,
T); H^1(0, L) \right)}  \| y\|_{L^2\left( (0, T); H^{-1}(0, L) \right)} \right\} \\
&\leq C  \min \left\{\left(\| y_0\|_{L^2(0, L)} + \| u \|_{H^{-\frac{1}{3}}(\R)} \right)^2, \varepsilon_0 \left(\| y_0\|_{L^2(0,L)} + \| u \|_{H^{-\frac{2}{3}}(\R)} \right)  \right\}
\end{align*}

Let $\Tilde{y} \in C\left([0, + \infty); L^2(0, L) \right) \cap L^2_{loc}\left([0, + \infty);
H^1(0, L) \right)$ be the unique solution of \eqref{eq: general linear kdv} with $h_1=h_2\equiv0$, $h_3(t)=u_4(t)+u(t)$, and $f=0$. Then, by the choice of $u_4$,
\[
\Tilde{y}(t, \cdot) = 0 \mbox{ for } t \ge T_*.
\]
Multiplying the equation of $y$ with $\Psi(t, x)$, integrating by parts on $[0,T]\times[0, L]$, 
\begin{equation*}
\int_{0}^L y_0 (x)  \Psi(0, x) \diff x  + \frac{1}{2} \int_0^T \int_0^L y^2 (t, x)
\p_x\Psi(t, x) \diff x \diff t  = 0.
\end{equation*}
Considering the systems for $y-y_1$ and $y_1-\Tilde{y}$, we have
\begin{align*}
&\left| \int_{0}^{T} \int_0^L y^2 (t, x) \p_x\Psi (t, x)  \diff x \diff t  - \int_{0}^{T}
\int_0^L y_1^2 (t, x) \p_x\Psi (t, x)  \diff x \diff t  \right|  \\
\leq & C \| y - y_1 \|_{L^2( (0, T); H^1(0, L) )}  \| (y, y_1) \|_{L^2( (0, T);
H^{-1} (0, L) )}\\
\leq &C \varepsilon_0 \| y_0 \|_{L^2(0, L)}  + C \left(\| y_0 \|_{L^2(0, L)} + \| u\|_{H^{-\frac{2}{3}}
(\R)} \right)  \left(\| y_0 \|_{L^2(0, L)} + \| u\|_{H^{-\frac{1}{3}} (\R)} \right)^2 .
\end{align*}
Similarly,  we have
\begin{align*}
\left| \int_{0}^{+ \infty} \int_0^L y_1^2 (t, x) \p_x\Psi (t, x)  \diff x \diff t  -
\int_{0}^{+\infty} \int_0^L \Tilde{y}^2 (t, x) \p_x\Psi (t, x)  \diff x \diff t  \right|  \\
\leq C\left( \| y_0 \|_{L^2(0, L)} + \| u\|_{L^2(\R_+)} \right)^2 \left( \| y_0 \|_{L^2(0,
L)} + \| u\|_{H^{-\frac{2}{3}}(\R)} \right).
\end{align*}
As a consequence, we obtain
\begin{multline*}
\left| \int_{0}^{T} \int_0^L y^2 (t, x) \p_x\Psi (t, x)  \diff x \diff t  -
\int_{0}^{+\infty} \int_0^L \Tilde{y}^2 (t, x) \p_x\Psi (t, x)  \diff x \diff t  \right| \\
\le C \varepsilon_0 \| y_0 \|_{L^2(0, L)}  + C \left(\| y_0 \|_{L^2(0, L)} + \| u\|_{H^{-\frac{2}{3}}
(\R)} \right) \left(\| y_0 \|_{L^2(0, L)} + \| u\|_{L^2 (\R_+)} \right)^2.
\end{multline*}
On the other hand, from Proposition \ref{prop: coercive property} and the choice of $y_0$, we have
\begin{multline*}
 \int_{0}^L y_0 (x)  \Psi(0, x) \diff x + \frac{1}{2}  \int_{0}^{+\infty} \int_0^L \Tilde{y}^2
(t, x) \p_x\Psi (t, x)  \diff x \diff t 
\ge C \left( \|y_0\|_{L^2(0, L)} +  \| u + u_4\|_{H^{-1}(\R)}^2 \right).
\end{multline*}
Indeed, we know that
\[
\| u + u_4\|_{H^{-1}(\R)}^2 \ge C  \| u\|_{H^{-1}(\R)}^2   - C \| u_4
\|_{L^2(\R)}^2
\geq C   \| u\|_{H^{-1}(\R)}^2-  C \left(\| y_0 \|_{L^2(0,
L)} + \| u\|_{H^{-\frac{1}{3}} (\R)} \right)^4,
\]
Therefore, we derive
\begin{multline*}
 \int_{0}^L y_0 (x)  \Psi(0, x) \diff x + \frac{1}{2}  \int_{0}^{\infty} \int_0^L \Tilde{y}^2
(t, x) \p_x\Psi (t, x)  \diff x \diff t
 \ge C \left( \|y_0\|_{L^2(0, L)} +  \| u\|_{H^{-1}(\R)}^2 \right) - C \| u\|_{H^{-\frac{1}{3}}
(\R)}^4 .
\end{multline*}
As a consequence, we derive the following inequality.
\begin{align*}
&C \varepsilon_0 \| y_0 \|_{L^2(0, L)}  + C \left(\| y_0 \|_{L^2(0, L)} + \| u\|_{H^{-\frac{2}{3}} (\R)}
\right) \left(\| y_0 \|_{L^2(0, L)} + \| u\|_{L^2 (\R_+)} \right)^2 \\
\geq &C \left( \|y_0\|_{L^2(0, L)} +  \| u\|_{H^{-1}(\R)}^2 - C \| u\|_{H^{-\frac{1}{3}}(\R)}^4
\right). 
\end{align*}
It follows that, if $\varepsilon_0$ is fixed but sufficiently small,
\begin{equation}\label{eq: contradiction-inequality}
 \| u\|_{H^{-\frac{1}{3}} (\R)}^4 + \| u\|_{H^{-\frac{2}{3}} (\R)} \| u\|_{L^2 (\R_+)}^2 \ge C \|
u\|_{H^{-1}(\R)}^2.
\end{equation}

By the interpolation of Sobolev norms, we have
\begin{equation*}
\| u \|_{H^{-\frac{1}{3}}(\R)}^2 \le C  \| u\|_{H^{\frac{1}{3}}(\R)} \| u\|_{H^{-1}(\R)} \textcolor{black}{\le C \varepsilon_0 \|u\|_{H^{-1}(\R)}},
\end{equation*}
and
\begin{equation}\label{eq: interpolation-2-1}
\| u \|_{L^{2}(\R)}^2 \le C  \| u\|^{8/7}_{H^{-1}(\R)} \| u\|^{6/7}_{H^{\frac{4}{3}}(\R)},
\end{equation}
\begin{equation*}
\| u \|_{H^{-\frac{2}{3}}(\R)} \le C  \| u\|^{6/7}_{H^{-1}(\R)} \| u\|^{1/7}_{H^{\frac{4}{3}}(\R)},
\end{equation*}
Recall that we extended $u$ by 0 for $t < 0$. Thanks to Lemma \ref{lem: fractional-sobolev} and Corollary \ref{cor: s-Sobolev}, using the definition of fractional Sobolev spaces (see Appendix \ref{sec: Sobolev–Slobodeckij spaces} for proof details), we know that $u\in H^{\frac{4}{3}}(\R)$. Moreover, we have the following estimates:    
\begin{equation}\label{eq: interpolation-2-2}
\begin{aligned}
\| u \|_{H^{\frac{4}{3}}(\R)} &\le C \| u \|_{H^{\frac{4}{3}}(\R_+)},\\
\| u \|_{H^{\frac{1}{3}}(\R)} &\le C \| u \|_{H^{\frac{1}{3}}(\R_+)}
\end{aligned}
\end{equation}
Here this constant $C$ is independent of $u$.

Therefore, we obtain 
\begin{equation*}
\| u\|_{H^{-\frac{1}{3}} (\R)}^4 +\| u\|_{H^{-\frac{2}{3}} (\R)} \| u\|_{L^2 (\R_+)}^2\leq C\left(\| u \|_{H^{\frac{1}{3}}(\R)}^2+\| u \|_{H^{\frac{4}{3}}(\R)}\right)\| u\|^2_{H^{-1}(\R)}\leq C\left(\varepsilon_0^2+\varepsilon_0\right)\| u\|^2_{H^{-1}(\R)}
\end{equation*}

Therefore, we derive from
\eqref{eq: contradiction-inequality} that, $\|u\|_{H^{-1}(\R)}^2 \leq C\varepsilon_0^2 \|u\|_{H^{-1}(\R)}^2 + 
C\varepsilon_0\|u\|_{H^{-1}(\R)}^{2}$. So, for fixed sufficiently small $\varepsilon_0$,
\[
u =0.
\]
As a consequence, we know that $Y=y-\varepsilon\Psi$ is a solution to 
\begin{equation*}
\left\{
\begin{array}{cl}
\p_t Y  + \p_x^3Y  + \p_x Y  + Y \p_xY  =\varepsilon^2\Psi\p_x\Psi+\varepsilon\Psi\p_xy+\varepsilon y\p_x\Psi  &  \mbox{ in } (0, + \infty)
\times  (0, L), \\
y(t, 0) = y(t, L) = 0 & \mbox{ in }  (0, + \infty), \\
\p_xy(t ,  L) = 0 & \mbox{ in } (0, \infty),  \\
y(0, \cdot) = 0,
\end{array}\right.
\end{equation*}
Therefore, by energy estimates, we derive that
\begin{align*}
\|Y(t,\cdot)\|_{L^2(0,L)}&\leq C\|\varepsilon^2\Psi\p_x\Psi+\varepsilon\Psi\p_xy+\varepsilon y\p_x\Psi \|_{L^2((0,T)\times(0,L))}\\
&\leq C\left(\varepsilon^2+\varepsilon\left(\|\p_xy\|_{L^2((0,T)\times(0,L))}+\|y\|_{L^2((0,T)\times(0,L))}\right)\right).
\end{align*}
Due to Lemma \ref{lem: kdv-NL},
\[
\|y\|_{L^2((0,T);H^1(0,L))}\leq C\varepsilon\|\Psi(0,\cdot)\|_{L^2(0,L)}\leq C\varepsilon.
\]
Finally, in particular, we obtain the following estimate\footnote{This can be seen as a ``trapping" estimate: if $u=0$, then for small time, the solution $y$ to \eqref{eq: kdv-psi-sys} is very close to $\Psi$.}:
\[
\|y(T, \cdot) - \varepsilon \Psi(T, \cdot) \|_{L^2(0, L)} \le C \varepsilon^2. 
\]
If $\varepsilon_0$ is sufficiently small, this implies that $y(T,\cdot)\neq0$, which is a contradiction.  The proof is complete.
\end{proof}

\vspace{3mm}

\noindent {\bf Acknowledgements} \; 
The authors would like to thank Yuxuan Chen for valuable discussions.
 Shengquan Xiang is partially supported by NSF of China 12301562 and by “The Fundamental Research Funds for the Central Universities, 7100604200, Peking University”. Jingrui Niu is supported by Defi Inria EQIP. Part of this work was done when Jingrui Niu was visiting Peking University. They would like to thank the hospitality and the financial support of these institutions.

\appendix
\section{}
\subsection{Commutator estimates}\label{sec: commutator}
In this section, we provide a proof of Corollary \ref{cor: commutator-est}. 
\begin{proof}[Proof of Corollary \ref{cor: commutator-est}]
For any $f\in \mathcal{S}(\R)$, using the quantization formula, 
\begin{gather*}
(\langle D_x\rangle^{-s}\varrho_Rf)(x)=\frac{1}{2\pi}\int_{\R^2}e^{\ii (x-y)\xi}(1+\xi^2)^{-\frac{s}{2}}\varrho(\frac{y}{R})f(y)\diff y\diff \xi,\\
(\varrho_R\langle D_x\rangle^{-s}f)(x)=\frac{1}{2\pi}\int_{\R^2}e^{\ii (x-y)\xi}\varrho(\frac{x}{R})(1+\xi^2)^{-\frac{s}{2}}f(y)\diff y\diff \xi.
\end{gather*}
By definition, the commutator $[\langle D_x\rangle^{-s},\varrho_R]$ acting on $f$ is given by
\begin{align*}
\left([\langle D_x\rangle^{-s},\varrho_R]f\right)(x)=&\frac{1}{2\pi}\int_{\R^2}e^{\ii (x-y)\xi}(1+\xi^2)^{-\frac{s}{2}}\left(\varrho(\frac{y}{R})-\varrho(\frac{x}{R})\right)f(y)\diff y\diff \xi\\
=&\frac{1}{2\pi}\int_{\R^2}e^{\ii (x-y)\xi}(1+\xi^2)^{-\frac{s}{2}}\left(\int_0^1\frac{y-x}{R}\varrho'(\frac{\theta y+(1-\theta)x}{R})\diff\theta\right)f(y)\diff y\diff \xi\\
=&-\frac{1}{2\ii\pi}\int_{\R^2}\p_{\xi}(e^{\ii (x-y)\xi})(1+\xi^2)^{-\frac{s}{2}}\left(\int_0^1\frac{1}{R}\varrho'(\frac{\theta y+(1-\theta)x}{R})\diff\theta\right)f(y)\diff y\diff \xi.
\end{align*}
Integrating by parts in $\xi$ variable, we obtain
\begin{align*}
\left([\langle D_x\rangle^{-s},\varrho_R]f\right)(x)=&\frac{1}{2\ii\pi}\int_{\R^2}e^{\ii (x-y)\xi}\p_{\xi}(1+\xi^2)^{-\frac{s}{2}}\left(\int_0^1\frac{1}{R}\varrho'(\frac{\theta y+(1-\theta)x}{R})\diff\theta\right)f(y)\diff y\diff \xi\\
=&\frac{s}{2\ii\pi R}\int_{\R^2}e^{\ii (x-y)\xi}\frac{\xi}{(1+\xi^2)^{\frac{1}{2}}}(1+\xi^2)^{\frac{-s-1}{2}}\left(\int_0^1\varrho'(\frac{\theta y+(1-\theta)x}{R})\diff\theta\right)f(y)\diff y\diff \xi.
\end{align*}
We only need to show that there exists a constant $C>0$ such that
\[
\|\langle D_x\rangle^{1+s}[\langle D_x\rangle^{-s},\varrho_R]f\|_{L^2}\leq C\|f\|_{L^2}.
\]
We notice that 
\[
\langle D_x\rangle^{1+s}[\langle D_x\rangle^{-s},\varrho_R]f(x)=\frac{s}{4\ii\pi^2 R}\int_{\R^2}K(x,z)f(z)\diff z\diff \xi,
\]
where
\[
K(x,z)=\int_{\R^3}e^{\ii\left((x-y)\xi+(y-z)\eta\right)}\langle\xi\rangle^{s+1}\langle\eta\rangle^{-s-2}\eta\left(\int_0^1\varrho'(\frac{\theta z+(1-\theta)y}{R})\diff\theta\right)\diff y\diff\eta\diff\xi.
\]
We apply Schur's test, and it suffices to check 
\[
\sup_{x}\int_{\R}|K(x,z)|\diff z<\infty,\mbox{ and }\sup_{z}\int_{\R}|K(x,z)|\diff x<\infty.
\]
Without loss of generality, we show that $\sup_{x}\int_{\R}|K(x,z)|\diff z<\infty$.
\begin{gather*}
|K(x,z)|= |\int_{\R^3}\langle D_{\eta}\rangle^6e^{\ii\left((x-y)\xi+(y-z)\eta\right)}\langle\xi\rangle^{s+1}\langle\eta\rangle^{-s-2}\eta\frac{\left(\int_0^1\varrho'(\frac{\theta z+(1-\theta)y}{R})\diff\theta\right)}{(1+(y-z)^2)^3}\diff y\diff\eta\diff\xi|.
\end{gather*}
Integrating by parts in $\eta-$variable, we obtain
\begin{gather*}
|K(x,z)|= |\int_{\R^3}e^{\ii\left((x-y)\xi+(y-z)\eta\right)}\langle\xi\rangle^{s+1}\langle D_{\eta}\rangle^6\langle\eta\rangle^{-s-2}\eta\frac{\left(\int_0^1\varrho'(\frac{\theta z+(1-\theta)y}{R})\diff\theta\right)}{(1+(y-z)^2)^3}\diff y\diff\eta\diff\xi|.
\end{gather*}
Applying Peetre's inequality (in Lemma \ref{eq: peetre}), we know
\[
\frac{1+z^2}{1+(y-z)^2}\leq C(1+y^2).
\]
Hence,
\begin{gather*}
|K(x,z)|
\leq \frac{C}{1+z^2} |\int_{\R^3}e^{\ii\left((x-y)\xi+(y-z)\eta\right)}\langle\xi\rangle^{s+1}\langle D_{\eta}\rangle^6\langle\eta\rangle^{-s-2}2\eta\frac{\left(\int_0^1\varrho'(\frac{\theta z+(1-\theta)y}{R})\diff\theta\right)}{1+(y-z)^2}\frac{1+y^2}{1+(y-z)^2}\diff y\diff\eta\diff\xi|
\end{gather*}
Let $1\leq k< s+1<k+1$. We note that $\frac{\langle\xi\rangle^{2k+2}}{\langle\xi-\eta\rangle^{2k+2}}\leq C\langle\eta\rangle^{2k+2}$ due to Peetre's inequality. As a consequence, using the oscillatory integral techniques again,
\begin{gather*}
|K(x,z)|
\leq \frac{C}{1+z^2} |\int_{\R^3}e^{\ii\left((x-y)\xi+(y-z)\eta\right)}\frac{\langle\xi\rangle^{s+1}}{\langle\xi-\eta\rangle^{2k+2}}\langle D_{\eta}\rangle^{2k+6}\langle\eta\rangle^{-s-2}\eta\\
\langle D_y\rangle^{2k+2}\frac{\left(\int_0^1\varrho'(\frac{\theta z+(1-\theta)y}{R})\diff\theta\right)}{(1+(y-z)^{2})^{k+1}}\frac{1+y^2}{1+(y-z)^2}\diff y\diff\eta\diff\xi|\\
\leq \frac{C}{R^{2k+2}(1+z^2)}\int_{\R^3}\langle\eta\rangle^{-s-5}\langle\xi\rangle^{s-2k-1}(1+y^2)^{-1-k}\diff y\diff\eta\diff\xi
\end{gather*}
Therefore, 
\[
\sup_{x}\int_{\R}|K(x,z)|\diff z<\frac{C}{R^{2s+4}}.
\]
Similarly, by the oscillatory integral techniques, we can derive that
\[
\sup_{z}\int_{\R}|K(x,z)|\diff x<C.
\]
As a consequence of Schur's Lemma, there exists a constant $C>0$ such that
\[
\|\langle D_x\rangle^{1+s}[\langle D_x\rangle^{-s},\varrho_R]f\|_{L^2}\leq \frac{C}{R^{s+4}}\|f\|_{L^2},
\]
from which we deduce that
\[
\|[\langle D_x\rangle^{-s},\varrho_R]\|_{\mathcal{L}(H^{-s-1},L^2)}\leq \frac{C}{R^{s+4}}.
\]
\end{proof}
\begin{remark}
The study of the commutator between a Fourier multiplier, represented as $a(D_x)$, and a multiplication operator $b(x)$, symbolically expressed as $[a(D_x),b(x)]$, holds significant importance in microlocal analysis. This interaction is intricately connected to the Calderon-Vaillancourt Theorem. For comprehensive details, one can consult the works \cite{Cordes,C-V}. Furthermore, it's noteworthy that in this context, the multiplication operator is parameterized by $R$. Specifically, the expression
\[
\|[\langle D_x\rangle^{-s},\varrho_R]\|_{\mathcal{L}(H^{-s-1},L^2)}\leq \frac{C}{R^{s+4}},
\]
illustrates the dependence, yet the rate in terms of $R$ may not be optimal.
\end{remark}
\subsection{Fractional Sobolev spaces}\label{sec: Sobolev–Slobodeckij spaces}
The approach to define fractional order Sobolev spaces arises from the idea to generalize the H\"older condition to the $L^p$-setting. For $d\geq1$, $1 \leq p < \infty$, $\theta \in (0,1)$ and $f \in L^p(\Omega)$, where $\Omega\subset\R^d$ is a measurable set, the so-called Slobodeckij seminorm (roughly analogous to the H\"older seminorm) is defined by
\[
[f]_{\theta,p,\Omega} = \left( \int_\Omega \int_\Omega \frac{|f(x) - f(y)|^p}{|x - y|^{\theta p + d}} \, dx \, dy \right)^{\frac{1}{p}}.
\]
In particular,  we focus exclusively on the $L^2-$setting in $\R$. We use $[f]_{\theta,\Omega}$ to denote
\[
[f]_{\theta,\Omega} = \left( \int_\Omega \int_\Omega \frac{|f(x) - f(y)|^2}{|x - y|^{2\theta  + n}} \, dx \, dy \right)^{\frac{1}{2}}.
\]
Let $s > 0$ be not an integer and set $\theta = s - \lfloor s \rfloor \in (0,1)$. The $L^2$-based fractional Sobolev space (or Sobolev-Slobodeckij space) $H^{s}(\Omega)$ is defined as
\[
H^{s}(\Omega) = \left\{ f \in H^{\lfloor s \rfloor}(\Omega) : \sup_{k = \lfloor s \rfloor} [\p_x^k f]_{\theta,\Omega} < \infty \right\}.
\]
It is a Banach space for the norm
\[
\|f\|_{H^{s}(\Omega)} = \|f\|_{H^{\lfloor s \rfloor}(\Omega)} + \sup_{k = \lfloor s \rfloor} [\p_x^k f]_{\theta,\Omega}.
\]
\begin{lemma}\label{lem: fractional-sobolev}
Let $s\in(0,\frac{1}{2})$ and $T>0$. Assume that $u\in H^{1+s}(\R_+)$, $u(0)=0$ and $\supp{u}\subset[0,T]$. Define 
\begin{equation*}
U(t)=\left\{
\begin{array}{ll}
   u(t)  ,&t>0,  \\
   0  ,&t\leq0. 
\end{array}
\right.
\end{equation*}
Then, there exists a constant $C>0$ such that any $u\in H^{1+s}(\R_+)$, with $\supp{u}\subset[0,T]$ and $u(0)=0$
\[
\|U\|_{H^{1+s}(\R)}\leq C\|u\|_{H^{1+s}(\R_+)}.
\]
\end{lemma}
\begin{proof}
By the definition of Sobolev--Slobodeckij spaces, 
\[
u\in H^{1+s}(\R_+)\Leftrightarrow u\in H^1(\R_+)\mbox{ and }[u']_{s,\R_+}<\infty.
\]
And the Sobolev--Slobodeckij norm is the following
\[
\|u\|^2_{H^{1+s}(\R_+)}=\|u\|^2_{H^1(\R_+)}+[u']^2_{s,\R_+}.
\]
Thanks to the fact that $\supp{u}\subset[0,T]$ and $u(0)=0$, we know that $u\in H^1_0(\R_+)$. Thus, the extension of $u$ by $0$ for $t<0$, i.e. $U$, is a function belonging to $H^1(\R)$ and
\[
\|U\|_{H^1(\R)}\leq \|u\|_{H^1(\R_+)}.
\]
Define 
\begin{equation*}
V(t)=\left\{
\begin{array}{ll}
     u'(t),&t>0,  \\
     0,&t=0,\\
     u'(-t),&t<0. 
\end{array}
\right.
\end{equation*}
It is clear that $V\in L^2(\R)$. We claim that $V\in H^s(\R)$, i.e., $[V]_{s,\R}<\infty$. Indeed, 
\begin{multline*}
[V]^2_{s,\R}=\int_{\R\times\R}\frac{|V(x)-V(y)|^2}{|x-y|^{1+2s}}\diff x\diff y\\
=\int_{\R_+\times\R_+}\frac{|V(x)-V(y)|^2}{|x-y|^{1+2s}}\diff x\diff y+2\int_{\R_+\times\R_-}\frac{|V(x)-V(y)|^2}{|x-y|^{1+2s}}\diff x\diff y+\int_{\R_-\times\R_-}\frac{|V(x)-V(y)|^2}{|x-y|^{1+2s}}\diff x\diff y.
\end{multline*}
Using the definition of $V$, we know that
\begin{gather*}
[V]^2_{s,\R}
=2\int_{\R_+\times\R_+}\frac{|u'(x)-u'(y)|^2}{|x-y|^{1+2s}}\diff x\diff y+2\int_{\R_+\times\R_-}\frac{|u'(x)-u'(-y)|^2}{|x-y|^{1+2s}}\diff x\diff y\\
=2[u']^2_{s,\R_+}+2\int_{\R_+\times\R_+}\frac{|u'(x)-u'(y)|^2}{|x+y|^{1+2s}}\diff x\diff y.
\end{gather*}
Since $|x+y|^2\geq |x-y|^2$, we know that $|x+y|^{1+2s}\geq |x-y|^{1+2s}$ for $s\in(0,1)$. Therefore,
\begin{align*}
[V]^2_{s,\R}&=2[u']^2_{s,\R_+}+2\int_{\R_+\times\R_+}\frac{|u'(x)-u'(y)|^2}{|x+y|^{1+2s}}\diff x\diff y\\
&\leq 2[u']^2_{s,\R_+}+2\int_{\R_+\times\R_+}\frac{|u'(x)-u'(y)|^2}{|x-y|^{1+2s}}\diff x\diff y\\
&=4[u']^2_{s,\R_+}<\infty.
\end{align*}
Hence, $V\in H^{s}(\R)$ and $\|V\|_{H^s(\R)}\leq 2\|u'\|_{H^s(\R_+)}$. Indeed, we know that $C^{\infty}_c(\R)$ is dense in $H^s(\R)$ for any $s>0$. Thanks to the Hardy inequality in \cite{NgSq18}, combining with a standard approximation argument, we derive that for $s\in(0,\frac{1}{2})$,
\[
\||\cdot|^{-s}V(\cdot)\|_{L^2(\R)}\leq C\|V\|_{H^s(\R)}\leq C\|u'\|_{H^s(\R_+)}.
\]
Therefore, $\||\cdot|^{-s}V(\cdot)\|_{L^2(\R)}\leq C\|u'\|_{H^s(\R_+)}$. Using again the definition of $V$,
\[
\||\cdot|^{-s}V(\cdot)\|^2_{L^2(\R)}=2\int_{\R_+}\frac{|u'(t)|^2}{|t|^{2s}}\diff t,
\]
which implies that for $s\in(0,\frac{1}{2})$, $\int_{\R_+}\frac{|u'(t)|^2}{|t|^{2s}}\diff t\leq C\|u'\|_{H^s(\R_+)}$. Now we look at the seminorm $[U']_{s,\R}$. By definition,  
\begin{align*}
[U']^2_{s,\R}&=\int_{\R\times\R}\frac{|U'(x)-U'(y)|^2}{|x-y|^{1+2s}}\diff x\diff y\\
&=\int_{\R_+\times\R_+}\frac{|U'(x)-U'(y)|^2}{|x-y|^{1+2s}}\diff x\diff y+2\int_{\R_+\times\R_-}\frac{|U'(x)-U'(y)|^2}{|x-y|^{1+2s}}\diff x\diff y+\int_{\R_-\times\R_-}\frac{|U'(x)-U'(y)|^2}{|x-y|^{1+2s}}\diff x\diff y.
\end{align*}
Since $U'=u'\mathbf{1}_{\R_+}$. Since $U'=0$ on $(-\infty,0)$ and $U'=u'$ on $(0,+\infty)$,
\begin{align*}
[U']^2_{s,\R}
&=\int_{\R_+\times\R_+}\frac{|u'(x)-u'(y)|^2}{|x-y|^{1+2s}}\diff x\diff y+2\int_{\R_+\times\R_-}\frac{|u'(x)|^2}{|x-y|^{1+2s}}\diff x\diff y\\
&=[u']^2_{s,\R_+}+\frac{1}{s}\int_{\R_+}\frac{|u'(x)|^2}{|x|^{2s}}\diff x\\
&\leq C[u']^2_{s,\R_+}
\end{align*}
Consequently,
\[
\|U\|^2_{H^{1+s}(\R)}=\|U\|^2_{H^1(\R)}+[U']^2_{s,\R}\leq C\left(\|u\|^2_{H^1(\R_+)}+[u']^2_{s,\R_+}\right)\leq C\|u\|^2_{H^{1+s}(\R_+)}.
\]
\end{proof}
\begin{corollary}\label{cor: s-Sobolev}
Let $s\in(0,\frac{1}{2})$ and $T>0$. Assume that $u\in H^s(\R_+)$, $\supp{u}\subset[0,T]$. Define 
\begin{equation*}
U(x)=\left\{
\begin{array}{ll}
   u(x)  ,&t>0,  \\
   0  ,&t\leq0. 
\end{array}
\right.
\end{equation*}
Then, there exists a constant $C>0$ such that for any $u\in H^s(\R_+)$, with $\supp{u}\subset[0,T]$,
\[
\|U\|_{H^{s}(\R)}\leq C\|u\|_{H^{s}(\R_+)}.
\]
\end{corollary}

\subsection{Eigenmodes at critical lengths}\label{sec: proof of prop-eigenmodes}
\begin{proof}[Proof of Proposition \ref{prop: type-1-2}]
Let $\G$ satisfy the equation 
\begin{equation}
\left\{
\begin{array}{ll}
     \G'''+\G'+\ii\lambda_{c}\G=0,  & \text{ in }(0,L),\\
     \G(0)=\G(L)=0,& \\
     \G'(0)=\G'(L),& 
\end{array}
\right.       
\end{equation}
with $\lambda_c=\frac{(2k+l)(k-l)(2l+k)}{3\sqrt{3}(k^2+k l+l^2)^{\frac{3}{2}}}$ and $L=2\pi\sqrt{\frac{k^2+k l+l^2}{3}}$. Then the general form of eigenfunctions is as follows
\begin{equation*}
\G(x):= e^{\ii x\frac{\sqrt{3} (2 k + l) }{3\sqrt{k^2 + k l + l^2}}}  +C_1 e^{-\ii x\frac{\sqrt{3} (k + 2 l) }{3 \sqrt{k^2 + k l + l^2}}}  +C_2 e^{\ii x\frac{\sqrt{3} (-k + l)}{ \sqrt{k^2 + k l + l^2}}}.
\end{equation*}
Using the boundary conditions, we obtain
\begin{equation*}
\left\{
\begin{array}{l}
     1+C_1+C_2=0,  \\
   -\frac{\ii((-2+C_1+C_2)k++(-1+2C_1-C_2) l)}{\sqrt{3}\sqrt{k^2 + k l + l^2}}=\frac{\ii e^{-\ii\frac{2}{3}(k+2l)\pi}((-2+C_1+C_2)k++(-1+2C_1-C_2) l)}{\sqrt{3}\sqrt{k^2 + k l + l^2}} .
\end{array}
\right.
\end{equation*}
There are two situations,
\begin{enumerate}
    \item If $e^{-\ii\frac{2}{3}(k+2l)\pi}\neq1\Leftrightarrow (k-l)\not\equiv0\mod{3}$, we know that $C_1=\frac{k}{l}$ and $C_2=-\frac{l+k}{l}$. In this case, we have
    \begin{equation*}
    \G(x)= e^{\ii x\frac{\sqrt{3} (2 k + l) }{3\sqrt{k^2 + k l + l^2}}}  +\frac{k}{l} e^{-\ii x\frac{\sqrt{3} (k + 2 l) }{3 \sqrt{k^2 + k l + l^2}}}  -\frac{l+k}{l} e^{\ii x\frac{\sqrt{3} (-k + l)}{ \sqrt{k^2 + k l + l^2}}},    
    \end{equation*}
    with $\G'(0)=\G'(L)=0$.
    \item If $e^{-\ii\frac{2}{3}(k+2l)\pi}=1\Leftrightarrow (k-l)\equiv0\mod{3}$, we know $C_2=-(1+C_1)$ and the second equation is trivial. In fact, we obtain two linearly independent solutions $\G$ as before and $\Tilde{\G}$ in the following form
    \begin{equation*}
    \Tilde{\G}(x)= e^{\ii x\frac{\sqrt{3} (2 k + l) }{3\sqrt{k^2 + k l + l^2}}}  - e^{-\ii x\frac{\sqrt{3} (k + 2 l) }{3 \sqrt{k^2 + k l + l^2}}},     
    \end{equation*}
    with $\Tilde{\G}'(0)=\Tilde{\G}'(L_0)=\ii\frac{\sqrt{3}(k+l)}{\sqrt{k^2+k l+l^2}}\neq0$.
\end{enumerate}
\end{proof}
\cqfd

\bibliographystyle{alpha}
\bibliography{ref.bib}

\newcommand{\etalchar}[1]{$^{#1}$}
\begin{thebibliography}{PMVZ02}

\bibitem[AG07]{AG-book}
S.~Alinhac and P.~G\'erard.
\newblock {\em Pseudo-differential operators and the {N}ash-{M}oser theorem}, volume~82 of {\em Graduate Studies in Mathematics}.
\newblock American Mathematical Society, Providence, RI, 2007.
\newblock Translated from the 1991 French original by Stephen S. Wilson.

\bibitem[BDE20]{BDE20}
K.~Beauchard, J.~Dard\'{e}, and S.~Ervedoza.
\newblock Minimal time issues for the observability of {G}rushin-type equations.
\newblock {\em Ann. Inst. Fourier (Grenoble)}, 70(1):247--313, 2020.

\bibitem[BLBM23]{BLBM}
K.~Beauchard, J.~Le~Borgne, and F.~Marbach.
\newblock On expansions for nonlinear systems error estimates and convergence issues.
\newblock {\em C. R. Math. Acad. Sci. Paris}, 361:97--189, 2023.

\bibitem[BLM24]{beauchard-splitting}
K.~Beauchard, A.~Laurent, and F.~Marbach.
\newblock Control theory and splitting methods.
\newblock {\em Arxiv: 2407.02127}, 2024.

\bibitem[BM14]{2014-Beauchard-Morancey-MCRF}
K.~Beauchard and M.~Morancey.
\newblock Local controllability of 1{D} {S}chr\"{o}dinger equations with bilinear control and minimal time.
\newblock {\em Math. Control Relat. Fields}, 4(2):125--160, 2014.

\bibitem[BM18]{BM-JDE}
K.~Beauchard and F.~Marbach.
\newblock Quadratic obstructions to small-time local controllability for scalar-input systems.
\newblock {\em J. Differential Equations}, 264(5):3704--3774, 2018.

\bibitem[BM20]{BM20}
K.~Beauchard and F.~Marbach.
\newblock Unexpected quadratic behaviors for the small-time local null controllability of scalar-input parabolic equations.
\newblock {\em J. Math. Pures Appl. (9)}, 136:22--91, 2020.

\bibitem[BM24]{BM24}
K.~Beauchard and F.~Marbach.
\newblock A unified approach of obstructions to small-time local controllability for scalar-input systems.
\newblock {\em Arxiv: 2205.14114}, 2024.

\bibitem[BMP25]{BMP-sch}
K.~Beauchard, F.~Marbach, and T.~Perrin.
\newblock Small-time local control of a schr\"odinger equation: a negative and a positive quadratic result.
\newblock {\em Arxiv: 2501.03882}, 2025.

\bibitem[Bou93]{Bourgain93-2}
J.~Bourgain.
\newblock Fourier transform restriction phenomena for certain lattice subsets and applications to nonlinear evolution equations. {II}. {T}he {K}d{V}-equation.
\newblock {\em Geom. Funct. Anal.}, 3(3):209--262, 1993.

\bibitem[Bou97]{Bourgain97}
J.~Bourgain.
\newblock Periodic {K}orteweg de {V}ries equation with measures as initial data.
\newblock {\em Selecta Math. (N.S.)}, 3(2):115--159, 1997.

\bibitem[BS75]{Bona-Smith}
J.~L. Bona and R.~Smith.
\newblock The initial-value problem for the {K}orteweg-de {V}ries equation.
\newblock {\em Philos. Trans. Roy. Soc. London Ser. A}, 278(1287):555--601, 1975.

\bibitem[BSZ03]{Bona03}
J.~L. Bona, S.~M. Sun, and B.-Y. Zhang.
\newblock A nonhomogeneous boundary-value problem for the {K}orteweg-de {V}ries equation posed on a finite domain.
\newblock {\em Comm. Partial Differential Equations}, 28(7-8):1391--1436, 2003.

\bibitem[BSZ09]{BSZ-2}
J.~L. Bona, S.~M. Sun, and B.-Y. Zhang.
\newblock A non-homogeneous boundary-value problem for the {K}orteweg-de {V}ries equation posed on a finite domain. {II}.
\newblock {\em J. Differential Equations}, 247(9):2558--2596, 2009.

\bibitem[CC04]{CC04}
J.-M. Coron and E.~Cr\'{e}peau.
\newblock Exact boundary controllability of a nonlinear {K}d{V} equation with critical lengths.
\newblock {\em J. Eur. Math. Soc. (JEMS)}, 6(3):367--398, 2004.

\bibitem[CC09]{CC09}
E.~Cerpa and E.~Cr\'{e}peau.
\newblock Boundary controllability for the nonlinear {K}orteweg-de {V}ries equation on any critical domain.
\newblock {\em Ann. Inst. H. Poincar\'{e} Anal. Non Lin\'{e}aire}, 26(2):457--475, 2009.

\bibitem[CC13]{CC13}
E.~Cerpa and J.-M. Coron.
\newblock Rapid stabilization for a {K}orteweg-de {V}ries equation from the left {D}irichlet boundary condition.
\newblock {\em IEEE Trans. Automat. Control}, 58(7):1688--1695, 2013.

\bibitem[CCS15]{Chu-Coron-Shang}
J.~Chu, J.-M. Coron, and P.~Shang.
\newblock Asymptotic stability of a nonlinear {K}orteweg--de {V}ries equation with critical lengths.
\newblock {\em J. Differential Equations}, 259(8):4045--4085, 2015.

\bibitem[Cer07]{Cerpa07}
E.~Cerpa.
\newblock Exact controllability of a nonlinear {K}orteweg-de {V}ries equation on a critical spatial domain.
\newblock {\em SIAM J. Control Optim.}, 46(3):877--899, 2007.

\bibitem[Cer14]{Cerpa14}
E.~Cerpa.
\newblock Control of a {K}orteweg-de {V}ries equation: a tutorial.
\newblock {\em Math. Control Relat. Fields}, 4(1):45--99, 2014.

\bibitem[CFPR15]{CRPR}
R.~A. Capistrano-Filho, A.~F. Pazoto, and L.~Rosier.
\newblock Internal controllability of the {K}orteweg--de {V}ries equation on a bounded domain.
\newblock {\em ESAIM Control Optim. Calc. Var.}, 21(4):1076--1107, 2015.

\bibitem[CKN24]{CKN}
J.-M. Coron, A.~Koenig, and H.-M. Nguyen.
\newblock On the small-time local controllability of a {K}d{V} system for critical lengths.
\newblock {\em J. Eur. Math. Soc. (JEMS)}, 26(4):1193--1253, 2024.

\bibitem[CKS{\etalchar{+}}03]{CKSTT}
J.~Colliander, M.~Keel, G.~Staffilani, H.~Takaoka, and T.~Tao.
\newblock Sharp global well-posedness for {K}d{V} and modified {K}d{V} on {$\Bbb R$} and {$\Bbb T$}.
\newblock {\em J. Amer. Math. Soc.}, 16(3):705--749, 2003.

\bibitem[CL14]{Coron-Lv}
J.-M. Coron and Qi~L\"u.
\newblock Local rapid stabilization for a {K}orteweg-de {V}ries equation with a {N}eumann boundary control on the right.
\newblock {\em J. Math. Pures Appl. (9)}, 102(6):1080--1120, 2014.

\bibitem[Cor75]{Cordes}
H.~O. Cordes.
\newblock On compactness of commutators of multiplications and convolutions, and boundedness of pseudodifferential operators.
\newblock {\em J. Functional Analysis}, 18:115--131, 1975.

\bibitem[Cor06]{Coron06}
J.-M. Coron.
\newblock On the small-time local controllability of a quantum particle in a moving one-dimensional infinite square potential well.
\newblock {\em C. R. Math. Acad. Sci. Paris}, 342(2):103--108, 2006.

\bibitem[Cor07a]{Coron07}
J.-M. Coron.
\newblock {\em Control and nonlinearity}, volume 136 of {\em Mathematical Surveys and Monographs}.
\newblock American Mathematical Society, Providence, RI, 2007.

\bibitem[Cor07b]{Coron07-Survey}
J.-M. Coron.
\newblock Some open problems on the control of nonlinear partial differential equations.
\newblock In {\em Perspectives in nonlinear partial differential equations}, volume 446 of {\em Contemp. Math.}, pages 215--243. Amer. Math. Soc., Providence, RI, 2007.

\bibitem[Cox13]{Cox13}
D.~A. Cox.
\newblock {\em Primes of the form {$x^2 + ny^2$}}.
\newblock Pure and Applied Mathematics (Hoboken). John Wiley \& Sons, Inc., Hoboken, NJ, second edition, 2013.
\newblock Fermat, class field theory, and complex multiplication.

\bibitem[CRX17]{coron-rivas-xiang}
J.-M. Coron, I.~Rivas, and S.~Xiang.
\newblock Local exponential stabilization for a class of {K}orteweg--de {V}ries equations by means of time-varying feedback laws.
\newblock {\em Anal. PDE}, 10(5):1089--1122, 2017.

\bibitem[CV71]{C-V}
A.-P. Calder\'on and R.~Vaillancourt.
\newblock On the boundedness of pseudo-differential operators.
\newblock {\em J. Math. Soc. Japan}, 23:374--378, 1971.

\bibitem[GG08]{glass-guerrero}
O.~Glass and S.~Guerrero.
\newblock Some exact controllability results for the linear {K}d{V} equation and uniform controllability in the zero-dispersion limit.
\newblock {\em Asymptot. Anal.}, 60(1-2):61--100, 2008.

\bibitem[H\"07]{Hormander-3}
L.~H\"ormander.
\newblock {\em The analysis of linear partial differential operators. {III}}.
\newblock Classics in Mathematics. Springer, Berlin, 2007.
\newblock Pseudo-differential operators, Reprint of the 1994 edition.

\bibitem[KPV93]{KPV93}
C.~E. Kenig, G.~Ponce, and L.~Vega.
\newblock Well-posedness and scattering results for the generalized {K}orteweg-de {V}ries equation via the contraction principle.
\newblock {\em Comm. Pure Appl. Math.}, 46(4):527--620, 1993.

\bibitem[KPV96]{KPV96}
C.~E. Kenig, G.~Ponce, and L.~Vega.
\newblock A bilinear estimate with applications to the {K}d{V} equation.
\newblock {\em J. Amer. Math. Soc.}, 9(2):573--603, 1996.

\bibitem[KTV14]{KTV-book}
H.~Koch, D.~Tataru, and M.~Visan.
\newblock {\em Dispersive equations and nonlinear waves}, volume~45 of {\em Oberwolfach Seminars}.
\newblock Birkh\"auser/Springer, Basel, 2014.
\newblock Generalized Korteweg-de Vries, nonlinear Schr\"odinger, wave and Schr\"odinger maps.

\bibitem[KV19]{KV-19}
R.~Killip and M.~Visan.
\newblock Kd{V} is well-posed in {$H^{-1}$}.
\newblock {\em Ann. of Math. (2)}, 190(1):249--305, 2019.

\bibitem[KV22]{KV-22}
R.~Killip and M.~Visan.
\newblock Orbital stability of {K}d{V} multisolitons in {$H^{-1}$}.
\newblock {\em Comm. Math. Phys.}, 389(3):1445--1473, 2022.

\bibitem[KX21]{KX}
J.~Krieger and S.~Xiang.
\newblock Cost for a controlled linear {K}d{V} equation.
\newblock {\em ESAIM Control Optim. Calc. Var.}, 27:Paper No. S21, 41, 2021.

\bibitem[Ler10]{Lerner-book}
N.~Lerner.
\newblock {\em Metrics on the phase space and non-selfadjoint pseudo-differential operators}, volume~3 of {\em Pseudo-Differential Operators. Theory and Applications}.
\newblock Birkh\"auser Verlag, Basel, 2010.

\bibitem[LRZ10]{LRZ10}
C.~Laurent, L.~Rosier, and B.-Y. Zhang.
\newblock Control and stabilization of the {K}orteweg-de {V}ries equation on a periodic domain.
\newblock {\em Comm. Partial Differential Equations}, 35(4):707--744, 2010.

\bibitem[Mar18]{M18}
F.~Marbach.
\newblock An obstruction to small-time local null controllability for a viscous {B}urgers' equation.
\newblock {\em Ann. Sci. \'{E}c. Norm. Sup\'{e}r. (4)}, 51(5):1129--1177, 2018.

\bibitem[MM01]{MM-01}
Y.~Martel and F.~Merle.
\newblock Asymptotic stability of solitons for subcritical generalized {K}d{V} equations.
\newblock {\em Arch. Ration. Mech. Anal.}, 157(3):219--254, 2001.

\bibitem[MV03]{MV}
F.~Merle and L.~Vega.
\newblock {$L^2$} stability of solitons for {K}d{V} equation.
\newblock {\em Int. Math. Res. Not.}, (13):735--753, 2003.

\bibitem[Ngu21]{Nguyen}
H.-M. Nguyen.
\newblock Decay for the nonlinear {K}d{V} equations at critical lengths.
\newblock {\em J. Differential Equations}, 295:249--291, 2021.

\bibitem[Ngu23]{nguyen2023}
H.-M. Nguyen.
\newblock Local controllability of the korteweg-de vries equation with the right dirichlet control.
\newblock {\em Arxiv: 2302.06237}, 2023.

\bibitem[Ngu24]{nguyen2024stabilization}
H.-M. Nguyen.
\newblock Rapid and finite-time boundary stabilization of a kdv system.
\newblock {\em Arxiv: 2409.11768}, 2024.

\bibitem[NS18]{NgSq18}
H.-M. Nguyen and M.~Squassina.
\newblock Fractional {C}affarelli-{K}ohn-{N}irenberg inequalities.
\newblock {\em J. Funct. Anal.}, 274(9):2661--2672, 2018.

\bibitem[PMVZ02]{PVZ}
G.~Perla~Menzala, C.~F. Vasconcellos, and E.~Zuazua.
\newblock Stabilization of the {K}orteweg-de {V}ries equation with localized damping.
\newblock {\em Quart. Appl. Math.}, 60(1):111--129, 2002.

\bibitem[Ros97]{Rosier97}
L.~Rosier.
\newblock Exact boundary controllability for the {K}orteweg-de {V}ries equation on a bounded domain.
\newblock {\em ESAIM Control Optim. Calc. Var.}, 2:33--55, 1997.

\bibitem[Ros04]{Rosier04}
L.~Rosier.
\newblock Control of the surface of a fluid by a wavemaker.
\newblock {\em ESAIM Control Optim. Calc. Var.}, 10(3):346--380, 2004.

\bibitem[RZ96]{RZ96}
D.~L. Russell and B.-Y. Zhang.
\newblock Exact controllability and stabilizability of the {K}orteweg-de {V}ries equation.
\newblock {\em Trans. Amer. Math. Soc.}, 348(9):3643--3672, 1996.

\bibitem[RZ06]{Rosier-Zhang-06}
L.~Rosier and B.-Y. Zhang.
\newblock Global stabilization of the generalized {K}orteweg-de {V}ries equation posed on a finite domain.
\newblock {\em SIAM J. Control Optim.}, 45(3):927--956, 2006.

\bibitem[RZ09]{RZ09}
L.~Rosier and B.-Y. Zhang.
\newblock Control and stabilization of the {K}orteweg-de {V}ries equation: recent progresses.
\newblock {\em J. Syst. Sci. Complex.}, 22(4):647--682, 2009.

\bibitem[Sus87]{1987-Sussmann-SICON}
H.~J. Sussmann.
\newblock A general theorem on local controllability.
\newblock {\em SIAM J. Control Optim.}, 25(1):158--194, 1987.

\bibitem[Tao06]{Tao06}
T.~Tao.
\newblock {\em Nonlinear dispersive equations}, volume 106 of {\em CBMS Regional Conference Series in Mathematics}.
\newblock Published for the Conference Board of the Mathematical Sciences, Washington, DC; by the American Mathematical Society, Providence, RI, 2006.
\newblock Local and global analysis.

\bibitem[TCSC18]{Tang-Chu-Shang-Coron}
S.~Tang, J.~Chu, P.~Shang, and J.-M. Coron.
\newblock Asymptotic stability of a {K}orteweg--de {V}ries equation with a two-dimensional center manifold.
\newblock {\em Adv. Nonlinear Anal.}, 7(4):497--515, 2018.

\bibitem[Xia18]{Xiang-18}
S.~Xiang.
\newblock Small-time local stabilization for a {K}orteweg--de {V}ries equation.
\newblock {\em Systems Control Lett.}, 111:64--69, 2018.

\bibitem[Xia19]{Xiang-19}
S.~Xiang.
\newblock Null controllability of a linearized {K}orteweg--de {V}ries equation by backstepping approach.
\newblock {\em SIAM J. Control Optim.}, 57(2):1493--1515, 2019.

\end{thebibliography}
\end{document}